\documentclass[12pt]{amsart}
\usepackage{setspace}
\usepackage{verbatim}

\usepackage{amsfonts,fancyhdr,ifthen,bm}
\usepackage{amscd,amssymb} 
\usepackage[letterpaper,left=1.35in,right=1.35in,top=1.35in,bottom=1.35in]{geometry}
 
\usepackage{times}
\usepackage{graphicx}
\usepackage{color}
\usepackage{epsfig}
\usepackage{mathrsfs}
\usepackage{paralist}
\usepackage{comment}
\usepackage{mathtools}
\usepackage{multirow}
\usepackage{tikz-cd}
\usepackage[foot]{amsaddr}
\usepackage{caption}

\usepackage[style=numeric, backend=bibtex, maxbibnames=99]{biblatex}
\renewbibmacro{in:}{}

\DeclareBibliographyCategory{Smia}
\addtocategory{Smia}{ Smi22a}

\DeclareFieldFormat{prefixnumber}{\ifcategory{Smia}{\mkbibbrackets{Smi1}}{\mkbibbrackets{#1}}} 
\DeclareFieldFormat{labelnumber}{\ifcategory{Smia}{Smi1}{#1}}

\usepackage{enumitem}

\usepackage{appendix}

\usepackage{tikz,ifthen,calc}
\usetikzlibrary{automata,arrows,positioning,calc}

\newtheorem{thm}{Theorem}[section]

\newtheorem*{thm*}{Theorem}
\newtheorem{prop}[thm]{Proposition}
\newtheorem*{prop*}{Proposition}

\newtheorem*{cor*}{Corollary}

\newtheorem*{lem*}{Lemma}

\newtheorem*{oquest*}{Open Question}

\theoremstyle{remark}
\newtheorem{rmk}[thm]{Remark}

\theoremstyle{remark}
\newtheorem*{rmk*}{Remark}

\theoremstyle{definition}
\newtheorem{defn}[thm]{Definition}

\theoremstyle{definition}
\newtheorem{notat}[thm]{Notation}

\theoremstyle{definition}
\newtheorem{ass}[thm]{Assumption}

\theoremstyle{definition}

\theoremstyle{definition}
\newtheorem*{defn*}{Definition}

\theoremstyle{definition}
\newtheorem{ex}[thm]{Example}

\theoremstyle{definition}

\theoremstyle{definition}
\newtheorem{case}[thm]{Case}

\numberwithin{equation}{section}

 \usepackage{chngcntr}
\counterwithin{figure}{section}

\newcommand{\Z}{\mathbb{Z}}
\newcommand{\QQ}{\mathbb{Q}}

\newcommand{\R}{\mathbb{R}}

\DeclareMathOperator{\FFF}{\mathbb{F}}

\newcommand{\Sel}{\textup{Sel}}
\newcommand{\Gal}{\textup{Gal}}

\newcommand{\mfp}{\mathfrak{p}}
\newcommand{\mfq}{\mathfrak{q}}
\newcommand{\Vplac}{\mathscr{V}}
\newcommand{\Frob}{\textup{Frob}}

\newcommand{\Hom}{\textup{Hom}}

\newcommand{\msS}{\mathscr{S}}
\newcommand{\royarrow}{\,\xrightarrow{\,\quad}\,}

\newcommand{\mfh}{\mathfrak{h}}

\newcommand{\res}{\textup{res}}
\newcommand{\Cl}{\textup{Cl}\,}

\newcommand{\Ssm}{S_{\textup{sm}}}
\newcommand{\Smed}{S_{\textup{med}}}
\newcommand{\Slg}{S_{\textup{lg}}}
\newcommand{\Spar}{S_{\textup{par}}}

\newcommand{\ovp}{\overline{\mfp}}
\newcommand{\ovq}{\overline{\mfq}}
\newcommand{\inv}{\textup{inv}}

\newcommand{\XXfav}{\mathbb{X}^{\textup{fav}}}

\newcommand{\bfQ}{\mathbf{Q}}

\newcommand{\GOsim}{G_1/{\sim}}

\DeclareFontFamily{U}{wncy}{}
\DeclareFontShape{U}{wncy}{m}{n}{<->wncyr10}{}
\DeclareSymbolFont{mcy}{U}{wncy}{m}{n}
\DeclareMathSymbol{\Sha}{\mathord}{mcy}{"58}

\newcommand\restr[2]{{
  \left.\kern-\nulldelimiterspace 
  #1 
  \vphantom{\big|} 
  \right|_{#2} 
  }}

\DeclareMathOperator*{\rank}{rank}

\newcommand{\FrobF}[2]{\textup{Frob}_{#1}\,#2}

\usepackage{scalerel}

\newcommand{\ovQQ}{\overline{\mathbb{Q}}}

\newcommand{\mfB}{\mathfrak{B}}

\newcommand{\symb}[2]{\left[ #1,\, #2\right]}
\newcommand{\class}[1]{\left[#1\right]}
\newcommand{\spin}[1]{\symb{#1}{#1}}

\newcommand{\Zlz}{\mathbb{Z}_{\ell}[\xi]}
\newcommand{\msM}{\mathscr{M}}

\begin{document}
\title[Selmer groups in twist families II]{The distribution of $\ell^{\infty}$-Selmer groups in degree $\ell$ twist families II}

\author{Alexander Smith}
\email{asmith13@stanford.edu}
\date{\today}

\maketitle

\begin{abstract}
We continue the investigation of the distribution of $\ell^{\infty}$-Selmer groups in degree $\ell$ twist families of Galois modules over number fields begun in \cite{Smi22a}. Building off the work on higher Selmer groups in that part, we find conditions under which we can compute the distribution of the $\ell^{\infty}$-Selmer groups for a  given degree $\ell$ twist family. Along the way, we show that the average rank in the quadratic twist family of any given abelian variety over a number field is bounded.
\end{abstract}

\setcounter{tocdepth}{1}

\tableofcontents

\section{Introduction}
\label{sec:intro}
In the first part of this two-part paper \cite{Smi22a}, the author developed a method for controlling $2^{\infty}$-Selmer groups in the quadratic twist family of a given Galois module $N$ over a number field. The main result of that paper gave the distribution of $2^{\infty}$-Selmer groups in certain collections of twists known as grid classes subject to stringent conditions on the $2$-Selmer groups for the twists in the grid class.

The goal of this second part is to translate the results of \cite{Smi22a} into results on the distribution of $2^{\infty}$-Selmer groups in more natural families of twists. To this end, our main task is to find the distribution of $2$-Selmer groups in natural twist families. In the process of finding the distribution of $2$-Selmer groups, we will carve the natural families into grids of twists, a step that is also necessary to apply the results of \cite{Smi22a}.

Like \cite{Smi22a}, the results of this paper also apply to non-quadratic twist families. The methods that allow us to control the $2$-Selmer groups in quadratic twist families allow us to control the \emph{fixed point Selmer group} in more general contexts. To elaborate, choose some prime power $\ell^{k_0}$, and take $\xi$ to be the image of $x$ in the quotient ring
\begin{equation}
\label{eq:Zlz}
\Z_{\ell}[x]\Big/\left(1 + x^{\ell^{k_0 - 1}} + x^{2 \cdot \ell^{k_0 - 1}} + \dots + x^{(\ell - 1) \cdot \ell^{k_0 - 1}}\right).\end{equation}
If we are given a Galois module $N$ over the number field $F$ whose ring of continuous $G_F$-equivariant endomorphisms contains $\Zlz$, and if we are given $\chi$ in $\Hom_{\text{cont}}(G_F, \langle \xi \rangle)$, we may construct the twist $N^{\chi}$ of $N$ by $\chi$. Taking $\omega = \xi - 1$, there is an equivariant isomorphism
\[N^{\chi}[\omega]  \cong N[\omega]\]
between the submodules of $N$ and $N^{\chi}$ fixed by $\xi$. This isomorphism gives us control over the $\omega$-Selmer groups in the twist family of $N$. Since $N[\omega]$ is the collection of fixed points of $\xi$, we also refer to this Selmer group as the fixed point Selmer group. For a quadratic twist family, $\xi$ is $-1$, and the fixed point Selmer group is the $2$-Selmer group. A major goal of this paper is to find the distribution of fixed point Selmer groups in natural twist families of Galois modules.

Our interest in fixed point Selmer groups is driven by our interest in higher Selmer groups, and this is largely reflected in the main results of this paper. But our results on fixed point Selmer groups can be applied fruitfully without use of the higher theory. One such application is to prove the following boundedness result.
\begin{thm}
\label{thm:rank_bnd}
Take $\ell$ to be a rational prime, take $F$ to be a number field, and take $A$ to be an abelian variety over $F$.  For any $H > 0$,  take $X(H)$ to be the set of cyclic degree $\ell$ extensions $K$ of $F$ for which the absolute discriminant of $K$ is at most $H$.

Then there are real numbers $c, C > 0$ depending on $A/F$ and $\ell$ so, for any $m > 0$ and $H > C$ satisfying $m < c \cdot \log \log \log H$, we have
\[\sum_{K \in X(H)} \exp\left( m \cdot \rank(A/K)\right)\, \le\, \exp(Cm^2) \cdot \# X(H).\]
\end{thm}
In the quadratic case, similar boundedness results have been proved for most elliptic curves over $\QQ$  with full rational $2$-torsion \cite{Kane13} and for abelian varieties for which $A[3]$ has a filtration over $F$ with $1$-dimensional factors \cite{BKOS19}. Our unconditional result is a consequence of Theorem \ref{thm:main_mmnt_coarse}, as we will see in Example \ref{ex:rough_av_twists}.

Now, in some twist families, the average $\omega$-Selmer rank grows without bound as more and more twists are considered. If we wish to prove distributional results for the $\omega$-Selmer ranks in such families, we need to have a strategy for dealing with this complication. Our strategy is to reduce the scope of our results to \emph{favored twists}, which we will fully define in Section \ref{sec:results} but which we will define in a simple case now.

\begin{ex}[\cite{Klag16, KaKl17}]
\label{ex:Klag}
Given an abelian variety $A/\QQ$ and a nonzero integer $d$, we take $A^d$ to be the quadratic twist of $A$ corresponding to $\QQ(\sqrt{d})$. We will also take $r_{2^k}(A)$ to be the $2^k$-Selmer rank of $A$ as defined in \cite{Smi22a}.

Choose rational numbers $a, b$ such that $b$ and $a^2 - 4b$ are nonzero, and consider the ratioinal elliptic curves
\[A\colon y^2 = x(x^2 + ax + b)\quad\text{and}\quad A_0 \colon y^2 = x(x^2 - 2ax + (a^2 - 4b)).\]
Note that there is a rational isogeny of degree $2$ from $A$ to $A_0$. 

Take
\[K = \QQ(\sqrt{a^2 - 4b})\quad\text{and}\quad K_0 = \QQ(\sqrt{b}).\]
Given a squarefree integer $d$, we call $d$ \emph{favored} with respect to $A$ if, among the odd prime divisors $p$ of $d$ where $A$ has good reduction, the number of primes at which $K_0/\QQ$ splits is no larger than the number of primes at which $K/\QQ$ splits.
\end{ex}

The natural density of favored $d$ in the set of squarefree integers is $1/2$. On the half which is not favored, the $2$-Selmer ranks do not have a nice distribution.
\begin{prop}[\cite{Klag16}]
Take $A$, $K$, and $K_0$ as in Example \ref{ex:Klag}, and choose $\epsilon > 0$.  We assume $K$ and $K_0$ are distinct nontrivial extensions of $\QQ$. Then
\[\lim_{H \to \infty} \frac{\#\left \{d\in \Z^{\ne 0}\,:\,\, |d| \le H, \,\,d\,\text{ is not favored, and }\, \,r_2(A^d) > (\log \log H)^{1/2 - \epsilon}\right\} }{\#\left \{d\in \Z^{\ne 0}\,:\,\, |d| \le H, \,\,d\,\text{ is not favored }\right\}} = 1.\]
\end{prop}

However, the distribution of $2$-Selmer ranks among the favored twists is the same as the distribution of $2$-Selmer ranks in the family of twists of an elliptic curve with no rational $2$-torsion.
\begin{thm}[{Cf. \cite[Theorem 1.5]{Smi22a}}]
\label{thm:case2}
Take $A/\QQ$, $K$, and $K_0$ as in Example \ref{ex:Klag}. We assume $K$ and $K_0$ are distinct nontrivial extensions of $\QQ$.

Define $P^{\textup{Alt}}$ as in \cite[Definition 1.4]{Smi22a}. Then, given $k \ge 1$ and any nonincreasing sequence
\[r_2 \ge r_4 \ge r_8 \ge  \dots\]
of nonnegative integers, we have
\begin{align*}
&\lim_{H \rightarrow \infty} \frac{\#\{d\in \Z^{\ne 0}\,:\,\, |d| \le H, \,\,d \,\text{ is favored and }\,r_{2^k}(A^d) = r_{2^k} \text{ for all } k \ge 1\}}{H} \\
& \qquad =  P^{\textup{Alt}}(r_2\,|\,\infty) \cdot \prod_{k= 2}^{\infty} P^{\textup{Alt}}(r_{2^{k}}\, |\, r_{2^{k-1}}).
\end{align*}
\end{thm}

Even though Theorem \ref{thm:case2} only applies to favored $d$, it still can be used to find the distribution of $2^{\infty}$-Selmer coranks for the full twist family of  $A$. Specifically, we find that any $d$ that is not favored with respect to $A$ is favored with respect to the isogenous curve $A_0$. We also find that a negligible set of $d$ is favored for both $A$ and $A_0$. Since $r_{2^{\infty}}(A_0^d) = r_{2^{\infty}} (A^d)$, we can find the distribution of $2^{\infty}$-Selmer coranks across the full twist family of $A$ by applying Theorem \ref{thm:case2} to both $A$ and $A_0$. This proves the second case of \cite[Theorem 1.2]{Smi22a}. 

The notion of a favored twist generalizes substantially, and the natural extension of the divide-and-conquer approach seen above allows us to prove Theorem \ref{thm:rank_bnd}.

Once we restrict to favored twists, the distribution of fixed point Selmer ranks is still affected by the structure of the Galois module $N$. Our main result, Theorem \ref{thm:main}, puts a number of technical conditions on $N$ and its submodules. A module that does not satisfy these conditions will have a distinct distribution of $\omega$-Selmer groups in its favored twist family compared to any module that satisfies these conditions. We hope that future work will give us more information about the wide variety of distributions that can appear if these technical conditions do not hold.

\subsection{Our method}
We introduced the notion of a grid in \cite{Smi22a}. To briefly recap, a grid is defined from a product space $\prod_{s \in S} X_s$ together with some auxiliary data, where the $X_s$ are finite sets of primes of $\overline{\QQ}$ obeying some mild conditions. From the auxiliary data, we may associate any point $(\ovp_s)_{s \in S}$ in this product space with a twist $\chi$. As we proved in \cite[Proposition 4.8]{Smi22a}, the group $\Sel^{\omega} N^{\chi}$ can be calculated from the classes $\class{\ovp_s}$ for all $s \in S$, the spins $\spin{\ovp_s}$ for all $s \in S$, and the symbols $\symb{\ovp_s}{\ovp_t}$ for all distinct $s, t \in S$. We reprove this result in an explicit form in Section \ref{ssec:explicit_local}.

Using this, we may express a given moment of the size of $\Sel^{\omega} N^{\chi}$ for $\chi$ varying in a grid of twists as a sum of terms defined from classes, spins, and symbols. When a given term in this sum involves a symbol $\symb{\ovp_s}{\ovp_t}$, we can often prove that the term has negligble contribution to the overall sum either by using the Chebotarev density theorem or by using the bilinear equidistribution result for symbols proved as \cite[Theorem 5.2]{Smi22a}. Once we remove the terms proved to be negligible using this method, the remaining terms can be controlled using linear algebra.

Our need to invoke the Chebotarev density theorem and the bilinear equidistribution result pressures us to consider grids where the collections of primes $X_s$ are large relative to their height. But our need to carve natural families of twists into such grids pressures us in the opposite direction. Fortunately, there is a middle ground that works. We give one possible specification for such grids in Section \ref{sec:gridding}.

\subsection{An overview of this paper}

Section \ref{sec:results} gives the necessary notation to state Theorem \ref{thm:main}, the main result for the distribution of $\ell^{\infty}$-Selmer groups in twist families of this two-part paper. This result is proved by combining the results on $\omega$-Selmer groups of this paper with the results for higher Selmer groups in \cite{Smi22a}.

Theorem \ref{thm:main} is largely interesting for its applications to class groups and Selmer groups of abelian varieties. We go through some of these applications in Section \ref{sec:examples}. This is where the theorems mentioned in the introduction to \cite{Smi22a} are all proved.

Section \ref{sec:tamagawa} explores why the average $\omega$-Selmer rank in families of unfavored twists tends to infinity. This involves certain terms known as Tamagawa ratios, which also appear in the definition of cofavored submodules. The moments of sizes of Selmer groups in grids can be expressed as sums over the cofavored submodules of the powers of $N$, so we need to have a classification of these submodules. Using the technical conditions appearing in Theorem \ref{thm:main}, we are able to find such a classification.

Sections \ref{sec:character}, \ref{sec:ignorable}, and \ref{sec:moments} all concern the moments of the sizes of the $\omega$-Selmer group in grids of twists. Section \ref{sec:character} gives the initial reduction of such a moment to a sum of terms involving spins and symbols. In Section \ref{sec:ignorable}, the terms that cannot be proved to be negligible using analytic number theory are classified and roughly estimated. Section \ref{sec:moments} then gives a more precise estimate for these moments under the conditions of Theorem \ref{thm:main}.

The highly analytic Section \ref{sec:gridding} carves natural families of twists into grids where the estimates of the previous three sections apply. In Section \ref{sec:rough}, we use this work to give rough upper bounds for the moments of sizes of $\omega$-Selmer groups in natural twist families. These estimates are refined in Section \ref{sec:distribution}, where the work in \cite{Smi22a} is incorporated to prove the main result of this project, Theorem \ref{thm:main}.

\subsection*{Acknowledgements}
This project benefited from the input and support of many mathematicians, including those I have enumerated in the acknowledgements section of \cite{Smi22a}. I am grateful to all of them.

This research was partially conducted during the period the author served as a Clay
Research Fellow. Previously, the author was supported in part by National Science Foundation grant DMS-2002011.

\section{Main results}
\label{sec:results}
The goal of this section is to give the precise form of our main theorems.  We start by recalling the definitions of twistable modules and their Selmer groups from \cite{Smi22a}. We retain our notational conventions from that paper.

\begin{notat}
\label{notat:twistable}
Fix a power $\ell^{k_0}$ of a rational prime, take $\Zlz$ to be the corresponding quotient ring \eqref{eq:Zlz}, and take $\omega = \xi - 1$ and $\FFF = \langle \xi \rangle$. Also choose a number field $F$.

A \emph{twistable module} $N$ is then a discrete nonzero $G_F$-module isomorphic as a group to some finite power of $\QQ_{\ell}/\Z_{\ell}$ together with an action of $\Zlz$ on $N$ that commutes with $G_F$. Given such a module and a continuous homomorphism $\chi \in \Hom_{\text{cont}}(G_F, \FFF)$, we may define a twist $N^{\chi}$ of $N$ as in \cite[Definition 4.1]{Smi22a} .

 Fix some twistable module $N$, and take $\Vplac_0$ to be a finite set of places of $F$ containing the archimedean places and the places dividing $\ell$. We assume that the action of $G_F$ on $N$ is ramified only at primes in $\Vplac_0$.

For $v \in \Vplac_0$ and $\chi \in \Hom_{\text{cont}}(G_v, \FFF)$, fix an $\ell$-divisible subgroup $W_v(\chi)$  of $H^1(G_v, N^{\chi})$. We call the collection of groups $\left(W_v(\chi)\right)_{v \in \Vplac_0}$ a \emph{set of local conditions} for $N$ at the places in $\Vplac_0$.

Setting the local conditions $W_v$ at the remaining places as in \cite[Definition 4.3]{Smi22a}, the Selmer group of $N^{\chi}$ is defined by
\[\Sel\left(N^{\chi}, (W_v)_{v \in \Vplac_0}\right) = \ker\left(H^1(G_F, N^{\chi}) \to \prod_{v \text{ of } F} H^1(G_v, N^{\chi})\Big/W_v\left(\res_{G_v\,} \chi\right)\right).\]
We will write this group as $\Sel\,N^{\chi}$ if the local conditions at the places in $\Vplac_0$ are clear.

For $k \ge 1$, we take $\Sel^{\omega^k} N^{\chi}$ to be the portion of $H^1\left(G_F, N^{\chi}\left[\omega^k\right]\right)$ mapping into $\Sel \,N^{\chi}$. We can refer to $\Sel^{\omega} N^{\chi}$ as either an $\omega$-Selmer group or as a fixed point Selmer group. For each $k \ge 1$, we define
\[r_{\omega^k}(N^{\chi}) = \dim \omega^{k-1}\Sel^{\omega^k}N^{\chi}\Big/\text{im} \,H^0(G_F, N^{\chi}[\omega]),\]
with the image being taken under the connecting map corresponding to the exact sequence $0 \to N^{\chi}[\omega] \to N^{\chi}[\omega^2] \to N^{\chi}[\omega] \to 0$.

We recall from \cite[Definition 4.9]{Smi22a} that we have a notion of the dual twistable module $N^{\vee}$ to $N$.
\end{notat}

\begin{defn}
Given $F$, $\Zlz$, and $\Vplac_0$ as above, we define the set of \emph{admissible twists} to be the set of continuous homomorphisms $\chi: G_F \to \FFF$ so that, for every prime $\mfp$ of $F$ outside $\Vplac_0$, the homomorphism $\chi$ is ramified at $\mfp$ only if its image in $\Hom(G_F, \FFF/\ell\FFF)$ is ramified at $\mfp$. We denote this set by $\mathbb{X}_F$. If $\xi$ has order $\ell$, every continuous homomorphism $\chi: G_F \to \FFF$ is an admissible twist.

We represent the ramification of an element $\chi$ of $\mathbb{X}_F$ using the squarefree ideal
\[\mfh_F(\chi) = \prod_{\substack{\mfp \not\in \Vplac_0 \\ \chi \text{ is ramified at } \mfp}} \mfp.\]
There are many possible ways to order these ideals. Writing $N_{F/\QQ}$ for the norm from $F$ to $\QQ$, we settle for the height function
\[h_F(\chi) = N_{F/\QQ}(\mfh_F(\chi)).\]
For $H > 0$, we take $\mathbb{X}_F(H)$ to be the set of admissible twists of height at most $H$.
\end{defn}

\begin{rmk}
We restrict our attention to admissible twists because, outside this case, we do not have a result like \cite[Proposition 4.8]{Smi22a} to determine the $\omega$-Selmer group of a twist from classes, spins, and symbols. 
\end{rmk}

\begin{defn}
\label{defn:tamagawa}
Take $T$ to be a $G_F$-submodule of $N[\omega]$. Given $\chi \in \mathbb{X}_F$, we define the Tamagawa ratio $\mathcal{T}_{N, T}(\chi)$ by
\[\mathcal{T}_{N, T}(\chi) = \prod_{\mfp | \mfh_F(\chi)} \frac{\#H^0(G_{\mfp}, \, N/T[\omega])}{\#H^0(G_{\mfp}, \, N[\omega])}.\]
\end{defn}

Tamagawa ratios give lower bounds for the size of the $\omega$-Selmer group of $N^{\chi}$. Specifically, we have the following.
\begin{prop}
\label{prop:tamagawa}
Suppose $N$ is a twistable module with local conditions defined with respect to $F$, $\Vplac_0$, and $\Zlz$. Then there is a $c > 0$ so that, given any $G_F$-submodule $T$ of $N[\omega]$, and given $\chi \in \mathbb{X}_F$, we have
\[\# \Sel^{\omega} N^{\chi} \ge c \cdot \mathcal{T}_{N, T}(\chi).\]
\end{prop}
We will prove this result in Section \ref{sec:tamagawa}. 

The next theorem shows that the bound of Proposition \ref{prop:tamagawa} tends to be fairly sharp.
\begin{thm}
\label{thm:main_mmnt_coarse}
Suppose $N$ is a twistable module with local conditions defined with respect to $F$, $\Vplac_0$, and $\Zlz$. Then there is $C > 0$ depending on $F$, $N$, and $\Vplac_0$ so that, given any $H > 30$, and given any positive integer $g$ satisfying
\[(\#N[\omega])^{200g} \le \log \log H,\]
we have
\[\sum_{\substack{\chi \in \mathbb{X}_F(H)}} \left(\frac{\# \Sel^{\omega}(N^{\chi})}{\max_T \mathcal{T}_{N, T}(\chi)}\right)^g \le \exp\left(Cg^2\right)\cdot \# \mathbb{X}_F(H).\]
\end{thm}
We will prove this theorem in Section \ref{sec:rough}.

\begin{defn}
\label{defn:pfav}
With $N$ fixed as above, we call a twist $\chi \in \mathbb{X}_F$ a \emph{favored twist} of $N$ if, for any $G_F$-submodule $T$ of $N[\omega]$, we have $\mathcal{T}_{N, T}(\chi) \le 1$. We denote the set of favored twists for $N$ by $\XXfav_{F, N}$, and we write those of height at most $H$ by $\XXfav_{F, N}(H)$. We call $N$ \emph{potentially favored} if
\[\lim_{H \to \infty } \frac{\#\XXfav_{F, N}(H)}{\#\mathbb{X}_F(H)} > 0.\]
This condition may be checked using the algebra of $N$; see Definition \ref{defn:pfav2}.

Take $F_1  = F(\mu_{\ell^{k_0}})$, where $\ell^{k_0}$ is the order of $\xi$. We call a $G_F$-submodule $T$ of $N[\omega]$ \emph{cofavored} if, for every $\sigma$ in $G_{F_1}$, we have
\[\dim H^0\big(\langle \sigma \rangle,\, (N/T)[\omega]\big) = \dim H^0\big(\langle \sigma \rangle,\, N[\omega]\big).\]
If the only cofavored submodules of $N[\omega]$ are $0$ and $N[\omega]$, we call $N$ \emph{uncofavored}.
\end{defn}

\begin{rmk}
\label{rmk:finite_favored}
By the Hermite--Minkowski theorem, there are only finitely many isomorphism classes of $G_F$-modules of order $\#N[\omega]$ so that the action of $G_F$ is unramified outside $\Vplac_0$. In particular,  there is a finite collection $T_1, \dots, T_k$ of $G_F$-submodules of $N$ so that, for any other $G_F$-submodule $T$ of $N$, there is $i \le k$ so $(N/T)[\omega]$ is isomorphic to $(N/T_i)[\omega]$. For any admissible twist $\chi$, there is some $i \le k$ so $\chi$ is a favored twist for  $N/T_i$.
\end{rmk}

\begin{defn}
\label{defn:dual_delta}
Given an equivariant endomorphism $\phi$ of $N[\omega]$, we say that $\phi$ \emph{commutes with the connecting maps} if, for all $\sigma \in G_{F_1}$, the connecting map 
\begin{equation}
\label{eq:connecting_map}
\delta_{\sigma}:  H^0\big(\langle \sigma \rangle,\, N[\omega]\big) \to  H^1\big(\langle \sigma \rangle,\, N[\omega]\big)
\end{equation}
corresponding to the exact sequence 
\begin{equation}
\label{eq:conn_map_basic}
0 \to N[\omega] \to N[\omega^2] \xrightarrow{\,\,\omega\,\,} N[\omega] \to 0
\end{equation}
satisfies
\[\phi \circ \delta_{\sigma} = \delta_{\sigma} \circ \phi.\]
Multiplication by any given element in $\FFF_{\ell}^{\times}$ defines an automorphism of $N[\omega]$ that commutes with the connecting map. If $N[\omega]$ has extra automorphisms beyond these, it will have a noticeable impact on the distribution of $\omega$-Selmer groups in its twist family.

Similarly, we can consider the connecting map 
\[\delta_{\sigma}^{\vee} : H^0\big(\langle \sigma \rangle,\, N^{\vee}[\omega]\big) \to  H^1\big(\langle \sigma \rangle,\, N^{\vee}[\omega]\big)\]
corresponding to the exact sequence 
\[0 \to N^{\vee}[\omega] \to N^{\vee}[\omega^2] \xrightarrow{\,\,\omega\,\,} N^{\vee}[\omega] \to 0.\] 
If there is an equivariant isomorphism $\phi: N[\omega] \to N^{\vee}[\omega]$ satisfying
\[\phi \circ \delta_{\sigma} = \delta_{\sigma}^{\vee} \circ \phi.\]
for all $\sigma \in G_{F_1}$, we say that $N[\omega]$ has self-dual structure. Otherwise, we say it lacks self-dual structure.
\end{defn}

\begin{defn}
With $\FFF$ and $\Vplac_0$ fixed as above, a choice of \emph{local twists} is a tuple $(\chi_v)_{v \in \Vplac_0}$ of homomorphisms $\chi_v \in \Hom_{\text{cont}}(G_v, \FFF)$ indexed by $v \in \Vplac_0$ that is globally represented in the sense that some homomorphism $\chi \in \Hom_{\text{cont}}(G_F, \FFF)$ restricts to $\chi_v$ on $G_v$ for each $v$ in $\Vplac_0$. By the Grunwald--Wang theorem \cite{Wang50}, the hypothesis that $(\chi_v)_{v \in \Vplac_0}$ is globally represented is automatically satisfied unless $\ell = 2$ and $\xi$ has order at least $8$.

Given a choice of local twists $(\chi_v)_{v \in \Vplac_0}$, we define $\mathbb{X}_F((\chi_v)_{v \in \Vplac_0})$ to be the set of $\chi$ in $\mathbb{X}_F$ whose restriction to $G_v$ equals $\chi_v$ for each $v$ in $\Vplac_0$. We similarly define $\mathbb{X}_F(H, (\chi_v)_{v \in \Vplac_0})$ and $\XXfav_{F, N}(H, (\chi_v)_{v \in \Vplac_0})$.

Given an element $\phi \in H^1(G_F, N[\omega])$, we call $\phi$ a \emph{universal $\omega$-Selmer element} if, for every $\chi$ in $\mathbb{X}_F((\chi_v)_{v \in \Vplac_0})$, the image of $\phi$ in $H^1(G_F, N^{\chi}[\omega])$ is in the $\omega$-Selmer group of $N^{\chi}$. More algebraically, given $(\chi_v)_v$, the set of universal Selmer elements is the kernel of the map
\[H^1\big(G_F, \,N[\omega]\big) \to \prod_{v \in \Vplac_0} H^1\big(G_v, \,N[\omega]\big)/W_{v, 1}(\chi_v) \times \prod_{\sigma \in G_{F_1}}  H^1\big(\langle \sigma \rangle,\, N[\omega]\big), \]
where $W_{v, 1}(\chi_v)$ is defined as the portion of $H^1(G_v, N[\omega])$ that maps into $W_v(\chi_v)$. 
\end{defn}
\begin{rmk}
Given any twist $\chi$ in $\mathbb{X}_F(H, (\chi_v)_{v \in \Vplac_0})$ ramified at some place outside $\Vplac_0$, we have
\begin{equation}
\label{eq:Wiles}
r_{\omega}(N^{\chi})  - r_{\omega}\left((N^{\vee})^{\chi}\right) = \sum_{v \in \Vplac_0} \left( \dim W_{v, 1}(\chi_v) - \dim H^0(G_v, N[\omega])\right).
\end{equation}
This formula follows Poitou-Tate duality and the global Euler-Poincar\'{e} characteristic formula of Tate. This particular statement is a special case of \cite[Theorem 8.7.9]{Neuk08}, which is itself a generalization of a formula of Wiles \cite{Wiles95}. We note that, in the enumeration of places, each pair of complex embeddings of $F$ corresponds to one place in $\Vplac_0$.
\end{rmk}

Our main results give the distribution of $\ell^{\infty}$-Selmer groups for twists of the twistable module $N$ in families $\XXfav_{F, N}(H, (\chi_v)_{v \in \Vplac_0})$ as $H$ tends to infinity. The distribution of these groups depends on the structure of $N$. We work out what these distributions are in two contrasting cases. The first case will be relevant to the study of class groups, and the latter case will be relevant to the study of Selmer groups of abelian varieties.

\begin{case}
\label{case:nodual}
Take $N$ to be a twistable module with local conditions defined with respect to $F$, $\Vplac_0$, and $\Zlz$. Choose a set of local twists $(\chi_v)_{v \in \Vplac_0}$. We say that $N$ is in the  \emph{non-self-dual case} if
\begin{itemize}
\item $N$ is potentially favored, is uncofavored, and lacks self-dual structure;
\item The only automorphisms of $N[\omega]$ that commute with connecting maps are those in $\FFF_{\ell}^{\times}$; and
\item Neither $N$ nor $N^{\vee}$ has any nonzero universal $\omega$-Selmer element with respect to the local twists $(\chi_v)_{v \in \Vplac_0}$.
\end{itemize}

Given $N$ in this case, take $u$ to be the sum on the right hand side of \eqref{eq:Wiles}. Given $n \ge j \ge \max(u, 0)$, we define $P(j \,|\, n)$ to be the probability that an $(n - u) \times n$ matrix with coefficients selected uniformly at random from $\FFF_{\ell}$ has kernel of dimension $j$.  Explicitly, we have
\[P(j\,|\,n) = \ell^{-j(j- u)} \cdot \prod_{k = 1}^{j-u} \frac{1- \ell^{-k - n + j}}{1 - \ell^{-k}}\cdot \prod_{k=1}^{j} ( 1- \ell^{-k})^{-1} \cdot \prod_{k=1}^n (1 - \ell^{-k}).\]

If we instead have $n \ge j \ge 0$ with $j < u$, we define $P(j \,|\, n)$ to be zero.

We then take
\[P(j \,|\, \infty) = \lim_{n \to \infty} P(j\,|\, n),\]
which, for $j \ge \max(u, 0)$, equals
\[\ell^{-j(j- u)} \cdot \prod_{k = 1}^{j-u} \left(1 - \ell^{-k}\right)^{-1} \cdot \prod_{k=1}^{j} ( 1- \ell^{-k})^{-1} \cdot \prod_{k=1}^\infty (1 - \ell^{-k}).\]
\end{case}

\begin{case}
\label{case:alt}
Take $N$ to be a twistable module with local conditions defined with respect to $F$, $\Vplac_0$, and $\Zlz$. Choose a set of local twists $(\chi_v)_{v \in \Vplac_0}$. We say that $N$ is in the  \emph{alternating case} if
\begin{itemize}
\item $\FFF$ has order $2$, so $\xi = -1$;
\item $N$ is potentially favored and is uncofavored;
\item $N$ has alternating structure, in the sense of \cite[Definition 4.10]{Smi22a};
\item The only automorphism of $N[\omega]$ that commutes with the connecting maps is  the identity; and
\item $N$ has no nonzero universal $\omega$-Selmer element with respect to the local twists $(\chi_v)_{v \in \Vplac_0}$.
\end{itemize}

Given $N$ in this case, and given integers $n \ge j \ge 0$, we define $P(j \,|\, n)$ to be the probability that an alternating $n \times n$ matrix with coefficients selected uniformly at random from $\FFF_{2}$ has kernel of dimension $j$. If $j$ and $n$ have different parities, this equals $0$. Otherwise,
\[P(j\,|\,n) =  2^{\frac{-j(j-1)}{2}}\cdot\prod_{k=1}^{j}\frac{1 - 2^{-n-j +k}}{1 - 2^{-k}} \cdot \prod_{k = 1}^{ \frac{n - j}{2}}\left(1 - 2^{-2k+ 1}\right) .\]

The dependence on the parity of $n$ lends some subtlety to our definition of $P(j\,|\, \infty)$. Define a map $\varepsilon : G_{F} \to \FFF_2$ by
\begin{equation}
\label{eq:Morgan_map}
\varepsilon(\sigma) = \dim H^0\left(\langle \sigma\rangle,\, N[\omega]\right)\,\,\text{ mod } 2.
\end{equation}
If this map is a homomorphism, we say that $N$ is in the \emph{parity-invariant subcase}. Under this assumption, the parity of $r_{\omega}(N^{\chi})$ does not depend on the choice of $\chi$ in $\mathbb{X}_F\left((\chi_v)_{v \in \Vplac_0}\right)$; see Section \ref{ssec:parity}. In this subcase, take $b \in \{0, 1\}$ to match the parity of $r_{\omega}(N^{\chi})$ and define
\[P(j\,|\, \infty) = \lim_{n \to \infty} P(j\,|\, 2n + b).\]
Otherwise, we say $N$ is in the \emph{non-parity-invariant subcase}, and we take
\begin{align*}
P(j\,|\, \infty) \,=\, \lim_{n \to \infty} \tfrac{1}{2} P(j\,|\, 2n + j) \,=\, \tfrac{1}{2} \cdot 2^{\frac{-j(j-1)}{2}} \cdot \prod_{k =1}^j \left(1 - 2^{-k}\right)^{-1} \cdot \prod_{k = 0}^{\infty}\left(1 -2^{-2k- 1}\right).
\end{align*}
\end{case}

We can now state the main result of this project.
\begin{thm}
\label{thm:main}
Choose a twistable module $N$ and local twists $(\chi_v)_v$ satisfying the conditions either of the non-self-dual case or alternating case, and take $P(j\,|\, n)$ to be the transition probabilities corresponding to the case of $N$. There are then real numbers $c, C > 0$ determined by $N$, $\Vplac_0$, and $F$ so that, for any $H > C$, we have the following:

Given a nonincreasing sequence of nonnegative integers $r_{\omega} \ge r_{\omega^2} \ge \dots $, define an empirical probability
\[P_{\le H}(r_{\omega}, r_{\omega}^2, \dots) \,=\,  \frac{\#\left\{\chi \in \XXfav_{F, N}(H, \,(\chi_v)_v)\,:\,\, r_{\omega^k}(N^{\chi}) = r_{\omega^k}\, \text{ for all } \,k \ge 1\right\}}{ \#\XXfav_{F, N}(H,\, (\chi_v)_v)}\]
and a theoretical probability 
\[P_{\textup{th}}(r_{\omega}, r_{\omega}^2, \dots)  = P(r_{\omega}\,|\, \infty) \cdot \prod_{k \ge 1} P(r_{\omega^{k+1}}\,|\, r_{\omega^k}).\]
Then
\[\sum_{r_{\omega} \ge r_{\omega^2} \ge \dots } \Big|P_{\le H}(r_{\omega}, r_{\omega}^2, \dots) \, -\, P_{\textup{th}}(r_{\omega}, r_{\omega^2}, \dots)\Big| \,\le\, \exp\left( -c \cdot (\log \log \log H)^{1/2}\right),\]
where the sum is over all nonincreasing sequence of nonnegative  integers.
\end{thm}

\section{Examples and non-examples}
\label{sec:examples}
The two principal applications of Theorem \ref{thm:main} are to finding the distribution of $\Cl K[\ell^{\infty}]$ as $K$ varies in a family of degree $\ell$ extensions of a fixed number field $F$ and to finding the distribution of $2^{\infty}$-Selmer groups in the quadratic twist families of most elliptic curves over $F$. Before we prove Theorem \ref{thm:main}, we show how it may be applied to these two examples and some other problems. We will also give some relevant examples where the theorem cannot be applied.

\subsection{Quadratic twists of elliptic curves}
Take $F$ to be a number field, and take $A$ to be an elliptic curve over $F$ with Weierstrass form $y^2 = x^3 + ax + b$. We are interested in the $2^{\infty}$-Selmer groups for the family of elliptic curves
\[\{A^d\colon y^2 = x^3 + d^2ax + d^3b\,:\,\, d \in F^{\times}\}.\]
As in \cite[Example 4.4]{Smi22a}, these are the Selmer groups associated to twists of the twistable module $N = A[2^{\infty}]$ with the Bloch--Kato local conditions. This example has alternating structure.

To apply Theorem \ref{thm:main}, we must check that the other conditions of Case \ref{case:alt} are satisfied. The condition on universal $\omega$-Selmer elements is always satisfied, leaving three conditions; $N$ must be potentially favored, it must be cofavored, and $N[\omega]$ may not have extra automorphisms commuting with the connecting maps.

So take $K = F(A[2])$ to be the minimal extension of $F$ such that $G_{F(A[2])}$ acts trivially on $A[2]$. We may then find $c_1, c_2, c_3$ in $K$ for which
\[x^3 + ax + b = (x - c_1)(x - c_2)(x - c_3).\]
Take $e_1$ to be the point $(c_1, 0)$, and take $e_2$ to be the point $(c_2, 0)$. These generate the $2$-torsion of $E$. The connecting map
\[A[2] \rightarrow H^1(G_{K}, A[2])\]
corresponding to the short exact sequence $0 \to A[2] \to A[4]  \to A[2] \to 0$ may be given explicitly in the form
\begin{align*}
&e_1 \mapsto e_1 \cdot \chi_{c_1 - c_2} + e_2 \cdot \chi_{(c_3 - c_1)\cdot (c_2 - c_1)}\\
& e_2 \mapsto e_1 \cdot \chi_{(c_3 - c_2)\cdot (c_1 - c_2)} + e_2 \cdot \chi_{c_2 - c_1},
\end{align*}
where $\chi_c: G_{K} \rightarrow \FFF_2$ denotes the quadratic character associated to $K(\sqrt{c})/K$; this calculation can be done directly, and is the content of \cite[Proposition X.1.4]{Silv09}.

We now split into cases. 

\begin{ex}
Suppose $K = F$. In this case, $N$ is automatically potentially favored. We find that $N$ is uncofavored if and only if none of the isogenous curves
\[A/\langle e_1 \rangle,\quad A/\langle e_2 \rangle,\quad A/\langle e_1 + e_2 \rangle\]
has full rational two torsion. This is equivalent to the condition that $E$ has no rational cyclic $4$-isogeny. Given the above form of the connecting map, this condition is equivalent to saying that none of 
\begin{equation}
\label{eq:three_negsquares}
(c_3 - c_1) \cdot (c_2 - c_1),\quad (c_3 - c_2) \cdot (c_1- c_ 2),\quad (c_1 - c_3) \cdot (c_2 - c_3)
\end{equation}
are in $(F^{\times})^2$. If $E$ is not uncofavored, there is then  $c \in F \backslash\{0, \pm 1\}$ and $d$ in $F^{\times}$ such that $E$ is isomorphic to a curve with equation
\[y^2 = x(x-d)(x - dc^2).\]

We now check for a nonidentity automorphism commuting with the connecting maps. From Remark \ref{rmk:Wedd}, we can restrict our attention to the isomorphism $T$ sending $e_1$ to $e_2$ to $e_1 + e_2$ back to $e_1$.  We calculate that this commutes with the connecting maps if and only if all three entries in \eqref{eq:three_negsquares} are all in $- (F^{\times})^2$. In the case that $c_1 = 0$ and $c_2 = 1$, this gives that $-c_3$ and $c_3 - 1$ are both squares in $F$. In other words, there are $z_0, z_1$ in $F$ satisfying 
\begin{equation}
\label{eq:darkside_conic}
z_0^2 + z_1^2 = -1
\end{equation}
so $c_3 = -z_0^2$. If this conic has no points over $F$, as is true if $F$ is any totally real field, then the condition of no extra automorphisms is automatically satisfied. Otherwise, any exception will be isomorphic to a curve of the form
\begin{equation}
\label{eq:T_bad_map}
y^2 = x(x-d)(x + dz_0^2)
\end{equation}
where $d$ is in $F^{\times}$ and $z_0 \in F\backslash\{0, \pm 1\}$ satisfies \eqref{eq:darkside_conic} for some $z_1$ in $F$.

Outside these two exceptional families, we find that the technical conditions on Theorem \ref{thm:main} are satisfied. We therefore find that case (3) of \cite[Theorem 1.5]{Smi22a} holds.
\end{ex}

\begin{ex}
Suppose $K/F$ is a $\Z/3\Z$ extension. The only proper $G_F$-submodule of $A[2]$ is $0$, so $N$ is potentially favored and  uncofavored, leaving just the condition on automorphisms. The automorphism $T$ defined above is $G_F$-equivariant, so the condition fails if and only if there are $d, z_0, z_1$ in $K$ so that \eqref{eq:darkside_conic} is satisfied and $E$ is isomorphic over $K$ to the curve with Weierstrass form \eqref{eq:T_bad_map}. This is again impossible if $F$ is totally real, as $K$ is totally real in this case.

However, there are curves that fail the condition. To give a particular example, take $F$ to be a number field containing $\QQ(\mu_3)$, and take $b$ to be any element of $F^{\times}$ that is not a cube. Then, if we consider the elliptic curve
\[A: y^2 = x^3 + b,\]
we find that the extension $K/F$ is of degree $3$, and the map $T$ commutes with the connecting maps. We do not need to check for an isomorphism with a curve of the form \eqref{eq:T_bad_map} to prove this, as this curve has CM by $\Z[\mu_3]$. Taking $\xi$ to be a primitive third root of unity, we see that the automorphism corresponding to $\xi$ either equals $T$ or $T^2$.  To check that this curve is an exception using our condition, we note that $A$ is isomorphic over $K$ to the curve with equation
\[y^2 = x\left( x- (\xi- 1)\sqrt[3]{b}\right) \left(x - (\xi^2 - 1)\sqrt[3]{b}\right),\]
which corresponds to the solution $z_0 = \xi$, $z_1 = \xi^2$ of \eqref{eq:darkside_conic}.

\end{ex}

\begin{ex}
If $K/F$ is an $S_3$ extension, then every condition is clearly satisfied, and we may always apply Theorem \ref{thm:main}.
\end{ex}

From these last two calculations, we see that Theorem \ref{thm:main} yields \cite[Theorem 1.5]{Smi22a} in case (1).

\begin{ex}
If $K/F$ is quadratic, the automorphism $T$ is non-equivariant, so we only need to check the potentially favored and uncofavored conditions. We will replace our notation with the notation of Example \ref{ex:Klag}, so we are considering a pair of isogenous elliptic curves
\[A\colon y^2 = x(x^2 + ax + b)\quad\text{and}\quad A_0 \colon y^2 =x (x^2 - 2ax + (a^2 - 4b))\]
with $a,b$ in $F$ for which $b$ is nonzero and $a^2 - 4b$ is outside $F^2$.

We find that $A$ is potentially favored if $G_F$ acts nontrivially on $A_0[2]$, which happens if $b$ is outside $(F^{\times})^2$. If $A$ is not potentially favored, we may instead consider the curve $A_0$, which fits into the $K = F$ case. Next, we see that $A$ is uncofavored if and only if $F(A[2]) \ne F(A_0[2])$, which is equivalent to the condition that
 $ba^2 - 4b^2$ is not in $(F^{\times})^2$.

If neither $b$ nor $ba^2 -4b^2$ are squares, we find that the main theorem applies to $A$. So Theorem \ref{thm:case2} follows from Theorem \ref{thm:main}.
\end{ex}
\begin{rmk}
Using the methods of this paper, one could find the distribution $2$-Selmer ranks in the family of favored twists for some of the exceptional cases listed above without too much difficulty. For every exceptional curve, the first or second moment of the size of the $2$-Selmer group in the twist family is notably larger than for a non-exceptional curve. To give a concrete example, the average size of the $2$-Selmer group  for the curve with equation
\[dy^2 = x(x - 1)(x-4)\]
over $\QQ$ is larger than the average size of the $2$-Selmer group for the curve with equation
\[dy^2 = x(x-1)(x+1)\]
as $d$ varies through the squarefree integers.

A more difficult question is whether the results on higher Selmer groups can be generalized to these exceptional cases. The author is hopeful that future research will shed some light on these examples.
\end{rmk}

\subsection{Quadratic twists of abelian varieties}
Generalizing from elliptic curves, we may ask about the $2^{\infty}$-Selmer groups in the quadratic twist family of an abelian variety $A/F$ of dimension $g > 1$. To define an alternating structure on $A[2^{\infty}]$ using the Weil pairing, we need to choose a self-dual isogeny $\lambda: A \to A^{\vee}$ of odd degree over $F$, where $A^{\vee}$ denotes the dual abelian variety to $A$. Any such self-dual isogeny arises from a symmetric line bundle $\mathscr{L}$ on $A$ defined over $\overline{F}$, but might not arise from a line bundle defined over $F$. However, the obstruction to the existence of this bundle is given by an element $c_{\lambda}$ of the $2$-Selmer group of $A$ \cite{Poon99, PR11}, and this Selmer element is universal across the quadratic twist family of $A$ \cite{Morg17}. The technical conditions force us to assume that $c_{\lambda} = 0$ and hence that $\lambda$ comes from a symmetric line bundle defined over $F$.

From the existence of this symmetric line bundle, we can conclude that there is a quadratic refinement $q: A[2] \to \pm 1$ of the Weil pairing on $A[2]$ induced by $\lambda$ \cite[p. 314f.]{Mumf66}. Taking $K = F(A[2])$, we must have that $\Gal(K/F)$ is a subgroup of the orthogonal group $O(q)$ of automorphisms of $A[2]$ preserving this quadratic form. 

\begin{ex}
In the case that $\Gal(K/F)$ equals $O(q)$, we find that the technical conditions are all satisfied. We start by noting that, in this case, $H^1(\Gal(K/F), A[2])$ is nonzero if and only if $g = 3$ and $q$ has Arf invariant $0$ \cite[Theorems 4.2 - 5.2]{Poll71}. In the one exceptional case, we find there is $\sigma$ in $G_F$ such that the restriction of the nontrivial cocycle class to $H^1(\langle \sigma \rangle, A[2])$ is nonzero; specifically, in the notation of the final paragraph of \cite[p. 417]{Poll71}, we see that the restriction of the derivation $\delta$ to the subgroup of $S_8$ generated by the element $(12)(34)(56)(78)$ is nontrivial. In all cases, there are no nonzero universal $2$-Selmer elements.

Given $v \in A[2]$ satisfying $q(v) = 1$, the map
\[w \mapsto w + (q(v + w) - q(w) - 1)v\]
defines an element in $O(q)$ known as a transvection. Using these transvections and the fact that $g > 1$, we may show that any  submodule of $A[2]$ closed under $O(q)$ contains a nonzero singular element $v$, and can conclude from \cite[Theorem 1.3]{Poll71} that $T = A[2]$. The potentially favored and uncofavored conditions hold. Another application of transvections shows that the only equivariant automorphism of $A[2]$ is the identity.

Since the technical conditions are all satisfied, we find that the $2^{\infty}$-Selmer groups of quadratic twists of abelian varieties in this case have the same distribution as the $2^{\infty}$-Selmer groups of elliptic curves in $\Gal(K/F) = S_3$ case.
\end{ex}

This previous example may be considered the most generic case where our theorem applies for quadratic twist families of abelian varieties. A more concrete example comes from the Jacobians of odd degree hyperelliptic curves.
\begin{ex}
Choose $g > 1$ and choose a degree $2g+1$ polynomial 
\[f(x) = a_{2g+1}x^{2g + 1} + a_{2g}x^{2g} + \dots + a_0\] 
with coefficients in $F$ and no repeated roots. Consider the hyperelliptic curve $C$ with equation $y^2 = f(x)$, and take $A$ to be the Jacobian of $C$ with the standard principal polarization. Take $P_1, \dots, P_{2g+1}$ to be the points on this curve where $y = 0$, take $P_0$ to be the point at $\infty$, and write $P_i - P_0$ as $e_i$. Then $e_1, \dots, e_{2g+1}$ generate $A[2]$ subject to the unique relation $e_1 + \dots + e_{2g+1} = 0$ \cite{Schae95}. From this, we find that $K$ equals the splitting field of the polynomial $f$, and that $\Gal(K/F)$ is necessarily a subgroup of $S_{2g+1}$.

In the maximal case where $\Gal(K/F) = S_{2g+1}$, we may apply \cite[Theorem 5.1]{Poll71} to show there are no universal nonzero $2$-Selmer elements. The other technical conditions are easily checked.

Using the Chabauty--Coleman method, one can go from bounds on the rank of the Jacobian of a curve of genus $> 1$ to bounds on the number of rational points on the curve. One concrete application of these methods and Theorem \ref{thm:main} is the following.
\begin{thm}
Choose $g >1$, and choose a polynomial $f(x) = a_{2g+1}x^{2g + 1} + \dots + a_0$ with rational coefficients. Taking $K$ to be the splitting field for $f$ over $\QQ$, we assume that $\Gal(K/\QQ)$ is isomorphic to $S_{2g+1}$. Given a positive integer $d$, take $C_d$ to be the hyperelliptic curve with equation $dy^2 = f(x)$ over $\QQ$. Take $N$ to be the maximal odd integer dividing the conductor of $C_1$, and choose a positive integer $b$ coprime to $2N$. Then we have the following:
\begin{itemize}
\item Suppose the Jacobian of $C_b$ has even $2$-Selmer rank. Then $100\%$ of the $C_d$ with $d$ equal to $b$ mod $8N$ have $1$ rational point, and we have
\[\lim_{H \to \infty}\, H^{-1}\cdot\sum_{\substack{0 < d < H \\ d \equiv b \textup{ mod } 8N }} \big(\#C_{d}(\QQ) - 1\big) = 0.\]
\item Suppose the Jacobian of $C_b$ has odd $2$-Selmer rank. Then $100\%$ of the $C_d$ with $d$ equal to $b$ mod $8N$ have at most $3$ rational points, and we have
\[\lim_{H \to \infty}\, H^{-1}\cdot \sum_{\substack{0 < d < H \\ d \equiv b \textup{ mod } 8N }} \max\big(0,\, \#C_{d}(\QQ) - 3\big) = 0.\]
\end{itemize}
\end{thm}

\begin{proof}
Take $A$ to be the Jacobian of $C_1$, and choose a positive integer $d$ equal to $b$ mod $8N$. The Jacobian of $C_d$ is then identified with the quadratic twist $A^d$ of $A$ associated to $\QQ(\sqrt{d})$. From the above discussion, we see that Theorem \ref{thm:main} applies to the $2^{\infty}$-Selmer groups of the quadratic twists of $A$. We can calculate that this example is in the parity-invariant subcase, so the $2$-Selmer parity of $A^d$ equals the $2$-Selmer parity of $A^b$; see \cite{Yu16} for more details. With this set, the bounds on rational points for $100\%$ of $d$ follow from \cite[Theorem 1.1]{Sto06}. 

From  Theorems \ref{thm:main_mmnt_coarse} and \ref{thm:main}, we also can prove an estimate
\[\sum_{\substack{d < H \\ \text{rank}(A^d/\QQ) > 1}} 7^{\text{ rank}(A^d/\QQ)} = o(H)\]
for $H > 0$. The result on the average of $\# C_d(\QQ)$ then follows from \cite[Theorem 1]{Silv93}.

\end{proof}
A stronger result than this can be proved if the $abc$ conjecture is assumed; see \cite[Corollary 1]{Gran07hyper}.
\end{ex}

\subsection{A general construction for twistable modules}
Suppose $|\FFF| = \ell$. A versatile method for constructing twistable modules is to start with a Galois module $N_0$ over $F$ isomorphic to some finite power of $\QQ_{\ell}/\Z_{\ell}$ as an abelian group and consider the tensor product $N = N_0 \otimes \Zlz$. 

\begin{ex}
\label{ex:rough_av_twists}
Choose $\Zlz$ so $\xi$ has order $\ell$. In the case that $N = A[\ell^{\infty}]$ for an abelian variety $A/F$, we can construct an abelian variety $(A \otimes \Z[\xi])^{\chi}$ over $F$ for any $\chi \in \mathbb{X}_F$. This abelian variety has dimension $(\ell -1 )\dim A$, and its $\ell^{\infty}$-torsion is isomorphic to $(N \otimes \Zlz])^{\chi}$ over $F$. In the case that $\ell = 2$, this variety is just the quadratic twist of $A$ by $\chi$. See e.g. \cite{Mazur07} for background on this construction. 

Assuming $\chi$ is nontrivial, we also have an identity,
\begin{equation}
\label{eq:AKAF}
\text{rank}(A/K)  = \text{rank}(A/F) + \text{rank}\left((A \otimes \Z[\xi])^{\chi}/F\right),
\end{equation}
where $K/F$ is the cyclic extension associated to $\chi$. This motivates the study of the Selmer groups of the twists of $A \otimes \Zlz$.

From Remark \ref{rmk:finite_favored}, we know there we can find finitely submodules $T_1, \dots, T_k$ of $N \otimes \Zlz$ so that every $\chi \in \mathbb{X}_F$ favors some $(N  \otimes Zlz)/T_i$. Giving these modules the usual local conditions associated to an abelian variety, we also have
\[\text{rank}\big((A \otimes \Z[\xi])^{\chi}/F\big) \,\le\, (\ell - 1)\cdot r_{\omega}\big(((N  \otimes \Z[\xi])/T_i)^{\chi}\big).\]
for all $i \le k$. From this relation, we see that Theorem \ref{thm:rank_bnd} follows from Theorem \ref{thm:main_mmnt_coarse}.

The non-self-dual case of Theorem \ref{thm:main} cannot apply to this example, as every polarization of $A$ over $F$ will either lead to self-dual structure for $N[\omega]$ or have a kernel that can be manipulated to give a proper nonzero cofavored submodule of $N[\omega]$. So suppose $N$ has alternating structure. In this case, the Cassels--Tate pairings indexed by even numbers for the twists of $N$ are symmetric rather than alternating for $\ell$ odd; see \cite[Remark 4.14]{Smi22a}. We cannot control the diagonal entries of the matrices corresponding to these pairings, which is why Case \ref{case:alt} included the assumption $\#\FFF =2$.

However, if $N$ otherwise obeys the conditions of the alternating case, our methods apply without adjustment to find the distribution of the Cassels--Tate pairings indexed by odd integers. Just from these pairings, we can show that
\[\text{rank}(A/K) - \text{rank}(A/F)  \le \ell - 1\]
for $100\%$ of the cyclic degree $\ell$ extensions $K$ of $F$.

In contrast to the quadratic extensions, we expect that $\text{rank}(A/K) $ should equal $\text{rank}(A/F)$ for $100\%$ of the cyclic degree $\ell$ extensions $K$ of $F$ for $\ell$ odd. See \cite[Conjecture 1.2]{Chant07} for precise conjectures for elliptic curves over $\QQ$. To prove such a zero density result using our methods would require some control on the diagonal terms for the Cassels--Tate pairings indexed by even numbers.
\end{ex}

\begin{ex}
Choose a nonzero element $m$ in $F$. If $F$ contains $\mu_3$, the $3^{\infty}$ torsion of the elliptic curve $A: y^2 = x^3 + m$ is a twistable module that satisfies the conditions of the alternating case besides $\#\FFF = 2$. Using our methods, we could show that the elliptic curves with Weierstrass form $y^2 = x^3 + md^2$ have $3^{\infty}$-Selmer corank at most $1$ for $100\%$ of integers $d$ in $F$, where we are ordering $d$ by the norm of the radical of the ideal $(d)$.

Unfortunately, our method completely fails if $F$ does not contain $\mu_3$, as an order $3$ automorphism of $A$ is only defined over $\QQ(\mu_3)$. We can adjust our methods to handle non-equivariant automorphisms $\xi$, and we can still control $\omega$-Selmer groups in this non-equivariant case. However, twice the $\omega$-Selmer rank is no longer an upper bound for the rank of the eliptic curve. All we can say is that the rank of $y^2 = x^3 + md^2$  is at most the sum of the $\omega$-Selmer ranks of this curve and the isogenous curve $y^2 = x^3 - 27md^2$. Our trick to focus on favored twists no longer works, and it is unclear how to adjust our methods to say anything interesting about this family.
\end{ex}

\begin{ex}
A simple example of the twisting construction is given by the twistable module $N = \QQ_{\ell}[\xi]/\Z_{\ell}[\xi]$ with a trivial $G_F$ action, where $\xi$ has order $\ell$. In \cite[Example 4.5]{Smi22a}, we saw that we could decorate this twistable module with local conditions so, given a cyclic degree $\ell$ extension $K/F$ and a nontrivial character $\chi:\Gal(K/F) \to \mu_{\ell}$, we had an injection
\[ \left(\Cl^* L\big/ (\Cl^* L)^{\Gal(L/F)}\right)[\ell^{\infty}] \hookrightarrow \Sel\, N^{\chi}.\]
A condition for this map to be an isomorphism is given in \cite[Proposition 7.4]{MS21b}. Using Proposition \ref{prop:hateful_eight}, we find that the set of twists not satisfying this condition in the families considered in \cite[Theorem 1.11]{Smi22a} is negligible.

We now check when the technical conditions apply in this example. The potentially favored, uncofavored, and no extra automorphisms conditions are automatic since $N[\omega]$ is $1$-dimensional. There are also no universal $\omega$-Selmer elements. So this example satisfies the conditions for Case \ref{case:nodual} if it lacks self-dual structure.

If $F$ does not contain $\mu_{\ell}$, then $N[\omega]$ and $N[\omega]^{\vee}$ are nonisomorphic, so there is no self-dual structure. So suppose $F$ contains $\mu_{\ell}$. If $\ell$ is odd, we find that $N[\omega^2]$ is isomorphic to $(\Z/\ell\Z)^2$ as a $G_F$-module, so $N[\omega^2]$ is isomorphic to $N[\omega^2]^{\vee}$ and there is self-dual structure. However, if $\ell$ is even, we see that $N[\omega^2]$ is isomorphic to $\Z/4\Z$, so the isomorphism between $N[\omega]$ and $N[\omega]^{\vee}$ only commutes with the connecting map if $F$ contains $\mu_4$.

This accounts for the technical conditions on \cite[Theorem 1.11]{Smi22a}. The translation from the set $X_{F, \ell, r_{1}'}(H)$ to sets of the form $\mathbb{X}_F(H', (\chi_v)_v)$ is straightforward. For any non-archimedean $v \in \Vplac_0$, we find that the preimage of $W_v(\chi_v)$ in $H^1(G_v, N[\omega])$ has cardinality $\ell$. For archimedean places, we find this preimage has order $\ell$ if $v$ is real and $\chi_v$ ramifies at $v$, and otherwise has order $1$. So, given $\chi$ in $\mathbb{X}_F((\chi_v)_v)$, we find that
\[r_{\omega}(N^{\chi}) - r_{\omega}((N^{\vee})^{\chi}) = -r_2 - r_1'\]
by \eqref{eq:Wiles}, where $r_2$ equals the number of conjugate pairs of complex places, and $r_1'$ equals the number of real places $v$ where $\chi_v$ is trivial. With this calculated, \cite[Theorem 1.11]{Smi22a} follows.

We note that the $\ell^{\infty}$-portion of the class groups of the Galois degree $\ell$ extensions of $F$ have a different distribution if $F$ contains $\mu_{2\ell}$. For example, we can use the methods of this paper to show that the $\omega$-portion of the class group has larger moments in this case. In our estimation, it should be challenging but possible to adjust the methods of this paper and \cite{Smi22a} to compute the distribution of class groups in the presence of extra roots of unity.
\end{ex}

\begin{rmk}
Given a number field $L$ and an integer $m \ge 1$, the $\ell^{\infty}$-torsion of the algebraic $K$-group $K_{2m} \mathscr{O}_L$ of the roots of unity of $L$ can be identified with the Selmer group of the Tate twist $\QQ_{\ell}/\Z_{\ell}(m+1)$ over $L$ endowed with a certain reasonable set of local conditions; see \cite[Theorem 7.2]{Weib05}. Because of this, one should be able to find distributional results for these groups that are analogous to our results for class groups, with Theorem \ref{thm:main} applying whenever $\mu_{2\ell}$ does not lie in $F$. For other work on the heuristics of these $K$-groups, see \cite{Jord17}.
\end{rmk}

\section{Tamagawa ratios and cofavored submodules}
\label{sec:tamagawa}

Proposition \ref{prop:tamagawa} claimed that large Tamagawa ratios lead to large $\omega$-Selmer groups. Our first goal is to prove this result.
\begin{proof}[Proof of Proposition \ref{prop:tamagawa}]
Take $T$ and $\chi$ as in the proposition, and take $\Vplac_1$ to be the set of primes dividing $\mfh_F(\chi)$. The group
\begin{equation}
\label{eq:tama_injectable}
\ker\left(H^1(G_F, T) \to \prod_{v \in \Vplac_0} H^1(G_v, N[\omega]) \times \prod_{v \in \Vplac_1} H^1(G_v, N^{\chi}) \times \prod_{v \not \in \Vplac_0 \cup \Vplac_1} H^1(I_v, N[\omega])\right)
\end{equation}
maps into $\Sel^{\omega} N^{\chi}$ with kernel equal to the image $H^0(G_F, N[\omega]/T)$ under the connecting map. For $v$ in  $\Vplac_1$, we have an exact sequence
\begin{align*}
0 \to & H^0(G_v, T) \to H^0(G_v, N^{\chi}) \to H^0(G_v, N^{\chi}/T)\\
\to & H^1(G_v, T) \to H^1(G_v, N^{\chi}),
\end{align*}
so
\begin{align*}
\# \ker\big(H^1(G_v, T) \to H^1(G_v, N^{\chi})\big) &= \,\frac{\# H^0(G_v, N^{\chi}/T)  \cdot \# H^0(G_v, T)}{\# H^0(G_v, N^{\chi})} \\
&=  \,\frac{\# H^0(G_v, N/T[\omega])  \cdot \# H^0(G_v, T)}{\# H^0(G_v, N[\omega])}.
\end{align*}
We then can use the global Euler-Poincar\'{e} characteristic formula of Tate \cite[Section 8.7.9]{Neuk08} to say that \eqref{eq:tama_injectable} has cardinality at least
\[c \cdot \prod_{v \in \Vplac_1} \frac{\# H^0(G_v, N/T[\omega])}{\# H^0(G_v, N[\omega])},\]
where $c$ is some positive number that does not depend on the choice of $\chi$. This gives the proposition.
\end{proof}

This proposition forced us to restrict our attention to favored twists in Theorem \ref{thm:main}. But this proposition is not the end of the importance of Tamagawa ratios in our work. Given a potentially-favored  twistable module $N$, the average size of $\Sel^{\omega} N^{\chi}$ as $\chi$ varies through $\XXfav_{F, N}$ tends towards a certain sum indexed by the submodules $T$ of $N[\omega]$ that have Tamagawa ratio exactly $1$. In particular, twistable modules with more cofavored submodules tend to have $\omega$-Selmer groups of larger average size.

This led to our restriction to uncofavored modules. But this only gives control over the first moment of the $\omega$-Selmer groups in this family. For $k > 1$, if the twistable module $N^{\oplus k}$ has extra cofavored submodules beyond the obvioius ones, we find that the $k$-moments of the sizes of the $\omega$-Selmer groups in this family are larger than would be seen for a family without extra cofavored modules. The main goal for the rest of this section is to show how one can control the collection of cofavored submodules in these powers.

Throughout this section, we fix $F$, $\Zlz$, and a Galois extension $K$ of $F$ containing $\mu_{\ell^{k_0}}$ where $\ell^{k_0}$ is the order of $\xi$. All twistable modules referenced will be defined with respect to $F$ and $\Zlz$, and the $\omega^2$-torsion of all referenced modules will be acted on trivially by $G_K$. We will take $G_0 = \Gal(K/F(\mu_{\ell^{k_0}}))$.
\begin{defn}
\label{defn:pfav2}
Given a twistable module $N$, we call $N$ \emph{potentially favored} if
\begin{itemize}
\item For any $G_F$-submodule $T$ of $N[\omega]$, we have
\begin{equation}
\label{eq:not_super_unfavored}
\sum_{\sigma \in G_0} \dim H^0(\langle \sigma \rangle,\, N[\omega]) \ge \sum_{\sigma \in G_0} \dim H^0(\langle \sigma \rangle,\, (N/T)[\omega]),
\end{equation}
and
\item There is some function
\[w: G_0 \rightarrow \mathbb{R}\]
so that, for any $G_F$-submodule $T$ of $N[\omega]$, we either have
\[\sum_{\sigma \in G_0} w(\sigma) \cdot \dim H^0(\langle \sigma \rangle,\, N[\omega]) \,>\, \sum_{\sigma \in G_0} w(\sigma) \cdot \dim H^0(\langle \sigma \rangle,\, (N/T)[\omega])\]
or
\[\dim H^0(\langle \sigma \rangle,\, N[\omega]) = \dim H^0(\langle \sigma \rangle,\ (N/T)[\omega]) \quad\text{for all }\, \sigma \in G_0.\]
\end{itemize}
We call the set of functions $w$ that obey the second property the \emph{superlatives} for $N$.
\end{defn}

We will prove that this definition of potentially favored is equivalent to Definition \ref{defn:pfav} in Proposition \ref{prop:pfav}.

\begin{defn}
Given an exact sequence
\[0 \rightarrow \Z^{m} \rightarrow \Z^{n} \rightarrow \Z^{n - m } \rightarrow 0,\]
we get an associated sequence  of twistable modules
\[0 \rightarrow N^{\oplus m } \rightarrow N^{\oplus n} \rightarrow N^{\oplus n - m}\rightarrow 0 \]
by tensoring with $N$. We call any such exact sequence a \emph{basic exact sequence}.
\end{defn}

I would like to thank David Yang for his help in finding the following proposition.
\begin{prop}
Suppose $N_1$ and $N_2$ are twistable modules. Then, if $N_1$ and $N_2$ are potentially favored, and if $w: G_0 \rightarrow \mathbb{R}^{+}$ is a superlative for both, then $N_1 \oplus N_2$ is potentially favored and has $w$ as a superlative.
\end{prop}
\begin{proof}
Take $N_1$, $N_2$, and $w$ as in the statement of the proposition, and take $T$ to be a submodule of $N_1 \oplus N_2$. Taking the quotient of the exact sequence
\[0 \rightarrow N_1 \xrightarrow{\iota} N_1 \oplus N_2 \xrightarrow{\pi} N_2 \rightarrow 0\]
by $T$ gives the exact sequence
\[0 \rightarrow N_1/\iota^{-1}(T) \rightarrow (N_1 \oplus N_2)/T \rightarrow N_2/\pi(T) \rightarrow 0.\]
Here, $\iota$ and $\pi$ are the obvious injection and projection.

The long exact sequence associated to this short exact sequence gives
\begin{align*}
&\dim H^0(\langle \sigma \rangle,\, (N_1 \oplus N_2)/T[\omega] )\\
&\qquad \le \dim H^0(\langle \sigma \rangle,\, N_1/\iota^{-1}(T)[\omega]) + \dim H^0(\langle \sigma \rangle,\, N_2/\pi(T)[\omega])
\end{align*}
for all $\sigma$ in $G_0$. Then
\begin{align*}
 &\sum_{\sigma \in G_0} w(\sigma) \cdot \dim H^0(\langle \sigma \rangle,\, (N_1 \oplus N_2)/T[\omega]) \\
&\,\,\,\le  \sum_{\sigma \in G_0} w(\sigma) \cdot \dim H^0(\langle \sigma \rangle,\, N_1/\iota^{-1}(T)[\omega])\, +\,\sum_{\sigma \in G_0} w(\sigma) \cdot \dim H^0(\langle \sigma \rangle,\, N_2/\pi(T)[\omega])\\
 &\,\,\,\le  \sum_{\sigma \in G_0} w(\sigma) \cdot \dim H^0(\langle \sigma \rangle,\, N_1[\omega]) \,+\,\sum_{\sigma \in G_0} w(\sigma) \cdot \dim H^0(\langle \sigma \rangle,\, N_2[\omega])\\
 &\,\,\,=  \sum_{\sigma \in G_0} w(\sigma) \cdot \dim H^0(\langle \sigma \rangle,\, (N_1 \oplus N_2)[\omega]).
\end{align*}
If the first and last entries in this sequence are equal, the positivity of $w$ forces
\[\dim H^0(\langle \sigma \rangle,\, (N_1 \oplus N_2)/T[\omega])  = \dim H^0(\langle \sigma \rangle,\, (N_1 \oplus N_2)[\omega]) \quad\text{for }\, \sigma \in G_0.\]
So $w$ is is a superlative for $N_1 \oplus N_2$. Applying this argument for the superlative $1 + \epsilon w$ with $\epsilon$ approaching zero shows that $N_1 \oplus N_2$ is potentially favored, giving the proposition.
\end{proof}

\begin{rmk}
\label{rmk:not_quite_cofavored}
Given a twistable module $N$, a nonnegative integer $a$, and a $G_F$-submodule  $T$ of $N^{\oplus a}[\omega]$, we see that the logic of the above proposition gives that there is some sequence of $G_F$-submodules $T_1, \dots, T_j$ of $N[\omega]$ and some sequence of nonnegative integers $a_1, \dots, a_j$ with sum $a$ so that
\[\dim H^0(\langle \sigma \rangle, N^{\oplus a}/T[\omega])\, \le \,\sum_{i \le j} a_i \cdot \dim H^0(\langle \sigma \rangle, N/T_i[\omega])\quad\text{for all }\, \sigma \in G_0.\]
This  gives a useful upper bound for the function $\sigma \mapsto \dim H^0(\langle \sigma \rangle,\, N^{\oplus a}/T[\omega])$ in the case that $T$ is not cofavored.
\end{rmk}

Given $\sigma \in G_0$ and twistable modules $N_1$ and $N_2$, we can consider the connecting maps $\delta_{\sigma, N_1}$ and $\delta_{\sigma, N_2}$ as defined in \eqref{eq:connecting_map}. Given a $G_F$-equivariant homomorphism $\phi: N_1[\omega] \to N_2[\omega]$, we say that $\phi$ commutes with the connecting maps if
\begin{equation}
\label{eq:comm_conn}
\phi \circ \delta_{\sigma, N_1} = \delta_{\sigma, N_2} \circ \phi
\end{equation}
for all $\sigma$ in $G_0$.

\begin{prop}
\label{prop:2sum_cofav}
Take $N_1$ and $N_2$ to be potentially favored modules sharing some superlative $w$. Suppose the only cofavored submodules of $N_i$ are $0$ and $N_i[\omega]$ for $i = 1, 2$. Then the cofavored submodules of $N_1 \oplus N_2$ are $0$, $0 \oplus N_2[\omega]$, $(N_1 \oplus N_2)[\omega]$, and those modules of the form
\[\left\{(x, \beta(x))\,:\,\, x \in N_1[\omega]\right\},\]
where $\beta: N_1[\omega] \rightarrow N_2[\omega]$ is an equivariant homomorphism commuting with the connecting maps.
\end{prop}
\begin{proof}
Choose a cofavored submodule $T$ of $N_1 \oplus N_2$, and consider the standard exact sequence
\[0 \rightarrow N_2 \xrightarrow{\iota} N_1 \oplus N_2 \xrightarrow{\pi} N_1 \rightarrow 0.\]
From the argument of the previous proposition, we see that $\iota^{-1}(T)$ is cofavored in $N_2$ and $\pi(T)$ is cofavored in $N_1$. If $\iota^{-1}(T)$ is nonzero, it is $N_2[\omega]$, and we find that $T$ is either $0 \oplus N_2[\omega]$ or $(N_1 \oplus N_2)[\omega]$. So we assume it is zero.

If $\pi(T)$ is also zero, $T$ is zero, so we can assume $\pi(T)$ is $N_1[\omega]$. Then $T$ is the graph of an equivariant homomorphism $\beta: N_1[\omega] \rightarrow N_2[\omega]$.

Following the logic of the previous proposition we find that, since $T$ is cofavored, the connecting map
\begin{equation}
\label{eq:N1N2con}
H^0\big(\langle \sigma \rangle, \, N_1[\omega]\big) \royarrow H^1\big(\langle \sigma \rangle, \, N_2[\omega]\big)
\end{equation}
corresponding to the sequence
\[0 \rightarrow N_2[\omega] \rightarrow ((N_1 \oplus N_2)/T)[\omega] \rightarrow N_1[\omega] \rightarrow 0\]
is zero for $\sigma$ in $G_0$. More explicitly, for $x \in H^0(\langle \sigma \rangle, N_1[\omega])$, this means the cocycle on $\langle \sigma \rangle$ defined by
\[\sigma \mapsto \beta\big((\sigma - 1) \tfrac{1}{\omega} x\big) - (\sigma - 1) \tfrac{1}{\omega} \beta(x)\]
is a coboundary. But this is equivalent to $\beta$ commuting with the connecting maps.

Conversely, we find that the graph of any $G_F$-homomorphism $\beta$ that commutes with the connecting maps is cofavored in $N_1 \oplus N_2$, and we have the proposition.
\end{proof}
\begin{rmk}
\label{rmk:Wedd}
Given $N_1$ and $N_2$ satisfying the conditions of this proposition, we find that the nonzero homomorphisms from $N_1[\omega]$ to $N_2[\omega]$ that commute with the connecting maps are all isomorphisms, as the kernel and images of these maps are cofavored. In particular, the set of $G_F$ endomorphisms of $N_1[\omega]$ that commute with connecting maps form a finite division ring and hence a finite field of characteristic $\ell$ by Wedderburn's theorem. The set of $G_F$ homomorphisms from $N_1[\omega]$ to $N_2[\omega]$ that commute with connecting maps is then either $0$ or a one-dimensional vector space over this field.
\end{rmk}

Under some technical assumptions, we can now classify the cofavored submodules of direct sums of potentially favored modules.
\begin{prop}
\label{prop:cofav_form}
Take $N_1, \dots, N_r$ to be potentially favored modules sharing some superlative. Suppose we have the following:
\begin{itemize}
\item For $i \le r$, the only cofavored submodules of $N_i$ are $0$ and $N_i[\omega]$;
\item For $i \le r$, the only $G_F$-automorphisms of $N_i[\omega]$ that commute with the connecting maps are those in $\FFF_{\ell}^{\times}$; and
\item For $i, j \le r$ with $i \ne j$, there is no $G_F$-isomorphism of $N_i[\omega]$ and $N_j[\omega]$ that commutes with the connecting maps.
\end{itemize}
Then, for $a_1, \dots, a_r$ nonnegative integers, the cofavored submodules of
\begin{equation}
\label{eq:N1Nr}
N_1^{\oplus a_1} \oplus \dots \oplus N_r^{\oplus a_r}
\end{equation}
are those of the form
\begin{equation}
\label{eq:A1N1ArNr}
A_1 \otimes N_1[\omega] \oplus \dots \oplus A_r \otimes N_r[\omega],
\end{equation}
where $A_i$ is a vector subspace of $\FFF_{\ell}^{a_i}$ for $i \le r$.
\end{prop}
\begin{proof}
We work by induction on $ \sum_i a_i$. From Proposition \ref{prop:2sum_cofav}, the result holds for this sum at most two. Now take $a \ge 3$, and suppose every cofavored submodule of \eqref{eq:N1Nr} takes the form \eqref{eq:A1N1ArNr} for $\sum_i a_i < a$. Then suppose $\sum_i a_i = a$, and take $T$ to be a cofavored submodule of \eqref{eq:N1Nr}.

Choose any direct sum of basic exact sequences
\begin{equation}
\label{eq:biaibiai}
0 \rightarrow \bigoplus_i  N_i^{\oplus b_i} \xrightarrow{\,\,\iota\,\,} \bigoplus_i N_i^{\oplus a_i} \xrightarrow{\,\,\pi\,\,} \bigoplus_i N_i^{\oplus a_i - b_i} \rightarrow  0.
\end{equation}
with $\sum_i a_i > \sum_i b_i > 0$. We claim that we can assume that $\iota^{-1}(T)$ is zero and that
\begin{equation}
\label{eq:Niai_surj}
\pi(T) = \bigoplus_i N_i^{\oplus a_i - b_i}[\omega].
\end{equation}
First, suppose $\iota^{-1}(T)$ is nonzero. It is cofavored, and hence can be written in the form \eqref{eq:A1N1ArNr} by the induction step. There is then some other direct sum of basic exact sequences
\begin{equation}
\label{eq:ciaiciai}
0 \rightarrow \bigoplus_i  N_i^{\oplus c_i} \xrightarrow{\,\,\iota_c\,\,} \bigoplus_i N_i^{\oplus a_i} \xrightarrow{\,\,\pi_c\,\,} \bigoplus_i N_i^{\oplus a_i - c_i} \rightarrow  0
\end{equation}
with $\sum_i c_i > 0$ so that
\[\bigoplus_i N_i^{\oplus c_i}[\omega] =  \iota^{-1}_c(T).\]
The module $\pi_c(T)$ is cofavored, and hence expressible in the form \eqref{eq:A1N1ArNr}. $T$ equals its preimage in $\bigoplus_i N_i^{\oplus a_i}[\omega]$, and is hence also expressible in this form.

Now suppose \eqref{eq:Niai_surj} fails to hold. The module $\pi(T)$ is cofavored, and hence expressible in the form \eqref{eq:A1N1ArNr}. The same holds for the preimage $\pi^{-1}(\pi(T))[\omega]$. We now choose \eqref{eq:ciaiciai} so that
\[\iota_c \left(\bigoplus_i N_i^{\oplus c_i}\right) =\pi^{-1}(\pi(T))[\omega] \supseteq T.\]
Because \eqref{eq:Niai_surj} does not hold, we have that $\sum_i a_i > \sum_i c_i$, so $\iota_c^{-1}(T)$ is expressible in the form \eqref{eq:A1N1ArNr} by the induction step. But $T = \iota_c(\iota_c^{-1}(T))$, and hence is also expressible in this form.

We have now reduced to the case where, given any direct sum of basic exact sequences \eqref{eq:biaibiai} satisfying $\sum_i a_i  > \sum_i b_i > 0$, we have $\iota^{-1}(T) = 0$ and \eqref{eq:Niai_surj}. But in this case, we must have
\begin{align*}
\dim T = \sum_i &(b_i - a_i) \cdot \dim N_i[\omega]\\
&\text{for all }\,b_1, \dots, b_r \ge 0 \,\,\text{so that}\,\, \sum_i a_i  > \sum_i b_i > 0.
\end{align*}
Since $\sum_i a_i$ is at least $3$ and the $\dim N_i$ are all positive, it is easy to prove that this statement cannot hold. We have thus proved the proposition.
\end{proof}

\section{Local conditions for $\omega$-Selmer groups in grids}
\label{sec:character}

In \cite[Section 4]{Smi22a}, we showed that the $\omega$-Selmer group for a given twist $N^{\chi}$ could be calculated in terms of the classes, spins, and symbols of the primes dividing $\mfh(\chi)$. To prove our distributional results about $\omega$-Selmer groups, we need to make this result more explicit. This starts with a reconsideration of symbols.

\subsection{Involution spins}
Our terminology will come from Section 3 of \cite{Smi22a}, and will in particular involve the alternative symbols defined in Section 3.1.

Most of the identities we need for the symbols $\symb{\ovp}{\ovq}$ appear in \cite[Proposition 3.15]{Smi22a}. These are easily adapted to alternative symbols. In particular, given $\tau \in G_F$ and primes $\ovp, \ovq$ of $\ovQQ$ not over $\Vplac_0$ satisfying $\tau \ovp \cap K \ne \ovq \cap K$, we can determine the quotient
\begin{equation}
\label{eq:tau_rec}
\symb{\tau\ovp}{\ovq}' \big/ \tau\left(\symb{\tau^{-1}\ovq}{\ovp}' \right)
\end{equation}
just from $\tau$, the class of $\ovp$, and the class of $\ovq$.

There is a subtle supplement to this result that turns out to be important.
\begin{prop}
\label{prop:involution_spin}
Choose any unacked starting tuple $(K/F, \Vplac_0, e_0)$, and choose a ramification section for the starting tuple. Take $\ovp$ and $\ovp_0$ to be primes in the same class, and define $E = E(\ovp)$ to be the minimal extension of $F$ so $K/E$ is inert at $E \cap \ovp$ as above. Choose $\tau$ in $G_F$ outside $G_E$ so that
\[G_E \cdot \tau \cdot G_E \,=\, G_E \cdot \tau^{-1}\cdot G_E.\]
Then
\[\symb{\tau \ovp}{\ovp}' \big/ \symb{\tau \ovp_0}{\ovp_0}' \in \left((\mu_{e_0})_{G_{E + \tau E}}\right)^2.\]
\end{prop}
\begin{proof}
We first note that it suffices to consider the case where $e_0 = 2$. We also claim that we can reduce to the case where $K/F$ is a quadratic extension. To prove this, take $\sigma = \FrobF{F}{\ovp}$, and find integers $j, k$ so we have
 \[\sigma^k \tau = \tau^{-1} \sigma^j\,\text{ in }\, \Gal(K/F).\]
By replacing $\tau$ with $\sigma^k \tau$, we may as well assume that
\[\tau^2 = \sigma^{j+k}.\]
Then $\sigma^{j + k}$ is in both $G_E$ and $G_{\tau E}$, but $\tau$ is in neither. Taking $K' = E + \tau E$, we see $\tau$ is a nontrivial involution on $K'$, and we take $F'$ to be the subfield fixed by this involution. We can then calculate $\symb{\tau\ovp}{\ovp}'$ with respect to the tuple $(K/F, \Vplac_0, 2)$ by calculating it with respect to the tuple $(K'/F', \Vplac_0', 2)$, where $\Vplac_0'$ is the set of places of $F'$ over $\Vplac_0$.

Having made these reductions, the result follows from the theory of theta groups \cite[Proposition 8.1]{MS21b}.
\end{proof}

\subsection{Explicit local conditions}
\label{ssec:explicit_local}
Take $N$ to be a twistable module defined with respect to $F$ and $\FFF$, choose $(K/F, \Vplac_0)$ unpacking $N$, and choose an associated ramification section $\mfB$. Take $X = \prod_{s \in S} X_s$ to be a grid of ideals in the sense of \cite[Definition 4.7]{Smi22a}. In particular, we assume the primes in $X_s$ share a class $\class{X_s}$. Taking $\Gal(K(\Vplac_0)/F)/\sim$ to be the set of equivalence classes considered in \cite[Definition 3.6]{Smi22a}, we choose $\sigma_s$ in $\Gal(K(\Vplac_0)/F)$ so $\FrobF{F}{\ovp}$ lies in the class of $\sigma_s$ in $\Gal(K(\Vplac_0)/F)/\sim$ for every $\ovp$ in $X_s$. We assume that $\sigma_s$ acts trivially on $\mu_{\ell^{k_0}}$, where $\ell^{k_0}$ is the order of $\xi$.

Fix $\kappa$ in $\FFF(-1)^{\times}$ and choose a unit $\gamma_s \in (\Z/\#\FFF \Z)^{\times}$ for each $s \in S$. Also choose some $\chi_0 \in \Hom_{\text{cont}}(G_F, \FFF)$ that is ramified only at places in $\Vplac_0$. With this auxiliary data fixed, we associate a homomorphism $\chi(x) \in \mathbb{X}_F$ to a given $x \in X$ using the formula
\[\chi(x) = \chi_0 + \sum_{s \in S} \mfB_{\pi_s(x),\, \FFF}(\gamma_s \kappa).\]

Recall that $\msS_{N[\omega]/F}(\Vplac_0)$ denotes the cocycle classes in $H^1(G_F, N[\omega])$ that are unramified away from $\Vplac_0$. Following  \cite[Definition 4.7]{Smi22a}, define
\[\mathscr{M} = \mathscr{M}(N) = \msS_{N[\omega]/F}(\Vplac_0) \oplus \bigoplus_{s \in S} N[\omega]^{\sigma_s}.\]
If we now choose $m = (\phi_0, \,(n_s)_s)$ in $\mathscr{M}$, 
we may consider the sum
\[\Psi_{x, N}(m) = \phi_0 + \sum_{s \in S} \mfB_{\pi_s(x),\, N[\omega]}(\gamma_s n_s \cdot \kappa ).\]
Here, we have identified $N[\omega] \otimes \FFF$ with $N[\omega]$ via the isomorphism sending $n \otimes \xi$ to $n$ for $n \in N[\omega]$. We take $\mathscr{M}_0$ to be the subset of such tuples for which $\Psi_x(m)$ obeys the local conditions at each $v \in \Vplac_0$ for any (or equivalently every) $\chi(x)$. Following the logic of \cite[Proposition 4.8]{Smi22a}, we also see that $\Psi_x(m)$ obeys the local condition at $\pi_s(x)$ if and only if
\[a_{s, x}(m)\,:=\, (\Psi_x(m) - \chi(x) \cup n_s)\big(\FrobF{F}{\pi_s(x)}\big) - \delta_{\sigma_s}(n_s)\]
is $0$ in $N[\omega]_{\sigma_s}$, where $\delta_{\sigma_s}(n_s)$ denotes $(\sigma_s - 1) \omega^{-1} n_s$ evaluated in $N[\omega^2]$. This uses the fact that $\delta_{\sigma_s}(n_s)$ only depends on the image of $\sigma_s$ in $\Gal(K(\Vplac_0)/F)/\sim$, which follows from \cite[Remark 3.7]{Smi22a}.

Given $\sigma, \tau \in G_F$, take $B(\sigma, \tau)$ to be a set of representatives of the collection of double cosets $\langle \sigma G_K \rangle \backslash G_F / \langle \tau G_K\rangle$. We assume this set of representatives contains $1$. Applying \cite[Proposition 3.22]{Smi22a} gives
\begin{align}
\label{eq:ps_cond}
a_{s, x}(m) \,=\, &\phi_0(\sigma_s)  -\delta_{\sigma_s}(n_s) - n_s \cdot \chi_0(\sigma_s) \\
&\qquad+\,\,\sum_{t} \sum_{\tau \in B(\sigma_t, \sigma_s)}  \gamma_t \left( \tau^{-1} n_t - n_s\right) \kappa \left( \symb{\tau\pi_s(x)}{\pi_t(x)}'\right) \nonumber.
\end{align}

As a general rule, the proportion of $x \in X$ for which $\Psi_x(m)$ is a Selmer element is smaller for choices of $(n_s)_s$ where the submodule of $N[\omega]$ generated by the $(n_s)_s$ is large. The following definition helps make this more precise.
\begin{defn}
\label{defn:ram_sub}
Define the \emph{ramification subspace} $Q = \bfQ((n_s)_s)$ associated to $(n_s)_s$ to be the submodule of $N[\omega]$ spanned by all elements of the form $\tau_1 n_s - \tau_2 n_t$ with $s, t \in S$ and $\tau_1, \tau_2 \in G_F$. 

If $\phi \in H^1(G_F, N^{\chi(x)})$ is parameterized by $m = (\phi_0, (n_s)_s)$, we also write $\bfQ(\phi)$ and $\bfQ(m)$ for $\bfQ((n_s)_s)$.
\end{defn}

With this definition set,  we see that a necessary condition for the $\Psi_x(m)$ to obey the local condition at $\pi_s(x)$ is that
\begin{equation}
\label{eq:reasonable_Q}
\phi_0(\sigma_s) - \delta_{\sigma_s}(n_s) - n_s\cdot \chi_0(\sigma_s) \,\equiv\, 0 \,\,\text{ mod }\, Q + (\sigma_s - 1)N[\omega].
\end{equation}
In particular, if $[X_s] = [X_t]$, we must have
\[\delta_{\sigma_s}(n_s - n_t) \,\equiv 0 \,\,\text{ mod }\, Q + (\sigma_s - 1)N[\omega].\]
In other words, $n_s - n_t$ must lie in the image of the map $ (N/Q)[\omega]^{\sigma_s}\to Q^{\sigma_s}$ coming from the long exact sequence associated to
\[0 \to N[\omega]/Q \to (N/Q)[\omega] \to Q \to 0.\]
Given any submodule $Q$ of $N[\omega]$, we take $\mathscr{M}_1(Q)$ to be the subspace of elements in $\mathscr{M}_0$ satisfying the condition \eqref{eq:reasonable_Q} for every $s \in S$. We then take 
\[\mathscr{M}_1 = \bigcup_Q\,\left \{m \in \mathscr{M}_1(Q)\,:\,\, \bfQ(m) = Q\right\}.\]

\subsection{The main character sum}
We take the notation
\[\mathscr{R} = \bigoplus_{s \in S} (N[\omega]^{\vee})^{\sigma_s}.\]
We then define a pairing
\[\big\langle\,,\,\big\rangle_x\,\colon\,\mathscr{R}\, \times\, \mathscr{M} \to \tfrac{1}{\ell}\Z/\Z\]
by
\[\big\langle\,(r_s)_s, \,m)\big\rangle_x = \sum_s \text{lg}\left(\gamma_s \cdot r_s\big(a_{s, x}(m)\big)\right),\]
where $\text{lg}$ is the homomorphism taking the image of $\overline{\zeta}$ in $\mu_{\ell}$ to $1/\ell$ and where the pairing between $\FFF$ and $\mu_{\ell}$ takes $\xi \otimes \alpha$ to $\alpha$ for $\alpha$ in $\mu_{\ell}$. The $\omega$-Selmer group of $N^{\chi(x)}$ then satisfies
\[ \# \Sel^{\omega} N^{\chi(x)} = \frac{1}{\#\mathscr{R}}\sum_{m \in \mathscr{M}_1} \sum_{r \in \mathscr{R}} \exp\left(2\pi i \big\langle r,\, m\big\rangle_x\right).\]

We are interested in the average size of the $\omega$-Selmer group of $N^{\chi(x)}$ as $x$ varies through $X$, which can be calculated as
\begin{equation}
\label{eq:moment_to_pairs}
\frac{1}{\# X} \sum_{x \in X} \# \Sel^{\omega}N^{\chi(x)} = \frac{1}{\# X \cdot \# \mathscr{R}} \sum_{m \in \mathscr{M}_1} \sum_{r \in \mathscr{R}} \sum_{x \in X} \exp\left(2\pi i \big\langle r,\, m\big\rangle_x\right).
\end{equation}
For fixed $(r, m)$, the inner sum
\begin{equation}
\label{eq:pair_classifier}
\frac{1}{\# X} \sum_{x \in X} \exp\left(2\pi i \big\langle r,\, m\big\rangle_x\right)
\end{equation}
can behave one of three ways:
\begin{enumerate}
\item (Invariant pairs) The pairing  $\langle r,\, m\rangle_x$ does not depend on the choice of $x$, and \eqref{eq:pair_classifier} has magnitude $1$.
\item (Ignorable pairs) The pairing depends on $x$, and \eqref{eq:pair_classifier} can be proved to be negligible using analytic number theoretic techniques.
\item (Uncontrolled pairs) The pairing depends on $x$, and \eqref{eq:pair_classifier} cannot be controlled.
\end{enumerate}

Our approach to estimating \eqref{eq:moment_to_pairs} is to approximate this sum over all invariant pairs. The sum over invariant pairs turns out to be substantially larger than $|X|$ times the total number of uncontrolled pairs, and the ignorable pairs also contribute negligibly by definition, so we find that this simple approximation suffices for our applications.

Our first goal is to be able to distinguish these types of pairs. The following proposition is essential.

\begin{notat}
For $\sigma$ in $G_{F\left(\mu_{|\FFF|}\right)}$, note that the map $\tau \mapsto \tau^{-1}$ defines an involution on the collection of double cosets $\langle \sigma G_K\rangle \backslash G_F / \langle \sigma G_K\rangle$, and hence also an involution on the set of representatives $B(\sigma, \sigma)$. We take $B_{\text{fixed}}(\sigma)$ to be the subset of $B(\sigma, \sigma)$ fixed by this involution, and we take $B_{1/2}(\sigma)$ to be a set of representatives from the collection of orbits in $B(\sigma, \sigma)\backslash B_{\text{fixed}}(\sigma)$ under the group generated by this involution. We will also fix some total order on the set $S$.
\end{notat}

\begin{prop}
\label{prop:change_p}
With all notation as above, choose $s \in S$, choose $x_{0t} \in X_t$ for all $t \ne s$, and take $L$ to be the line $X_s \times \prod_{t \ne s} \{x_{0t}\}$. Fix $\phi_0$, $(n_s)_s$, and $(r_s)_s$ as above. Then there is some $C_0$ in $\tfrac{1}{\ell}\Z/\Z$ so, for all $x \in L$, we have
\begin{align*}
\big\langle\,(r_s)_s, \,(\phi_0,\, &(n_s)_s)\big\rangle_x  \\
=\, C_0 &- \,\sum_{\substack{s, t \in S \\ s < t}}\, \sum_{\tau \in B(\sigma_s, \sigma_t)}  \textup{lg}\left( \gamma_s \gamma_t \left(\tau r_t - r_s \right)\cdot \left(\tau n_t - n_s\right) \cdot \kappa\left(\symb{\tau^{-1} \pi_s(x)}{\pi_t(x)}'\right)\right)  \\
&-\,\sum_{s\in S} \sum_{\tau \in B_{1/2}(\sigma_s)} \textup{lg}\left( \gamma_s^2 \left(\tau r_s - r_s\right)\cdot \left(\tau n_s - n_s\right) \cdot \kappa\left(\symb{\tau^{-1} \pi_s(x)}{\pi_s(x)}'\right)\right)  \\
&-\, \sum_{s\in S} \sum_{\tau \in B_{\textup{fixed}}(\sigma_s)} \textup{lg}\left( A\,\gamma_s^2\,(\tau r_s - r_s)\cdot (\tau n_s - n_s) \cdot \kappa\left(\symb{\tau^{-1}\pi_s(x)}{\pi_s(x)}'\right)\right),
\end{align*}
where $A = 1/2$ if $\ell$ is odd and is otherwise $0$.
\end{prop}

\begin{proof}
Given $s, t$ in $S$ with $s < t$, and given $\tau \in B(\sigma_s, \sigma_t)$, we combine the $(t, \tau^{-1})$ term of $r_s(a_{s, x}(\phi_0, (n_s)_s))$ with the $(s, \tau)$ term of $r_t(a_{t, x}(\phi_0, (n_s)_s))$ using the fact that \eqref{eq:tau_rec} is determined by classes to get
\[C -  \textup{lg}\left( \gamma_s \gamma_t \left(\tau r_t - r_s \right)\cdot \left(\tau n_t - n_s\right) \cdot \kappa\left(\symb{\tau^{-1} \pi_s(x)}{\pi_t(x)}'\right)\right)\]
for $C$ some constant not depending on $x$. This similarly works for combining the $(s, (s, \tau))$ term with the $(s, (s, \tau^{-1}))$ term for $\tau$ in $B_{1/2}(\sigma_s)$; by halving, it can also handle $\tau$ from $B_{\text{fixed}}(\sigma_s)$ in the case $\ell > 2$.

This just leaves the case of $\ell = 2$, where we may simply apply Proposition \ref{prop:involution_spin} to handle the terms corresponding to $\tau$ in $B_{\text{fixed}}(\sigma_s)$.
\end{proof}

\section{Ignorable pairs}
\label{sec:ignorable}
We keep all notation as in Section \ref{ssec:explicit_local}. In addition, we partition $S$ into three sets $\Ssm$, $\Smed$, and $\Slg$ indexing small, medium, and large sets of primes. In practice, $\Slg$ will be far larger than $\Ssm \cup \Smed$.

Given $s$ and $t$ in $S$, we have two methods of controlling the distribution of the symbol $\symb{\ovp_s}{\ovp_t}$ as $\ovp_s$ and $\ovp_t$ vary through $X_s$ and $X_t$. First, if neither $X_s$ nor $X_t$ is small, we can control this symbol using the bilinear estimate for symbols given in \cite[Section 5]{Smi22a}. On the other hand, if one of these sets is small but the other is large, we can control this symbol using the Chebotarev density theorem.

Using Proposition \ref{prop:change_p}, we can identify when one of these strategies suffices to control the sum \eqref{eq:pair_classifier}. This leads to the following definition.

\begin{defn}
\label{defn:ignorable}
We say that $\big((r_s)_s, \,(\phi_0, (n_s)_s)\big)$ is \emph{bilinearly ignorable} if there are distinct indices in $\Smed \cup \Slg$ and some $\tau_0 \in G_F$ so that
\begin{equation}
\label{eq:bil_ign}
(\tau_0 r_s - r_t) \cdot (\tau_0 n_s - n_t) \ne 0.
\end{equation}
We say that it is \emph{ignorable by Chebotarev} if it is not bilinearly ignorable and there is some $s \in \Slg$, $t \in \Ssm$, and $\tau_0 \in G_F$ so \eqref{eq:bil_ign} holds and for which we also have
\begin{equation}
\label{eq:nochebspin}
(\tau r_s - r_s) \cdot (\tau n_s - n_s) = 0\quad\text{for all } \, \tau \in G_F.
\end{equation}

We call the pair unignorable if it is neither bilinearly ignorable nor ignorable by Chebotarev.
\end{defn}
We note that the condition \eqref{eq:nochebspin} is needed because the Chebotarev density theorem is likely insufficient to control spins; see \cite{KoMi21}. The flexible coefficients allowed by \cite[Proposition 5.9]{Smi22a} make this condition unnecessary for bilinearly ignorable pairs.

\begin{notat}
\label{notat:cJuti}
Take $c_{\textup{bilin}}$ to be the maximum value taken by the expression \eqref{eq:pair_classifier} among pairs that are bilinearly ignorable. If there is no such pair, take $c_{\textup{bilin}} = 0$. Similarly, take $c_{\textup{Cheb}}$ to be the maximum value of this expression among pairs that are ignorable by Chebotarev. If there is no such pair, take $c_{\textup{Cheb}} = 0$.

\end{notat}
\subsection{Characterizing unignorable pairs}
\label{ssec:pariah}
The basic principle of our characterization of unignorable pairs is that, outside of an irrelevant set of $(n_s)_s$, if $(r_s)_s$ is chosen so $(n_s)_s, (r_s)_s$ defines an unignorable pair, there is an alternating homomorphism $\Gamma: Q \to Q^{\vee}$ so
\[\Gamma(\tau_1 n_s - \tau_2 n_t) \equiv  \tau_1 r_s - \tau_2 r_t\, \text{ mod } Q^{\perp}\]
for all $s, t \in S$ and $\tau_1, \tau_2 \in G_F$, where $Q = \bfQ((q_s)_s)$.

However, we still need methods for controlling the irrelevant set of $(n_s)_s$. We start with a definition that will be important to handle the condition \eqref{eq:nochebspin}.
\begin{defn}
\label{defn:pariah}
 Given $(n_s)_s$ as above, we call $s_0 \in \Slg$ a \emph{pariah index} for $(n_s)_s$ if there are not disjoint subsets $S_1$, $S_2$ of $\Slg - \{s_0\}$ and some choice of 
\[(a_{1s})_{s \in S_1} \in \bigoplus_{s \in S_1} \FFF_{\ell} \quad\text{and}\quad (a_{2s})_{s \in S_2} \in \bigoplus_{s \in S_2} \FFF_{\ell} \]
satisfying
\[\sum_{s \in S_1} a_{1s} = \sum_{s \in S_2} a_{2s} = 1\quad\text{and}\quad \sum_{s \in S_1} a_{1s} n_s =  \sum_{s \in S_2} a_{2s} n_s =  n_{s_0}.\]
We denote the set of pariah indices by $\Spar((n_s)_s)$.

We call $s_0$ an \emph{m/l pariah index} if $s_0$ lies in $\Smed \cup \Slg$ and the condition given above is satisfied with $\Slg$ replaced by $\Smed \cup \Slg$. We denote the set of m/l pariahs by $S_{\textup{m, par}}$.
\end{defn}

Our first observation about pariahs indices is that there cannot be that many of them.
\begin{prop}
\label{prop:parbound}
Taking $g$ to be the corank of $N$, and given $(n_s)_s$ as above, we have
\[\# \Spar((n_s)_s) \le 2g + 2\]
\end{prop}
\begin{proof}
This statement is vacuous if $\Slg$ has fewer than $g + 2$ elements. So we assume $\#\Slg \ge g + 2$.

Choose $s_0 \in \Slg$, and choose a sequence $s_1, \dots, s_r \in \Slg$ so that 
\[n_{s_1} -n_{s_0}, \,\,\dots,\,\, n_{s_r} - n_{s_0}\]
gives a basis for the $\FFF_{\ell}$-vector space spanned by all elements of the form $n_s - n_t$ with $s, t \in \Slg$.

Take $\Slg' = \Slg - \{s_0, \dots, s_r\}$. Choose $t_0$ in $\Slg'$, and choose a sequence $t_1, \dots, t_{r'} \in \Slg'$ so that 
\[n_{t_1} - n_{t_0},\,\, \dots,\,\, n_{t_{r'}} - n_{t_0}\]
gives a basis for the space spanned by all elements of the form $n_s - n_t$ with $s, t \in \Slg'$.

We then have $r  \le g$ and $r' \le g$, and we also have
\[\Spar((n_s)_s) \subseteq \{s_0, \dots, s_r\} \cup \{t_0, \dots, t_{r'}\},\]
giving the proposition.
\end{proof}
The same proposition holds for $S_{\textup{m, par}}$ using the same logic.

\begin{defn}
\label{defn:unlawful}
Given $(n_s)_s$ as above and  $S' \subseteq S$, we take 
\[\bfQ(S') = \bfQ(S', \,(n_s)_s)\]
to be the vector subspace of $\bfQ((n_s)_s)$ generated by the elements
\[\tau_1 n_s - \tau_2 n_t \quad\text{for } \, s, t\in S',\,\, \tau_1, \tau_2 \in G_F.\]
We also take
\[S_{\textup{npar}}((n_s)_s) = \Slg - S_{\textup{par}}((n_s)_s)\,\text{ and }\, S_{\textup{m,npar}}((n_s)_s) = \Smed \cup\Slg - S_{\textup{m, par}}((n_s)_s).\]
We call $(n_s)_s$ \emph{borderline} if $\bfQ(S_{\textup{npar}}, (n_s)_s)$ is a proper submodule of $\bfQ(S, (n_s)_s)$.
\end{defn}

\begin{prop}
\label{prop:four_qrs}
In the above situation, suppose the pair $(n_s)_s, (r_s)_s$ is not bilinearly ignorable. Then the equation
\[\Gamma(\tau_1 n_s - \tau_2 n_t) \equiv \tau_1 r_s - \tau_2r_t \quad\text{for all }\, s, t \in S', \,\,\tau_1, \tau_2 \in G_F\]
defines $G_F$-equivariant homomorphisms
\begin{alignat*}{5}
&\Gamma_{\textup{m,npar}}&&: \bfQ(S') \rightarrow N[\omega]^{\vee}/ \bfQ(\Slg \cup \Smed)^{\perp}\quad&&\text{with } S' = S_{\textup{m,npar}}\\
&\Gamma_{\textup{m, l}}&&: \bfQ(S') \rightarrow N[\omega]^{\vee}/ \bfQ(S_{\textup{m,npar}})^{\perp}\quad&&\text{with } S' = \Slg \cup \Smed,
\end{alignat*}
and we have an identity $\Gamma_{\textup{m, npar}}(m) \cdot m = 0$ for all $m$ in $\bfQ(S_{\textup{m, npar}})$.

If the pair is unignorable, the same equation defines $G_F$-equivariant homomorphisms
\begin{alignat*}{5}
&\Gamma_{\textup{npar}}&&: \bfQ(S') \rightarrow N[\omega]^{\vee}/ \bfQ(S)^{\perp}\quad&&\text{with } S' = S_{\textup{npar}}\\
&\Gamma&&: \bfQ(S') \rightarrow N[\omega]^{\vee}/ \bfQ(S_{\textup{npar}})^{\perp}\quad&&\text{with } S' = S,
\end{alignat*}
and we have an identity $\Gamma_{\textup{npar}}(m) \cdot m = 0$ for all $m$ in $\bfQ(S_{\textup{npar}})$.
\end{prop}
\begin{rmk}
If $(n_s)_s$ is not borderline, these four maps are all equal.
\end{rmk}

\begin{proof}
Assume to start that $\big((n_s)_s, (r_s)_s)$ is not bilinearly ignorable.

Given
\[a: \Smed \cup \Slg \rightarrow \FFF_{\ell}\quad\text{satisfying}\quad \sum_{s \in \Smed \cup \Slg} a_s = 0\]
and any function $\tau: S \rightarrow G_F$, we have
\begin{equation}
\label{eq:nonign_A}
\left( \sum_{s \in \Smed \cup \Slg}a_s \tau_s r_s \right) \cdot \left( \sum_{s \in \Smed \cup \Slg}a_s \tau_s n_s \right) = 0,
\end{equation}
as this product can be rewritten in the form
\[-\sum_{\{s, t\} \subseteq \Smed \cup \Slg} a_s a_t \left(\tau_s r_s - \tau_t r_t\right) \cdot (\tau_s n_s - \tau_t n_t).\]

If $s_0 \in \Slg \cup \Smed$ is not an m/l pariah index, we can choose $S_1$, $S_2$, $a_1$, and $a_2$ as in Definition \ref{defn:pariah} so that we have
\[n_{s_0} = \sum_{s \in S_1} a_{1s}n_s =\sum_{s \in S_2} a_{2s}n_s.\]
Choosing $\tau \in G_F$, we can show
\begin{align*}
&\left(r_{s_0} - \sum_{s \in S_1} a_{1s}\tau r_s\right) \cdot \left(n_{s_0} - \sum_{s \in S_1} a_{1s}\tau n_s\right) = 0,\\
& \left( \sum_{s \in S_1} a_{1s}\tau r_s -  \sum_{s \in S_2} a_{2s} r_s \right) \cdot \left( \sum_{s \in S_2} a_{2s} n_s  - \sum_{s \in S_1} a_{1s}\tau n_s \right) = 0,\\
& \left(-\tau r_{s_0} + \sum_{s \in S_2} a_{2s}r_s \right) \cdot \left(-\tau n_{s_0} + \sum_{s \in S_2} a_{2s} n_s \right)  = 0
\end{align*}
by applying \eqref{eq:nonign_A} three times. But these sum to give
\begin{equation}
\label{eq:nonign_B}
(r_{s_0} - \tau r_{s_0}) \cdot (n_{s_0} - \tau n_{s_0}) = 0\quad\text{ for }\, s_0 \in S_{\textup{m,npar}},\,\, \tau \in G_F.
\end{equation}

Take $\text{val}: \FFF_{\ell}[G_F] \rightarrow \FFF_{\ell}$ to be the homomorphism sending $[\sigma]$ to $1$ for all $\sigma$ in $G_F$. For $S'$ a subset of $S$, define
\[I(S') = \left\{ (\alpha_s)_s \in \bigoplus_{s \in S'} \FFF_{\ell}[G_F]\,:\,\, \sum_s \text{val}(\alpha_s) = 0\right\}.\]

Repeating the argument used to prove \eqref{eq:nonign_A} and applying \eqref{eq:nonign_B} as necessary, we can show that, for all $(\alpha_s)_s \in I(S_{\textup{m,npar}})$, all $s_1 \in \Smed \cup \Slg$, all $s_0 \in S_{\text{m,npar}}$, and any $\tau \in G_F$, we have
\begin{equation}
\label{eq:Gmnpar_alt}
\left(\tau r_{s_1} -  r_{s_0}  + \sum_{s \in S_{\textup{m,npar}}} \alpha_s r_s\right)\cdot \left(\tau n_{s_1} - n_{s_0} + \sum_{s \in S_{\textup{m,npar}}} \alpha_s n_s\right) = 0.
\end{equation}
By taking the difference with the case given by $s_1 = s_0$, $\tau = 1$, we get
\[(\tau r_{s_1} - r_{s_0}) \cdot \left(\sum_{s \in S_{\textup{m,npar}}} \alpha_s n_s\right) = -  \left(\sum_{s \in S_{\textup{m,npar}}} \alpha_s r_s\right) \cdot (\tau n_{s_1} - n_{s_0}).\]
Except in the case where $S_{\textup{m,npar}}$ is empty, we can sum some number of identities of this form together to show
\begin{equation}
\label{eq:nonign_C}
 \left(\sum_{s \in \Smed \cup \Slg} \alpha'_s r_s\right) \cdot \left(\sum_{s \in S_{\textup{m,npar}}} \alpha_s n_s\right) = -  \left(\sum_{s \in S_{\textup{m,npar}}} \alpha_s r_s\right) \cdot  \left(\sum_{s \in \Smed \cup \Slg} \alpha'_s n_s\right)
\end{equation}
for all $(\alpha_s)_s \in I(S_{\textup{m,npar}})$ and $(\alpha'_s)_s \in I(\Smed \cup \Slg)$. If $S_{\textup{m,npar}}$ is empty, \eqref{eq:nonign_C} is vacuously true. The existence and equivariance of $\Gamma_{\text{m,npar}}$ and $\Gamma_{\text{m,l}}$ follow from this equation, and the alternating property for $\Gamma_{\text{m,npar}}$ follows from \eqref{eq:Gmnpar_alt}.

If $\big((n_s)_s, (r_s)_s\big)$ is also not ignorable by Chebotarev, we can conclude from \eqref{eq:nonign_B} that
\[(\tau r_{s_1} - r_{s_0}) \cdot (\tau n_{s_1} - n_{s_0}) = 0 \quad\text{for all }\, s_1 \in S, \,s_0 \in S_{\text{npar}}, \,\tau \in G_F.\]
Following the same argument as before, we find
\begin{equation}
\label{eq:nonign_C2}
 \left(\sum_{s \in S} \alpha'_s r_s\right) \cdot \left(\sum_{s \in S_{\textup{npar}}} \alpha_s n_s\right) = -  \left(\sum_{s \in S_{\textup{npar}}} \alpha_s r_s\right) \cdot  \left(\sum_{s \in S} \alpha'_s n_s\right)
\end{equation}
for all $(\alpha_s)_s \in I(S_{\textup{npar}})$ and $(\alpha'_s)_s \in I(S)$. The existence and equivariance of $\Gamma_{\text{npar}}$ and $\Gamma$  follow from this equation, as does the alternating property of the former map.
\end{proof}

\subsection{Bounds for borderline $(n_s)_s$}
For our first result, we fix $(n_s)_s$ and give upper bounds for the number of unignorable pairs that contain it.
\begin{prop}
\label{prop:borderline_fewpair}
With all notation as above, there is a $C > 0$ determined just from $(K/F, \Vplac_0, \FFF)$ so we have the following:

Choose $(n_s)_s$ as above, and take $\bfQ(S') = \bfQ(S', (n_s)_s)$ as in Definition \ref{defn:unlawful} for any subset $S'$ of $S$. Then the number of $(r_s)_s$ in $\mathscr{R}$ for which $\big((n_s)_s,\, (r_s)_s\big)$ is not bilinearly ignorable is bounded by
\[e^{Cg^2}\cdot \prod_{s \in \Smed \cup \Slg} \# (N[\omega]/\bfQ(\Smed \cup \Slg))^{\sigma_s} \cdot \prod_{s \in \Ssm} \# N[\omega]^{\sigma_s},\]
and the number of such $(r_s)_s$ for which this pair is unignorable is bounded by
\[e^{Cg^2}\cdot \prod_{s \in  \Slg} \# (N[\omega]/\bfQ(S))^{\sigma_s}  \cdot \prod_{s \in \Smed} \#(N[\omega]/\bfQ(\Smed \cup \Slg))^{\sigma_s} \cdot \prod_{s \in \Ssm} \# (N[\omega]/\bfQ(S_{\textup{npar}}))^{\sigma_s},\]
where $g$ is the corank of $N$.
\end{prop}

\begin{proof}
We first note that we have an identity
\[\# Q^{\sigma_0} = \# (Q^{\vee})^{\sigma_0}\]
for any $\sigma_0$ in $G_{F\left(\mu_{\ell^{k_0}}\right)}$ and any $G_F$ submodule $Q$ of $N[\omega]$.

To bound the number of $(r_s)_s$ for which the pair is not bilinearly ignorable, we use the map $\Gamma_{\textup{m,npar}}$ defined in Proposition \ref{prop:four_qrs}. Having fixed this map and the value of $r_{s_0}$ for some $s_0 \in \Smed \cup \Slg$, there are at most $\left(\bfQ(\Smed \cup \Slg)^{\perp}\right)^{\sigma_s}$ consistent choices of $r_s$ for each $s$ in $S_{\textup{m, npar}}$. The choice of $\Gamma_{\textup{m,npar}}$, the value of $r_s$ at the m/l pariahs, and the value of $r_{s_0}$ can then be absorbed by the $e^{Cg^2}$ term. Accouting for the possible values of $r_s$ for $s \in \Ssm$ gives the first estimate.

For the estimate of unignorable pairs, we use the maps $\Gamma$, $\Gamma_{\textup{m,npar}}$, and $\Gamma_{\textup{npar}}$ of Proposition \ref{prop:four_qrs} to constrain the $r_s$ for $s$ in $\Ssm$, $\Smed$, and $\Slg$ respectively. With this setup, the proof follows as before.
\end{proof}

The following bound will be convenient in our bound for the number of borderline $(n_s)_s$.
\begin{defn}
We define the set of \emph{jury indices} $S_{\textup{jury}}$ to be the set of $s \in \Slg$ for which $\sigma_s$ is trivial in $\Gal(K(\Vplac_0)/F)$.
\end{defn}
\begin{prop}
\label{prop:borderline_few}
With all notation as above, there is a $C > 0$ determined just from $(K/F, \Vplac_0, \FFF)$ so we have the following:

Choose $G_F$-submodules $Q_1 \subseteq Q_0$ of $N[\omega]$. Take
\[k = \left( \#S_{\textup{jury}} - \# \Ssm \cup \Smed \right) \cdot \dim Q_0/Q_1.\]

Then the size of the set
\[\big\{ (\phi_0,\, (n_s)_s )\in \mathscr{M}_1(Q_0) \,:\,\, \bfQ( S_{\textup{npar}}((n_s)_s), (n_s)_s) \subseteq Q_1\, \text{ and }\, \bfQ(  (n_s)_s) = Q_0\big\}\]
has upper bound
\begin{equation}
\label{eq:borderline_few}
\ell^{-k} \cdot e^{Cg^2} \cdot (1 + \# \Slg)^{Cg}  \cdot \prod_{s \in S}  {\min}_{\,Q\,} \left(\frac{\#(N/Q)[\omega]^{\sigma_s}}{ \#(N[\omega]/Q_0)^{\sigma_s}}\right),
\end{equation}
where each minimum is taken over $G_F$-modules $Q$ satisfying $Q_1 \subseteq Q \subseteq Q_0$.
\end{prop}
\begin{proof}
From Proposition \ref{prop:parbound}, we see that we can bound the choices of pariah indices, together with the value of $n_s$ at the pariah indices, by $e^{Cg^2} (1 + \# \Slg)^{Cg}$. From \eqref{eq:reasonable_Q}, we can then bound the number of tuples $(n_s)_{s \in \Slg}$ completeable to some $(\phi_0, (n_s)_s)$ in the set  by
\[e^{Cg^2} (1 + \# \Slg)^{Cg} \cdot \prod_{s \in \Slg} \# \left(Q_1^{\sigma_s} \cap \text{im}\big((N/Q_0)[\omega]^{\sigma_s} \to Q_0^{\sigma_s}\big)\right).\]

Meanwhile, given $(n_s)_{s \in \Slg}$, the number of completions to $(\phi_0, (n_s)_s)$ in $\mathscr{M}_1(Q_0)$ is bounded by 
\[e^{Cg}\cdot \prod_{s \in \Ssm \cup \Smed}\#\text{im}\big((N/Q_0)[\omega]^{\sigma_s} \to Q_0^{\sigma_s}\big).\]
Takinng $\tfrac{1}{\omega}Q_1$ to be the set of $x$ in $N[\omega^2]$ for which $\omega x$ lies in $Q_1$, we have an identity
\[Q_1^{\sigma_s} \cap \text{im}\big((N/Q_0)[\omega]^{\sigma_s} \to Q_0^{\sigma_s}\big) \,=\, \text{im}\big(\big(\tfrac{1}{\omega}Q_1/ Q_0\big)^{\sigma_s} \to Q_{1}^{\sigma_s}\big),\]
where the image on the right comes from the second map in the exact sequence
\[0 \to N[\omega]/Q_0 \to \tfrac{1}{\omega} Q_1/Q_0 \to Q_1 \to 0.\]
So we have
\[\# \left(Q_1^{\sigma_s} \cap \text{im}\big((N/Q_0)[\omega]^{\sigma_s} \to Q_0^{\sigma_s}\big)\right) \,=\, \frac{\#\left(\tfrac{1}{\omega} Q_1 / Q_0 \right)^{\sigma_s}}{\#\left(N[\omega] / Q_0 \right)^{\sigma_s}}.\]
Given $Q$ satisfying $Q_1 \subseteq Q \subseteq Q_0$, we have the chain
\[\# \left(\tfrac{1}{\omega} Q_1 / Q_0 \right)^{\sigma_s} \le \# \left(\tfrac{1}{\omega} Q/ Q_0 \right)^{\sigma_s} \,=\, \#\left(\tfrac{1}{\omega} Q/ Q_0 \right)_{\sigma_s} \le \#\left(N/Q\right)[\omega]_{\sigma_s},\]
with the first and last inequalities following from the existence of a natural inclusion and projection, respectively. For $s \in S_{\text{jury}}$, we can make use of the exact result
\[\#  \left(\tfrac{1}{\omega} Q_1 / Q_0 \right) \,=\, \ell^{-\dim Q_0/Q_1}\cdot \# \left(N/Q\right)[\omega].\]
Finally, for $s \in \Ssm \cup \Smed$, we have
\[ \# \left(N/Q\right)[\omega]^{\sigma_s}\, \ge\,  \ell^{-\dim Q_0/Q_1} \cdot \# \left(N/ Q_0 \right)^{\sigma_s} ,\]
as follows from an application of the long exact sequence. Combining these bounds gives the proposition.
\end{proof}
We can now handle the borderline $(n_s)_s$ by combining Propositions \ref{prop:borderline_fewpair} and \ref{prop:borderline_few}.
\begin{prop}
\label{prop:borderline_rare}
With all notation as above, there is a $C > 0$ determined just from $(K/F, \Vplac_0, \FFF)$ so we have the following:

Choose distinct $G_F$-submodules $Q_0$ and $Q_1$ of $N[\omega]$ satisfying $Q_1 \subset Q_0$. Take $\mathscr{M}'$ to be the set of $(\phi_0, (n_s)_s)$ in $\mathscr{M}$ satisfying
\[\bfQ((n_s)_s) = Q_0\quad\text{and}\quad \bfQ\big(S_{\textup{npar}}((n_s)_s), \,(n_s)_s\big) = Q_1.\]\
Take
\[b =  \# S_{\textup{jury}} - 2\cdot \# \Ssm \cup \Smed.\]
We assume $b > 0$. Then, for any $G_F$-module $Q$ satisfying $Q_1 \subseteq Q \subseteq Q_0$, we have
\begin{align*}
&\frac{1}{\# X}\sum_{x \in X} \# \big\{m \in \mathscr{M}'\,:\,\, \Psi_x(m) \in \Sel^{\omega}N^{\chi(x)}\big\} \\
&\qquad\,\le\, \ell^{-b} \cdot e^{Cg^2} (1 + \# \Slg)^{Cg} \cdot  \mathcal{T}_{N, Q}(\chi(x_0)) \cdot \,+\, (c_{\textup{bilin}} + c_{\textup{Cheb}}) \cdot e^{C g|S|},
\end{align*} 
where $x_0$ is any point in $X$.
\end{prop}
\begin{proof}
We can rewrite the average we are bounding by
\[\frac{1}{\# X \cdot \# \mathscr{R}} \sum_{m \in \mathscr{M}'} \sum_{r \in \mathscr{R}} \sum_{x \in X} \exp\left(2\pi i \big\langle r,\, m\big\rangle_x\right).\]
We may restrict this sum to $m$ that are also in $\mathscr{M}_1(Q_0)$. Using Proposition \ref{prop:borderline_few}, we can then bound this average by
\[\frac{A}{\# X \cdot \# \mathscr{R}}  \max_{m \in \mathscr{M}'} \left(\sum_{r \in \mathscr{R}}  \sum_{x \in X}  \exp\left(2\pi i \big\langle r,\, m\big\rangle_x\right)\right),\]
where $A$ equals \eqref{eq:borderline_few}. The inner sum can be bounded using Proposition \ref{prop:borderline_fewpair}, giving
\begin{align*}
&\max_{m \in \mathscr{M}'} \left(\frac{1}{\#X} \sum_{r \in \mathscr{R}}  \sum_{x \in X}  \exp\left(2\pi i \big\langle r,\, m\big\rangle_x\right)\right) \\
&\qquad  \le\,  (c_{\textup{bilin}} + c_{\textup{Cheb}}) \cdot \# \mathscr{R} \, +\, e^{Cg^2}\cdot  \ell^{\dim Q_0/Q_1 \cdot \# \Ssm \cup \Smed } \cdot \prod_{s \in S} \#(N[\omega]/Q_0)^{\sigma_s}
\end{align*}
for some $C > 0$ only depending on $(K/F, \Vplac_0, \FFF)$. The result follows.
\end{proof}

\subsection{Rough bounds for average Selmer group sizes}
Having dealt with borderline $(n_s)_s$, we need upper bounds on the average number of non-borderline $(n_s)_s$ that give Selmer elements. Our main result is the following.
\begin{prop}
\label{prop:rough_grid_average}
With all notation as above, there is a $C > 0$ determined just from $(K/F, \Vplac_0, \FFF)$ so we have the following:

Choose a $G_F$-submodule $Q$ of $N[\omega]$, and take $\mathscr{M'}$ to be the set of non-borderline $(\phi_0, (n_s)_s) \in \mathscr{M}$ satisfying $Q = \bfQ((n_s)_s)$. Then
\begin{align*}
&\frac{1}{\# X}\sum_{x \in X} \# \big\{m \in \mathscr{M}'\,:\,\, \Psi_x(m) \in \Sel^{\omega}N^{\chi(x)}\big\}\\
&\qquad\quad\le \, e^{Cg^2} \cdot\mathcal{T}_{N, Q}(\chi(x_0))  + (c_{\textup{bilin}} + c_{\textup{Cheb}}) \cdot e^{Cg|S|},
\end{align*}
where $x_0$ is any point in $X$.
\end{prop}
\begin{proof}
We may restrict the sum to $m$ lying in $\mathscr{M}_1(Q)$. The number of such $m$ can be bounded by an expression of the form
\[e^{Cg} \cdot \prod_{s \in S} \frac{\# (N/Q)[\omega]^{\sigma_s}}{\# (N[\omega]/Q)^{\sigma_s}}.\]
The result then follows from Proposition \ref{prop:borderline_fewpair}.
\end{proof}

For some grids, we will have access to bilinear results but not to the Chebotarev density theorem. For these, we will need the following rough bound on the average size of Selmer groups in the grid.
\begin{prop}
\label{prop:rough_grid_average_nocheb}
With all notation as above, there is a $C > 0$ determined just from $(K/F, \Vplac_0, \FFF)$ so we have
\begin{align*}
&\frac{1}{\# X}\sum_{x \in X} \# \big\{m \in \mathscr{M}\,:\,\, \Psi_x(m) \in \Sel^{\omega}N^{\chi(x)}\big\} \\
&\qquad\quad \le e^{Cg^2} \cdot \ell^{g\cdot \# \Ssm} \cdot  \max_{Q \subseteq N[\omega]} \mathcal{T}_{N, Q}(\chi(x_0)) + c_{\textup{bilin}} \cdot e^{Cg|S|},
\end{align*}
where $x_0$ is any point in $X$.
\end{prop}
\begin{proof}
If $\Smed \cup \Slg$ is empty, the result is immediate. So choose $s_0$ in $\Smed \cup \Slg$, and take $\msM'$ to be the subgroup of $\msM$ of elements that are $0$ at each $s \in \Ssm \cup \{s_0\}$. Combining the first part of Proposition \ref{prop:borderline_fewpair} with the calculation for the size of $\msM' \cap \msM_1(Q)$ for a given $Q \subseteq N[\omega]$  gives
\begin{align*}
&\frac{1}{\# X}\sum_{x \in X} \# \big\{m \in \mathscr{M}' \cap \msM_1(Q)\,:\,\, \Psi_x(m) \in \Sel^{\omega}N^{\chi(x)}\big\} \\
&\qquad\quad \le e^{C_0g^2} \cdot \mathcal{T}_{N, Q}(\chi(x_0)) + c_{\textup{bilin}} \cdot e^{C_0g|S|},
\end{align*}
for some $C_0$ depending on $(K/F, \Vplac_0, \FFF)$. These sum to give the proposition.
\end{proof}

\section{Moments of sizes of $\omega$-Selmer groups in grids}
\label{sec:moments}

Propositions \ref{prop:borderline_rare} and \ref{prop:rough_grid_average} give us rough upper bounds on the average size of $\omega$-Selmer groups in a grid of twists. In this section, we give conditions under which we can estimate these averages. Throughout, we will take $G_1 = \Gal(K(\Vplac_0)/F(\mu_{\ell^{k_0}}))$.
\subsection{Invariant pairs}
Our first goal will be to show that certain collections of invariant pairs do not contribute meaningfully to the sum \eqref{eq:moment_to_pairs}. The reason these pairs will not contribute involves the appearance of the connecting map $\delta_{\sigma_s}$ in the expression \eqref{eq:ps_cond}. Isolating the contribution of this map is relatively subtle, and our process of isolating it starts with the following variant of the group $\mathscr{M}_1(Q)$ constructed in Definition \ref{defn:ram_sub}.

\begin{notat}
For $\sigma \in G_1$, take $Q(\sigma)$ to be the subset of $n$ in $Q^{\sigma}$ satisfying
\[\delta_{\sigma} n \equiv 0 \text{ mod } Q + (\sigma -1)N[\omega]\]
Take 
\[\mathscr{M}_{10}(Q) = \{(0, (n_s)_s) \in \mathscr{M} \,:\,\, n_s \in Q(\sigma_s)\text{ for all } s \in S\}.\]
\end{notat}

With this set, choose some $G_F$-equivariant homomorphism 
\[\Gamma: Q \to Q^{\vee}.\]
We assume that this map is alternating, so $\Gamma(m) \cdot  m = 0$ for every $m$ in $Q$. We also assume that, for all $\sigma \in G_1$ and $n \in Q(\sigma)$, there is an $r \in (N[\omega]^{\vee})^\sigma$ projecting to $\Gamma(n)$ under the dual to the inclusion map $Q \hookrightarrow N[\omega]$.

From this information, we construct a quadratic form $f_{\Gamma}: \mathscr{M}_{10}(Q) \to \tfrac{1}{\ell}\Z/Z$ using the expression
\[f_{\Gamma}(m) = \langle r, m \rangle_x,\]
 where $r$ is chosen so $\pi_s(r)$ projects to $\Gamma(\pi_s(m))$ for all $s \in S$, and where $x$ is chosen from $x \in X$. We can check that the expression does not depend on the choice of $r$ or $x$.

We recall the notion of commuting with connecting maps given in \eqref{eq:comm_conn}. We need a variant of this definition to handle certain kinds of invariant pairs. This definition also uses the dual connecting map appearing in Definition \ref{defn:dual_delta}.
\begin{defn}
\label{defn:cancellable}
Choose $\Gamma: Q \to Q^{\vee}$ obeying the assumptions above. Suppose there is a choice of $\sigma \in G_1$, elements $n_0, n_1$ in $Q(\sigma)$, and elements $r_0, r_1$ in $(N[\omega]^{\vee})^{\sigma}$ such that $r_0$ projects to $\Gamma(n_0)$, such that $r_1$ projects to $\Gamma(n_1)$, and such that
\begin{equation}
\label{eq:cancellable_condition}
\delta_{\sigma}(n_0) \cdot r_1 \ne - n_1 \cdot \delta_{\sigma}^{\vee}(r_0),
\end{equation}
where $\delta_{\sigma}^{\vee}$ is the map appearing in Definition \ref{defn:dual_delta}, where the left product is defined using the identification
\[H^1(\langle \sigma \rangle, N[\omega]) \cong (N[\omega])_{\sigma}\]
and the evaluation pairing $N[\omega]_{\sigma} \times (N[\omega]^{\vee})^{\sigma} \to \mu_{\ell}$, and where the right product is defined using the dual construction. Then we call the map $\Gamma$ \emph{cancellable}.
\end{defn}

In the case that $\Gamma$ is cancellable, we will find that the sum 
\[\sum_{(n_s)_s \in \mathscr{M}_{10}(Q)} \exp(2\pi i f_{\Gamma}((n_s)_s))\]
tends to have small magnitude. We can prove this by showing that the associated bilinear form
\[B_{\Gamma}: \mathscr{M}_{10}(Q) \times \mathscr{M}_{10}(Q) \to \tfrac{1}{\ell}\Z/\Z\]
given by
\[B_{\Gamma}\big(m_1, \,m_2\big)\, =\, f_{\Gamma}(m_1 + m_2) - f_{\Gamma}(m_1) - f_{\Gamma}(m_2)\]
has large rank.

\begin{prop}
\label{prop:cancellable}
Take $\Gamma$ as above, and suppose that $\Gamma$ is cancellable. Choose $\sigma \in G_1$ so we can find $n_0$, $n_1$, $r_0$, and $r_1$ satisfying the condition of Definition \ref{defn:cancellable}, and take $S_0$ to be the subset of $S$ of $s$ so that $\sigma_s$ maps to $\sigma$. Then $B_{\Gamma}$ has rank at least $|S_0| - 1$.
\end{prop}
\begin{proof}
Throughout this proof, we will identify $\mu_{\ell}$ and $\tfrac{1}{\ell}\Z/\Z$ using $\text{lg}$, and we will often write the group operation for $\mu_{\ell^2}$ additively.

Given $k \ge 1$, take $\iota_k$ to be the standard isomorphism of $N^{\vee}[\omega^k]$ and $N[\omega^k]^{\vee}$. We have a commutative diagram
\[
\begin{tikzcd}
0 \arrow{r} & N^{\vee}[\omega] \arrow{r} \arrow{d}{-\iota_1} & N^{\vee}[\omega^2] \arrow{r}{\omega} \arrow{d}{\iota_2} & N^{\vee}[\omega] \arrow{r} \arrow{d}{\iota_1} & 0 \\
0 \arrow{r} &N[\omega]^{\vee} \arrow{r}{\omega^{\vee}}& N[\omega^2]^{\vee} \arrow{r} & N[\omega]^{\vee} \arrow{r} & 0
\end{tikzcd}\]
We may explicitly calculate the connecting map $(N[\omega]^{\vee})^{\sigma} \to (N[\omega]^{\vee})_{\sigma}$ corresponding to the lower sequence here. This connecting map takes $\phi$ in $(N[\omega]^{\vee})^{\sigma}$ to the class of the map
\[m \mapsto (\sigma - 1)\left(\phi'(\omega^{-1} m)\right) - \phi(\delta_{\sigma} m),\]
where $\phi'$ is an element of $(N[\omega^2])^{\vee}$ whose restriction to $(N[\omega])^{\vee}$ equals $\phi$.

Taking
\[A = \begin{cases} \tfrac{1}{2}\left(1 - \tfrac{\sigma(\sqrt{-1})}{\sqrt{-1}}\right) &\text{ if } \#\FFF = 2 \\ 0 & \text{otherwise,}\end{cases}\]
we see that the image of $\phi$ under this connecting map may be given as $A \cdot \phi - \phi \circ \delta_{\sigma}$. So
\[\iota_1 \left(\delta_{\sigma}^{\vee} \iota_1^{-1}(\phi)\right) = A \cdot \phi + \phi \circ \delta_{\sigma}\quad\text{in} \left(N[\omega]^{\sigma}\right)^{\vee} \cong \left(N[\omega]^{\vee}\right)_{\sigma}.\]
Given distinct $s, t \in S_0$, we also have
\begin{equation}
\label{eq:hi_negative_one}
\kappa(\symb{\pi_s(x)}{\pi_t(x)}') \cdot \phi -  \kappa(\symb{\pi_t(x)}{\pi_s(x)}') \cdot \phi = A \cdot \phi,
\end{equation}
as can be checked using Hilbert reciprocity and the general identity $(a, -a) = 1$ of Hilbert symbols. With this calculated, we return to proving the lower bound for the rank of $B_{\Gamma}$.

Choose $n_0, n_1, r_0, r_1$ so the condition of Definition \ref{defn:cancellable} is satisfied.  For $k = 0, 1$ and $s \in S_0$, take $E_k(s)$ to be the element of $\mathscr{M}_1(Q)$ with $s^{th}$-coordinate $n_k/\gamma_s$ and all other coordinates $0$, and take $\Gamma E_k(s)$ to be the element of $\mathscr{R}$ with $s^{th}$-coordinate $r_k/\gamma_s$ and all other coordinates $0$.

 Fix some $x$ in $X$. For any $s, t \in S$, we have
\[B_{\Gamma}(E_0(s), \, E_1(t)) = \left\langle \Gamma E_0(s), E_1(t)\right\rangle_x  + \left\langle \Gamma E_0(t), E_1(s)\right\rangle_x .\]
Suppose $s$ and $t$ are distinct, and take $\ovp = \pi_s(x)$ and $\ovq = \pi_t(x)$. We start by noting that
\[\langle \Gamma E_0(s), E_1(t)\rangle_x = \sum_{\tau \in B(\sigma, \sigma)} \Gamma(n_0) \cdot \tau^{-1}n_1\cdot \kappa\left(\symb{\tau\ovp}{\ovq}'\right).\]

Given $\tau_0, \tau_1 \in G_F$ representing the same double coset not containing the identity, we have that
\[\symb{\tau_0 \ovp}{\ovq}' - \tau_1\left(\symb{\tau_1^{-1} \ovq}{\ovp}'\right) = \symb{\tau_0 \ovp}{\ovp}' - \tau_1\left(\symb{\tau_1^{-1} \ovp}{\ovp}'\right).\]
Applying this and \eqref{eq:hi_negative_one}, we can calculate that $B_{\Gamma}(E_0(s), \, E_1(t))$ equals
\[A\Gamma(n_1) \cdot n_0 + \sum_{\tau \in B(\sigma, \sigma)}\left(\Gamma(n_0) \cdot \tau^{-1} n_1 + \Gamma(n_0) \cdot \tau^{-1}n_1 \right)  \cdot  \kappa\left(\symb{\tau \ovp}{\ovp}'\right).\]
This sum does not depend on the choice of $s$ or $t$.

We also can calculate
\begin{align*}
\langle \Gamma E_0(s), E_1(s) \rangle_x &= C_{s, x} \Gamma(n_0) \cdot n_1 -  \gamma_s^{-1} r_0 \cdot \delta_{\sigma}(n_1) \\
& + \sum_{\tau \in B(\sigma, \sigma)} \Gamma(n_0) \cdot \tau^{-1}(n_1) \kappa\left(\symb{\tau\ovp}{\ovp}'\right),
\end{align*}
where $C_{s, x}$ depends on $s$ and $x$ but not $k$ and $j$.
From the fact that $\Gamma$ is alternating, we then have
\begin{align*}
&B_{\Gamma}(E_0(s), \, E_1(s)) \, -\, B_{\Gamma}(E_0(s), \, E_1(t))\\
&\qquad \qquad=  -\gamma_s^{-1} \cdot \big( r_0 \cdot \delta_{\sigma}(n_1) +r_1 \cdot \delta_{\sigma}(n_1)\big) \,+\, A \Gamma(n_1) \cdot n_0 .
\end{align*}
We see that this is nonzero. In particular, there is $C$ in $\tfrac{1}{\ell}\Z/\Z$ so
\[B_{\Gamma}\big( E_0(s), \, \,E_1(t)\big) = \begin{cases} C &\text{ if } s \ne t \\ \ne C &\text{ if } s = t \end{cases}\]
for all $s, t \in S_0$. This is enough to give the proposition.
\end{proof}

In particular, suppose we have taken $\Gamma$ and $S_0$ as in this proposition. Given a vector subspace $V$ of $\mathscr{M}_{10}(Q)$ of codimension $r \le\tfrac{1}{2}|S_0|- \tfrac{1}{2}$, and given a homomorphism $\psi: V \to \mu_{\ell}$, we have
\begin{equation}
\label{eq:direct_Gauss}
\left|\sum_{v \in V} \psi(v) \exp(2\pi i f_{\Gamma}(v))\right| \le \ell^{-b/2} \cdot |V|
\end{equation}
where $b = |S_0| - 1 - 2r$. To see this, take $V_0$ to be a $b$-dimensional subspace of $V$ where $B_{\Gamma}$ is nondegenerate. A standard argument for Gauss sums \cite[Section 3.4]{Iwan04} gives
\begin{align*}
\left|\sum_{v \in V_0} \psi(v) \exp(2\pi i f_{\Gamma}(v)) \right|^2& = \sum_{v, w \in V_0} \psi(v - w) \exp(2\pi i (f_{\Gamma}(v)- f_{\Gamma}(w)))\\
& = \sum_{v, w\in V_0} \psi(v-w) \exp(2\pi i (f_{\Gamma}(v - w) + B_{\Gamma}(v - w, w)) ) = |V_0|,
\end{align*}
which is enough to prove \eqref{eq:direct_Gauss}.

As a converse to Proposition \ref{prop:cancellable}, we could prove an upper bound on the rank of $B_{\Gamma}$  depending on $N$ and $(K/F, \Vplac_0, \FFF)$ for non-cancellable $\Gamma$. Instead of this general converse, our main results require a stronger result under more stringent conditions. 
\begin{prop}
\label{prop:theta}
Suppose that $N$ has alternating structure $\nu: N \to N^{\vee}$ and  that $\#\FFF = 2$. Then, given $m = (\phi_0, (n_s)_s)$ in $\mathscr{M}_0$ and $x$ in $X$, we have
\[\left\langle m, (\nu(n_s))_{s \in S}\right\rangle_x = 0.\]
\end{prop}
\begin{proof}
Given a closed subgroup $H$ of $G_F$, take
\[q_H: H^1(H, N[2]) \to H^2\left(H, \mu_{4}\right)\]
to be the map defined in \cite[Section 5.3]{MS21} as the connecting map associated to the theta group above $N[2]$ constructed from $N[4]$. We note that
\[\sum_{v \text{ of } F} \inv_v(q_{G_v}(\Psi_x(m))) = 0\]
by Poitou--Tate duality. Since $\Psi_x(m)$ obeys the local conditions for $v \in \Vplac_0$, we find that  $\inv_v(q_{G_v}(\Psi_x(m))) = 0$ at these places, so
\[\sum_{s \in S} \inv_{\pi_s(x) \cap F} \left(q_{G_{\pi_s(x)}}(\Psi_x(m))\right) = 0.\]
Take $\chi = \chi(x)$, and take $\delta_{\chi, \pi_s(x)}: H^0(G_{F, \pi_s(x)}, N[2]) \to H^1(G_{F, \pi_s(x)}, N[2])$ to be the connecting map associated to $N^{\chi}$. For each $s$, we find that 
\[q_{G_{\pi_s(x)}}(\Psi_x(m) - \delta_{\chi, \pi_s(x)}( n_s)) =0\]
because the cocycle class is unramified and valued in $\mu_{4}$. We also have
 \[q_{G_{\pi_s(x)}}(\delta_{\chi, \pi_s(x)}( n_s)) = 0,\]
as follows from  \cite[Example 5.20]{MS21}. Combining these identities and \cite[(5.3)]{MS21} then gives
\[\sum_{s \in S} \inv_{\pi_s(x) \cap F}\left(\left(\Psi_x(m) - \delta_{\chi, \pi_s(x)}( n_s)\right) \cup \nu(\Psi_x(m))\right) = 0.\] This reduces to the claimed identity.
\end{proof}

A similar but simpler result that applies for more general $N$ and $\FFF$ is the following.

\begin{prop}
\label{prop:poitou_0}
Given $m$ in $\mathscr{M}_0$ and $r_0$ in $H^0(G_F, N[\omega]^{\vee})$, we have
\[\left\langle m,\, (r_0)_{s \in S}\right\rangle_x = 0\]
for all $x$ in $X$.
\end{prop}
\begin{proof}
Take $\chi = \chi(x)$, take $\iota_1: N^{\vee}[\omega] \to N[\omega]^{\vee}$ to be the standard isomorphism, and take $\psi \in H^1(G_F, N^{\vee}[\omega])$ to be the image of $\iota_1(r_0)$ under the connecting map associated to $N^{\chi}[\omega^2]^{\vee}$. We see that $\beta_{\chi}^{-1} \circ \psi$ lies in the $\omega$-Selmer group for $(N^{\vee})^{\chi}$. Poitou--Tate duality gives
\[\sum_{v \text{ of } F} \inv_v(\Psi_x(m) \cup \psi) = 0.\]
Since $\Psi_x(m)$ obeys the local conditions for $v \in \Vplac_0$, this reduces to a sum over $S$. The Selmer group and dual Selmer group have orthogonal local conditions, so we also have
\[\inv_{\pi_s(x) \cap F} (\delta_{\chi, \pi_s(x)}( n_s) \cup \psi) = 0\]
for each $s \in S$. We thus have
\[\sum_{s \in S} \inv_{\pi_s(x) \cap F}((\Psi_x(m) - \delta_{\chi, \pi_s(x)}(n_s)) \cup \psi) = 0,\]
which reduces to $\left\langle m, (r_0)_{s \in S}\right\rangle_x = 0$.

\end{proof}

\subsection{Moment calculations}
We can now prove our main estimate for the moments of $\omega$-Selmer groups in grids. We will need some assumptions.
\begin{ass}
\label{ass:grid_moments}
With all notation as above, we choose a twistable module $N_0$ defined with respect to $(K/F, \Vplac_0, \FFF)$ in either the alternating case or non-self-dual case with respect to all sets of local twists. We will assume that $S$ is nonempty and that
\begin{equation}
\label{eq:cancel_cancel}
\min_{\sigma \in G_1} \#\left\{s \in S_{\text{lg}}\,:\,\, \sigma_s = \sigma \,\text{ in }\, G_1\right\} \ge \frac{| S|}{2\cdot |G_1|} \quad\text{and}\quad \#\Ssm \cup\Smed \le \frac{|S|}{8 \cdot| G_1|}.
\end{equation}

Choose integers $a \ge b \ge 0$, and take
\[N = N_0^{\oplus a}.\]
Choose a subspace $V$ of $\FFF_{\ell}^a$ of dimension $b$. We will consider the submodule
\[Q = V\otimes N_0[\omega]\]
of $N[\omega]$. We define the vector space $\mathscr{M}_1(Q)$ as above for $N$, and we take $\mathscr{N}^-$ to be the subset of $\mathscr{M}_1(Q)$ of elements whose ramification subspace equals $Q$. 

Choose a subset $\mathscr{N}$ of $\mathscr{N}^{-}$. In the case that there is some cancellable alternating homomorphism $\Gamma\colon N[\omega] \to N[\omega]^{\vee}$, we will assume that there is a subspace $W$ of $\mathscr{M}_1(Q)$ of codimension at most $|S| /8| G|$ and some $m$ in $\mathscr{M}_1(Q)$ so
\begin{equation}
\label{eq:cancel_assumption}
\mathscr{N} = \mathscr{N}^- \cap (W + m).
\end{equation}
\end{ass}

\begin{prop}
\label{prop:accurate_grid_moment}
Take all notation and assumptions as in Assumption \ref{ass:grid_moments}, and take $g$ to be the corank of $N$. Choose $x_0 \in X$. If $N_0$ is in the alternating case, take
\[A =  \frac{\# \mathscr{N}}{\# \mathscr{N}^-} \cdot \ell^{\frac{1}{2} b(b+1) } \cdot \left(\#H^0(G_F, N[\omega])\right)^a.\]
If $N_0$ is in the non-self-dual case, choose some $x_0$ in $X$ and take
\[A = \frac{\# \mathscr{N}}{\# \mathscr{N}^-}  \cdot \left(\#H^0(G_F, N[\omega])\right)^{a} \cdot  \left(\prod_{v \in \Vplac_0} \frac{\# W_{v, 1}(N_0, \chi(x_0))}{\#H^0(G_v, N_0[\omega])}\right)^b.\]
Then there is $C > 0$ depending only on $(K/F, \Vplac_0, \FFF)$ so
\begin{align*}
&\left| A \,-\,\frac{1}{\# X} \sum_{x \in X} \left \{m \in \mathscr{N}\,:\,\, \Psi_x(m) \in \Sel^{\omega} N^{\chi(x)}\right\} \right| \\
&\qquad\qquad\le e^{ -|S|/8|G_1| + Cg^2}  + e^{Cg |S|} \cdot (c_{\textup{bilin}} + c_{\textup{Cheb}}).
\end{align*}

\end{prop}

\begin{proof}
We may assume that $S$ is nonempty without issue. Our aim is to calculate the sum
\[ \frac{1}{\# X \cdot \# \mathscr{R}} \sum_{m \in \mathscr{N}} \sum_{r \in \mathscr{R}} \sum_{x \in X} \exp\left(2\pi i \big\langle r,\, m\big\rangle_x\right).\]
The subsum over ignorable $(r, m)$ has magnitude bounded by $ e^{Cg |S|} \cdot (c_{\textup{bilin}} + c_{\textup{Cheb}})$. The subsum over borderline $m$ also fits into the error term by Proposition \ref{prop:borderline_rare}. So we may consider the sum over non-borderline $m$ for which $(r, m)$ is non-ignorable.

Using Proposition \ref{prop:four_qrs}, we may sort these pairs into subcollections based on their associated equivariant alternating map $\Gamma: Q \to N[\omega]^{\vee}/Q^{\perp}$. Suppose we look at the collection of pairs $(r, m)$ with $m$ non-ignorable associated with a given cancellable $\Gamma$; call this set of pairs $\mathscr{Y}$.  Taking $W$ to be the vector space appearing in \eqref{eq:cancel_assumption}, we may use \eqref{eq:direct_Gauss} to show
\[\frac{1}{\# \mathscr{R}} \sum_{(r, m) \in \mathscr{Y}} \exp\left(2\pi i \big\langle r,\, m\big\rangle_x\right) \le e^{Cg^2} \cdot \ell^{-|S|/ 8 |G_1|}+ \left(1 - \frac{\#\mathscr{N}}{\#W}\right)\]
for every $x \in X$ and for some $C$ determined from the starting tuple. Noting that the final difference of the right hand term is no larger than $e^{Cg^2} \cdot \ell^{(|S|/8|G_1|) -|S_{\text{jury}}|}$, we find that the sum over cancellable pairs fits into the error term. 

This just leaves the sum over non-cancellable $\Gamma$. Choose a basis $e_1, \dots, e_b$ of  $V$ and endow $V$ with an inner product so this basis is orthonormal. We may think of $\Gamma$ as an alternating map
\[\Gamma \colon V \otimes N_0[\omega]  \to V \otimes N_0[\omega]^{\vee}.\]
Equivalently, we may think of $\Gamma$ as a $b \times b$ matrix whose $ij^{th}$ entry is an equivariant homomorphism $\beta_{ij}: N_0[\omega] \to N_0[\omega]^{\vee}$ that commutes with the connecting maps, subject to the condition $\beta_{ij} = - \beta_{ji}^{\vee}$ and that $\beta_{ii}$ is alternating for all $i, j \le b$. In the alternating case, we find that all these coefficients must be scalar multiples of $\nu$ by Remark \ref{rmk:Wedd} and the assumptions of the alternating case. The scalars $\nu$ is multiplied by form a symmetric matrix, so there are $\ell^{\frac{1}{2}b(b+1)}$ choices of non-cancellable $\Gamma$. In the non-self-dual case, we find that this map must be $0$.

Given such a $\Gamma$, and given non-borderline $m$ with $\bfQ(m) = Q$ , we see there are
\[\# H^0(G_F, N[\omega]^{\vee})^{b} \cdot \prod_{s \in S} (\# N[\omega]^{\sigma_s})^{a - b}\]
choices of $r$ for which $(r, m)$ is a unignorable and associated to the map $\Gamma$.  Propositions \ref{prop:theta} and \ref{prop:poitou_0} show that we have $\langle r,\,m\rangle_x = 0$ for all $x \in X$ and any such pair $(r, m)$. So
\[  \ell^c \cdot \# H^0(G_F, N[\omega]^{\vee})^{b} \cdot \prod_{s \in S} (\# N[\omega]^{\sigma_s})^{-b}\,-\, \frac{1}{|\mathscr{N}| |\mathscr{R}| | X|} \sum_{m \in \mathscr{N}} \sum_{r \in \mathscr{R}} \sum_{x \in X} \exp\left(2\pi i \big\langle r,\, m\big\rangle_x\right)\]
has magnitude within the error of the proposition, where $c = \frac{1}{2}b(b+1)$ in the alternating case and $c = 0$ otherwise.

To prove the proposition, we just need to estimate the size of $\mathscr{N}^-$. We can calculate the size of $\mathscr{M}_1(Q)$ directly using \cite[Section 8.7.9]{Neuk08}, the assumption \eqref{eq:cancel_cancel}, and the assumption that $N^{\vee}$ has no universal $\omega$-Selmer element.  This gives
\[\#\mathscr{M}_1(Q) = \frac{(\# H^0(G_F, N_0[\omega]))^a}{(\# H_0(G_F, N_0[\omega]^{\vee}))^b} \cdot \left(\prod_{v \in \Vplac_0} \frac{\# W_{v, 1}(N_0, \chi(x))}{\#H^0(G_v, N_0[\omega])}\right)^b \cdot \prod_{s \in S} \left(\#N[\omega]^{\sigma_s}\right)^b.\]
In the alternating case, this simplifies to
\[\# \mathscr{M}_1(Q) = \left(\# H^0(G_F, N_0[\omega])\right)^{a -b} \cdot \prod_{s \in S} \left(\#N[\omega]^{\sigma_s}\right)^b.\]
 The proposition follows once we note that $\mathscr{M}_1(Q) \backslash \mathscr{N}^-$ is negligible in $\mathscr{M}_1(Q)$.
\end{proof}

\section{Gridding}
\label{sec:gridding}
In the previous few sections, we have studied how one can control $\omega$-Selmer groups in a grid of twists of a given Galois module. To prove our main results, we will show that $\XXfav_{F, N}(H, (\chi_v)_{v \in \Vplac_0})$ can be partitioned into such grids with negligible remainder.

As a starting point, we give the following estimate for the size of $\mathbb{X}_{F}(H)$.
\begin{prop}
\label{prop:XFHsize}
Given $(K/F, \Vplac_0, \FFF)$ as above, there is a positive real number $C$ so that we have an asymptotic equivalence
\[\# \mathbb{X}_{F}(H)\sim C H (\log H)^{\kappa - 1}\]
as $H$ tends to infinity, where we have set $\kappa = \# \FFF(-1)^{\times}/ [F\left(\mu_{|\FFF|}\right) : F]$.
\end{prop}

\begin{proof}
This follows from an application of \cite[Proposition 4.16]{Alb22}.
\end{proof}

We will use the asymptotic expression $CH (\log H)^{\kappa - 1}$ as a comparison point to prove that certain collections of twists are negligible.

Our grids of twists are built out of grids of ideals in $F$. We set these up now.

\subsection{Grids of ideals}
\begin{notat}
\label{notat:pregrid}
Fix $(K/F, \Vplac_0, \FFF)$ as above. We take $\mathscr{P}$ to be the set of primes of $F$ outside $\Vplac_0$ that split completely in $F(\zeta)/F$. Taking $G_1 = \Gal\left(K(\Vplac_0)/F\left(\mu_{|\FFF|}\right)\right),$ we define $G_1/\sim$ to be the set of equivalence classes of $G_1$ under conjugation by $G_F$. Given a prime $\mfp$ in $\mathscr{P}$, we see that $\{\Frob\, \ovp\,:\,\, \ovp | \mfp\}$ maps to one of the classes of $G_1/\sim$, and we write this class as $\Frob\,\mfp$.

Take $H$ to be a positive real number satisfying $\log^{(3)} H > 1$. We define
\[\alpha(H) =\exp\left( \exp^{(3)}\left(\tfrac{1}{4} \log^{(3)} H\right)^{-1}\right)\quad\text{and}\quad a_0(H) = \exp^{(3)} \left(\tfrac{1}{3} \log^{(3)} H\right).\]
For $i$ a nonnegative integer, we then define
\[\mathscr{P}_i(H) = \left\{ \mfp \in \mathscr{P} \,:\,\, a_0(H) \cdot \alpha(H)^i \le N_{F/\QQ}(\mfp) <  a_0(H) \cdot \alpha(H)^{i+1} \right\}.\]
If $C$ is a class of $\GOsim$, we will also define
\[\mathscr{P}_i(H, C) = \left\{ \mfp \in \mathscr{P}_i(H) \,:\,\,\Frob \, \mfp = C \right\}.\]
We also take the notation
\[i_{\textup{med}}(H) = \left\lceil \exp^{(2)} \left(\tfrac{1}{2} \log^{(3)} H \right)\big/\log \alpha(H) \right\rceil.\]
\end{notat}

\begin{defn}
\label{defn:ideal_grid}
With $(K/F, \Vplac_0, \FFF)$ and $H$ as above, take $S$ to be a finite set partitioned as $\Ssm \cup \Smed \cup \Slg$. For $s \in \Ssm$, take $\mfp_s$ be a prime in $\mathscr{P}$ satisfying $N_{F/\QQ}(\mfp_s) \le a_0(H)$. For $s \in \Smed \cup \Slg$, take $i_s$ to be a nonnegative integer. We make the following assumptions:
\begin{itemize}
\item For $s_1 \ne s_2$ in $\Ssm$, $\mfp_{s_1}$ and $\mfp_{s_2}$ are distinct.
\item For $s_1 \ne s_2$ in $\Smed \cup \Slg$, $i_{s_1}$ and $i_{s_2}$ are distinct.
\item For $s \in \Smed \cup \Slg$, $s$ is in $\Smed$ if and only if $i_s < i_{\text{med}}(H)$.
\end{itemize}
Taking the notation
\[Y_s = \begin{cases} \{\mfp_s\} &\text{ if } s \in \Ssm \\ \mathscr{P}_{i_s}(H)&\text{ if } s \in \Smed \cup \Slg,\end{cases}\]
we define a set of ideals $Y$ of $F$ by
\begin{equation}
\label{eq:XfromXs}
Y = \left\{ \prod_{s \in S} \mfp_s \,:\,\, \mfp_s \in Y_s \text{ for } s \in S\right\}.
\end{equation}
If $Y$ contains no ideal of norm greater than $H$, we will say that $Y$ is an \emph{unfiltered grid of ideals of height} $H$.

Suppose that the tuple
\[\left(H, S,\, (\mfp_s)_{s \in \Ssm},\, (i_s)_{s \in \Smed \cup \Slg}\right)\]
defines an unfiltered grid of ideals of height $H$. For $s \in S$, take $C_s$ to be an class of $\GOsim$. For $s \in \Ssm$, we assume that $\mfp_s$ has Frobenius class $C_s$. Then, defining
 \[Y_s = \begin{cases} \{\mfp_s \} &\text{ if } s \in \Ssm \\ \mathscr{P}_{i_s}(H, \,C_s)&\text{ if } s \in \Smed \cup \Slg,\end{cases}\]
the set $Y$ defined by \eqref{eq:XfromXs} will be called a \emph{grid} (or \emph{filtered grid}) \emph{of ideals} of height $H$.
\end{defn}

\begin{notat}
\label{notat:hateful_eight}
Fix $(K/F, \Vplac_0, \FFF)$ and $H >20$. Take all notation as above.  Write $\mathscr{H}$ for the set of squarefree products of primes in $\mathscr{P}$. Given $\mfh \in \mathscr{H}$ that satisfies $N_{F/\QQ}(\mfh) \le H$, we will write $\omega(\mfh)$ for the number of distinct prime divisors of $\mfh$, and $\omega_{\textup{sm}}(\mfh)$ for the number of prime divisors of norm at most $a_0(H)$. We will take $\mathbb{X}_F(\mfh)$ to be the set of $\chi$ in $\mathbb{X}_F$ with $\mfh_F(\chi) = \mfh$.

In addition, we will call $\mfh$ \emph{good} if the following assumptions hold:
\begin{enumerate}
\item \emph{(Not too many primes)} We have
\[\omega(\mfh) \le (\log \log H)^2.\]
\item \emph{(Inside a grid)} There is  a filtered grid of ideals of height $H$ corresponding to the tuple
\[ \left(S,\, (\mfp_s)_{s \in \Ssm},\, (i_s)_{s \in \Smed \cup \Slg},\, (C_s)_s \right)\]
 that contains $\mfh$.
\item \emph{(Not too many small primes)} We have
\[|\Ssm|\, \le \,(\log^{(2)} H)^{\frac{1}{3} +\frac{1}{100} }\quad\text{and}\quad|\Smed| \le (\log^{(2)} H)^{\frac{1}{2} +\frac{1}{100} }.\]
\item \emph{(Enough primes)} We also have
\[|S| = \omega(\mfh) \, \ge\, \frac{\log^{(2)} H}{\log^{(3)} H}.\]
\item \emph{($(C_s)_s$ is balanced)} For any $C$ in $\GOsim$, we have
\[\left| \#\left\{ s \in \Slg\,:\,\, C_s = C\right\} \,\,- \,\,  \frac{\# C}{\# G_1} \cdot \#S  \right| \le  \#S^{3/4}.\]
\item \emph{(No Siegel zeros)}
Given $x \ge 2$, there is at most one positive squarefree integer $d_{x, \text{Sie}} \le x$ so the Dedekind zeta function associated to $\QQ(\sqrt{d_{x, \text{Sie}}})$ has a real zero particularly close to $s = 1$; we give a more precise specification for this potentially-defined integer in Proposition \ref{prop:dxsie}. Taking
\[x = \exp^{(3)}\left( \tfrac{2}{5} \log^{(3)} H\right),\]
we assume that either $\ell \ne 2$, that $d_{x, \textup{Sie}}$ does not exist, or that $N_{F/\QQ}(a_{\Vplac_0} \cdot \mfh_{\text{sm}})$ is indivisible by $d_{x, \textup{Sie}}$, where $\mfh_{\text{sm}}$ is the product of the primes of $\mfh$ indexed by $\Ssm$ and $a_{\Vplac_0}$ is the product of all rational primes divisible by some prime in $\Vplac_0$.
\item\emph{(Prepared for higher Selmer work)} The norm of each ideal in the grid is at least $\sqrt{H}$. Furthermore, among the primes indexed by 
\[\{s \in \Slg \,:\,\,  C_s = \{1\}\},\]
there are at least $(\log \log H)^{2/3 - 1/100}$ primes of norm at most
\[\exp^{(3)}\left(\tfrac{2}{3}\log^{(3)} H\right)\]
and at least $(\log \log H)^{1 - 1/100}$ primes of norm at least
\[\exp^{(3)}\left(\tfrac{3}{4}\log^{(3)}H\right).\]

\item \emph{($(C_s)_s$ is not overbalanced)} Given a nonzero function $f: \GOsim  \rightarrow \Z$
whose magnitude does not exceed $\exp^{(2)}\left(\tfrac{1}{2}\log^{(4)} H\right)$,
we have
\[\left|\sum_{s \in S} f(C_s) \right| \ge \# S^{1/4}.\]
\end{enumerate}

Any ideal not satisfying one of these properties will be called \emph{bad}. Given an integer $j$ satisfying $1 \le j \le 8$,  we will write $\mathscr{H}_{\textup{bad}, j}(H)$ for the set of $\mfh$ as above that have the first $j - 1$ properties enumerated above but which do not have the $j^{th}$ property.
\end{notat}
It can be difficult to control the Selmer groups of twists $N^{\chi}$ for which $\mfh(\chi)$ is a bad ideal. Fortunately, such twists  can be safely ignored when doing statistical work, as can be shown by comparing the next result to Proposition \ref{prop:XFHsize}.
\begin{prop}
\label{prop:hateful_eight}
With all notation as above, take $\kappa = \# \FFF(-1)^{\times}/ [F(\zeta) : F]$. Also fix some sufficiently small $\epsilon > 0$.
Then there is a real number $C > 0$  determined just from $(K/F, \Vplac_0, \FFF)$ and $\epsilon$ so that, for all $H > C$ and $j$ in $\{1, \dots, 8\}$, we have
\begin{equation}
\label{eq:bad_moments}
\sum_{\mfh \in \mathscr{H}_{\textup{bad}, j}(H)} (\# \FFF(-1)^{\times})^{\omega(\mfh)} \cdot \exp\big( a_j(\mfh, H)\big)\,\,  \le\,\, H(\log H)^{\kappa - 1} \cdot E_j(H),
\end{equation}
where we have taken
\begingroup
\allowdisplaybreaks
\begin{alignat*}{4}
&a_1(\mfh, H) =  \omega(\mfh) \cdot (1 - \epsilon) \log^{(3)} H   && E_1(H) = \exp\left( -\left(\log^{(2)} H\right)^2 \right) \\
&a_2(\mfh, H) = \omega(\mfh) \cdot \exp^{(2)}\left(\tfrac{1}{4} \log^{(3)} H\right)^{1 - \epsilon} \quad\,\,&& E_2(H) = \exp^{(3)}\left(\tfrac{1}{4} \log^{(3)} H\right)^{-1 + \epsilon}\\
&a_3(\mfh, H) = \omega_{\textup{sm}}(\mfh) \cdot \left(\tfrac{1}{100} - \epsilon\right) \log^{(3)} H && E_3(H) = \exp^{(2)}\left(\tfrac{1}{3} \log^{(3)} H\right) ^{-1}\\
&a_4(\mfh, H) = \omega_{\textup{sm}}(\mfh) \cdot \left(\log^{(2)} H\right)^{\frac{2}{3} - \frac{1}{100} - \epsilon} && E_4(H) = (\log H)^{-\kappa + \epsilon}\\
&a_5(\mfh, H) = \omega_{\textup{sm}}(\mfh) \cdot  \left(\log^{(2)} H\right)^{\frac{1}{6} - \frac{1}{100} - \epsilon} &&  E_5(H) = \exp^{(2)}\left(\left(\tfrac{1}{2} - \epsilon\right) \cdot \log^{(3)} H\right)^{-1}\\
&a_6(\mfh, H) = \omega_{\textup{sm}}(\mfh) \cdot \left(\log^{(2)} H\right)^{\frac{1}{15} - \frac{1}{100} - \epsilon} && E_6(H) = \exp^{(2)}\left(\tfrac{2}{5} \log^{(3)} H\right)^{-1 + \epsilon}\\
&a_7(\mfh, H) = 0 && E_7(H) = \exp^{(2)}\left(\left(\tfrac{2}{3} - \epsilon\right)\cdot  \log^{(3)} H\right)^{-1}\\
&a_8(\mfh, H) = 0 && E_8(H) = \left(\log^{(2)} H\right)^{-\frac{1}{4} + \epsilon}.
\end{alignat*}
\endgroup
\end{prop}
We will prove this result in Sections \ref{ssec:first_five} and \ref{ssec:Frob_prob}.

\subsection{The first seven cases of Proposition \ref{prop:hateful_eight}}
\label{ssec:first_five}
\subsubsection{Some preliminaries}
The unconditional effective Chebotarev density theorem we will use comes from \cite{LaOd77}. For convenience, we record the form of this result.

\begin{thm}[\cite{LaOd77}]
\label{thm:LaOdCheb}
There are absolute effective constants $C, c$ so we have the following:

Take $L/F$ to be a Galois extension of number fields, take $n_L$ to be the degree of $L$ over $\QQ$, and take $\Delta_L$ to be the absolute value of the discriminant of $L$. Take $\phi$ to be a real valued class function on $G = \textup{Gal}(L/F)$ of magnitude at most one. 
\begin{itemize}
\item The Dedekind zeta function for $L$ has at most one real zero in the interval 
\[[1 - \alpha,\, 1] \quad\text{ with } \alpha = \begin{cases} 1/2 &\text{ if } L = \QQ \\ 1/(4 \log \Delta_L) &\text{ otherwise.}\end{cases}\]
If this zero exists, it is simple. Take $\beta_0$ to be this zero if it exists, and take $\beta_0 = 1/2$ otherwise.
\item Given $H> 0$ satisfying
\[ \log H \ge 10 n_L (\log \Delta_L)^2,\]
we have
\[\left| \sum_{\substack{\mathfrak{p} \text{ prime in } F \\ N_{F/\QQ}(\mathfrak{p}) \le H}} \phi(\text{\emph{Frob }} \mathfrak{p})\,-\, \frac{1}{|G|} \left(\sum_{\sigma \in G} \phi(\sigma)\right) \text{\emph{Li }}(H)\right| \le \text{\emph{Li }}(H^{\beta_0}) + Cn_L He^{-c \sqrt{\frac{\log H}{n_L}}}.\]
\end{itemize}
\end{thm}
Stronger forms of this theorem are now known (see \cite{ThZa18}), but this statement suffices for our purposes. To use this theorem, we need some control on the Siegel zero $\beta_0$. Our source on this is \cite{Stark74}.

\begin{prop}
\label{prop:dxsie}
 There exists an absolute effective constant $c > 0$ so we have the following:

For any $x \ge 2$, there is at most one rational squarefree integer $d_{x, \textup{Sie}}$ of magnitude at most $x$ such that the zeta function associated to $\QQ(\sqrt{d_{x, \textup{Sie}}})$ has a real zero in the interval
\[\left(1 - \frac{c}{\log x},\,\,\, 1\right).\]
Furthermore, the integer $d_{x, \textup{Sie}}$ satisfies
\[d_{x, \textup{Sie}}\, \ge\, c \cdot \log x.\]
Finally, suppose $K$ is a number field of degree $n_K$ and discriminant of magnitude at most $x$. We assume $K$ does not contain $\QQ(\sqrt{ d_{x, \textup{Sie}}})$ if $d_{x, \textup{Sie}}$ exists and that $K/\QQ$ is normal. Then the Dedekind zeta function for $K$ contains no zero in the interval
\[\left(1 - \frac{c}{ \log x},\, 1\right).\]
\end{prop}

\begin{proof}
For the first part, suppose $d_1, d_2 < x$ are distinct squarefree rational integers other than $1$, and take $L = \QQ(\sqrt{d_1}, \sqrt{d_2})$.  $L/\QQ$ is then a degree $4$ extension of discriminant dividing $2^8 (d_1d_2)^2$. Its Dedekind zeta function is a product of the Riemann zeta function with three Artin $L$-functions corresponding to the quadratic characters for $\QQ(\sqrt{d_1})$, $\QQ(\sqrt{d_2})$ and $\QQ(\sqrt{d_1d_2})$. These Artin $L$-functions are entire. Per the first part of the above theorem, the Dedekind zeta function for $L$ has at most one zero in the range
\[(1 - (16 \log (4x))^{-1}, \,1).\]
In particular, only one of the three mentioned $L$-functions  can have a zero in this range. The uniqueness of $d_{x, \textup{Sie}}$ follows by adjusting $c$. Its effective lower bound follows from \cite[Theorem 1]{Stark74}, though this result was known earlier; this theorem actually gives the stronger, still-effective statement
\[d_{x, \textup{Sie}} \ge c (\log x)^2.\]
Finally, the result for $K/\QQ$ is a direct consequence of \cite[Lemma 8]{Stark74}.
\end{proof}

Our next preliminary result is an upper bound for the number of ideals of $F$ with a certain number of prime divisors of which a certain number are small. The proof of this proposition largely follows an argument of Hardy and Ramanujan \cite{HaRa17}.

\begin{prop}
\label{prop:Hardy}
Given $(K/F, \Vplac_0, \FFF)$ as above, there is $C > 0$ so we have the following:

Choose integers $r > k \ge 0$ and real numbers $x, y > 2$ satisfying
\[y^{k +1} \le x,\]
and take $\pi_{r, k}(x, y)$ to be the number of ideals $\mfh$ in $\mathscr{H}$ of norm at most $x$ that satisfy $\omega(\mfh) = r$ and for which $\mfh$ is divisible by exactly $k$ prime ideals of norm at most $y$. Take $d = [F(\zeta):F]$. Then
\[\pi_{r, k}(x, y) \,\le\, \frac{Cx}{\log x}\cdot\frac{(d^{-1} \cdot \log \log y + C)^k}{k!} \cdot \frac{(d^{-1} \cdot (\log \log x - \log \log y) + C)^{r-k - 1}}{(r- k -1)!}.\]
\end{prop}
\begin{proof}
Applying partial summation to Theorem \ref{thm:LaOdCheb}, we find that there is a real constant $c_{\text{Mert}}$ and positive numbers $C, c$ so that we have
\begin{equation}
\label{eq:Mert}
\Bigg|c_{\text{Mert}} + d^{-1} \log \log H - \sum_{\substack{\mfp \in \mathscr{P} \\ N_{F/\QQ}(\mfp) < H}} \frac{1}{N_{F/\QQ}(\mfp)} \Bigg| \le C \exp(-c \sqrt{\log H})
\end{equation}
for $H > 3$ and
\[\sum_{\substack{\mfp \in \mathscr{P} \\ N_{F/\QQ}(\mfp) < H}} \frac{\log N_{F/\QQ}(\mfp)}{N_{F/\QQ}(\mfp)} \le C \cdot \log H\]
for $H > 1$. These results generalize two of Mertens' theorems.
We can now follow the argument of \cite[Lemma A]{HaRa17} to show that there is some $C> 0$ so that 
\begin{equation}
\label{eq:HaRaMert}
\sum_{\substack{\mfp \in \mathscr{P}\\ y_2 \le N_{F/\QQ}(\mfp) \le y_1}} \frac{\log x}{N_{F/\QQ}(\mfp) \cdot \log (x/N_{F/\QQ}(\mfp))}\le d^{-1} \cdot (\log \log y_1 - \log \log y_2) + C
\end{equation}
for real numbers $x, y_1, y_2$ satisfying $\sqrt{x} \ge y_1 \ge y_2 \ge 2$. The argument from \cite{HaRa17} also gives the inequality
\[\pi_{r, 0}(x, y) \le \frac{1}{r-1} \sum_{\substack{\mfp \in \mathscr{P} \\ y \le N_{F/\QQ}(\mfp) \le \sqrt{x}}} \pi_{r-1, 0}\left(\frac{x}{N_{F/\QQ}(\mfp)},\, y\right),\]
and we can then use this inequality and \eqref{eq:HaRaMert} to inducitvely prove
\[\pi_{r, 0}(x, y) \,\le\, \frac{Cx}{\log x}\cdot \frac{(d^{-1} \cdot (\log \log x - \log \log y) + C)^{r - 1}}{(r-1)!},\]
for some $C > 0 $ determined from $(K/F, \Vplac_0, \FFF)$, with the Chebotarev density theorem sufficing for the case $r = 1$.

We also have the relation
\[\pi_{r, k}(x, y) \le \frac{1}{k} \sum_{\substack{\mfp \in \mathscr{P} \\ N_{F/\QQ}(\mfp) \le y}} \pi_{r-1, k-1}\left(\frac{x}{N_{F/\QQ}(\mfp)},\, y\right),\]
which holds for $k > 0$. Using the above estimate on $\pi_{r, 0}$ as a base case, the full proposition can be proved inductively from this relation and \eqref{eq:HaRaMert}.
\end{proof}

\subsubsection{The case $j =1$ of Proposition \ref{prop:hateful_eight}}
Take $r > (\log^{(2)} H)^2$. The contribution of the portion of $\mathscr{H}_{\text{bad},1}$ with $r$ prime divisors to \eqref{eq:bad_moments} can be bounded by Proposition \ref{prop:Hardy}. For sufficiently large $H$, this contribution is at most
\[\frac{H}{\log H} \cdot \frac{(d_{\FFF}^{-1} \cdot \log \log H + C)^{r-1}}{(r-1)!} \cdot \exp\left((1 - \epsilon/2)\cdot r \cdot \log^{(3)}H\right).\]
Using Stirling's approximation, we can find some $C_0 > 0$ so that this is at most
\begin{align*}
&\frac{H}{\log H} \exp\left( r \cdot C_0 + ( 2- \epsilon/2) r \log^{(3)} H - r \log r\right) \\
\le \,&\frac{H}{\log H}\exp\left( r \cdot C_0 - (\epsilon/2) r \log r\right) \le \frac{H}{\log H} e^{-r}
\end{align*}
for any and $H > C_0$. The sum of these terms over $r > (\log^{(2)} H)^2$ is then within the bound of the proposition.

\subsubsection{The case $j =2$ of Proposition \ref{prop:hateful_eight}}
For $j =2$, we note that there are two reasons an ideal $\mfh$ in $\mathscr{H}$ of norm at most $H$ can fail to appear in a grid of ideals of height $H$:
\begin{enumerate}
\item It can have more than one prime factor from some interval $\mathscr{P}_i(H)$, or
\item It can share a grid $Y$ as in \eqref{eq:XfromXs} with an ideal of norm greater than $H$.
\end{enumerate}
Our starting point is to note that there are constants $C_0, C_1, c > 0$ depending just on $F$ so that, for $H > C_1$, the number of integral ideals of $F$ of norm at most $H$ is within $C_1 H^{1-c}$ of $C_0 H$; see \cite{Murt07} for a streamlined approach to this old result. From \eqref{eq:Mert}, we can find a $C > 0$ depending just on $F(\zeta)$ so that
\[\sum_{\mfp \in \mathscr{P}_i(H)} \frac{1}{N_{F/\QQ}(\mfp)} \le C\cdot  \min\left(\log \alpha,\, (i+1)^{-1}\right),\]
where $\mathscr{P}_i(H)$ and $\alpha = \alpha(H)$ are defined as in Notation \ref{notat:pregrid}. As such, the number of ideals of norm at most $H$ with at least two prime divisors from $\mathscr{P}_i(H)$ is at most $C H \min\left(\log \alpha,\, (i+1)^{-1} \right)^2$ for some $C$ depending just on $F(\zeta)$. Summing over all $i$, we see the number of ideals in $\mathscr{H}_{\textup{bad}, 2}(H)$ that are bad for the first reason listed above is bounded by
\[\frac{CH}{ \exp^{(3)} \left(\frac{1}{4} \log^{(3)} H\right)}\]
for some $C$ depending just on $F(\zeta)$. 

If the ideal $\mfh$ is in $\mathscr{H}_{\textup{bad}, 2}$ for the second reason, we have
\[H \ge N_{F/\QQ}(\mfh) \ge H \cdot \alpha^{- \omega(\mfh)}.\]
Using the bound $\omega(\mfh) \le (\log^{(2)} H)^2$ that holds for these ideals, we can bound the number of ideals of this type by
\[\frac{CH}{ \left(\exp^{(3)} \left(\frac{1}{4} \log^{(3)} H\right)\right)^{1- \epsilon}},\]
where $C$ depends on $F$ and the positive constant $\epsilon > 0$. These bounds suffice to give the $j = 2$ case.

\subsubsection{The case $j =3$ of Proposition \ref{prop:hateful_eight}}
Given an integer $k \ge 0$ and real numbers $x, y \ge 2$ satisfying $y^{k+1} \le x$, we may apply the Taylor expansion of $e^x$ and Proposition \ref{prop:Hardy} to show that sum $\sum_{\mfh} (\# \FFF(-1)^{\times})^{\omega(\mfh)}$ taken over ideals of height at most $x$ with at least $k$ factors less than $y$ has upper bound
\[Cx \cdot (\log x)^{\kappa - 1} \cdot \frac{(d^{-1} \cdot \log \log y + C)^k}{k!},\]
where $C > 0$  is determined just from $F(\zeta)$. The part follows from Stirling's approximation.

\subsubsection{The case $j =4$ of Proposition \ref{prop:hateful_eight}}
We note that $a_4(\mfh, H)$ is bounded by $\left(\log^{(2)} H\right)^{1 - \epsilon}$ for $\mfh$ in $\mathscr{H}_{\text{bad}, 4}(H)$. With this observed, the result follows from Proposition \ref{prop:Hardy} and Stirling's approximation.

\subsubsection{The case $j =5$ of Proposition \ref{prop:hateful_eight}}
Given an unfiltered grid of ideals that are not bad for one of the previously mentioned reasons, the Chebotarev density theorem and Hoeffding's inequality \cite[Theorem 1]{Okam59} imply that a randomly selected ideal from the grid has probability at most
\[\exp(-c |\Slg|^{1/2}) +   |S| \cdot \exp^{(3)}\left(\left(\frac{1}{3} - \epsilon\right) \log^{(3)} H\right)^{-1} \]
of lying in $\mathscr{H}_{\text{bad}, 5}(H)$ for $H \ge C$, where $c, C >0$ are determined from the starting tuple and $\epsilon > 0$. The case then follows from the lower bound on $|S|$ and the upper bounds on $|\Ssm|$ and $|\Smed|$.

\subsubsection{The case $j =6$ of Proposition \ref{prop:hateful_eight}}
Suppose $d_{x,\text{Sie}}$ exists for $x = \exp^{(3)}\left(\frac{2}{5}\log^{(3)} H\right)$. From Proposition \ref{prop:dxsie}, we find there is some absolute $c > 0$ so that $d_{x, \text{Sie}}$ is at least $c \exp^{(2)}\left(\frac{2}{5}\log^{(3)} H\right)$. Take $d_0$ to be the product of the prime factors of $d_{x, \text{Sie}}$ not divisible by a prime in $\Vplac_0$. If this number has more that $(\log^{(2)}H)^{1/3 + 1/100}$ prime factors, then no ideals will lie in $\mathscr{H}_{\text{bad}, 6}(H)$. So we may assume it has at most this many prime factors, and we can bound the number of squarefree ideals in $F$ whose norm divides some power of $d_0$ by 
\[\exp\left(C \cdot (\log^{(2)}H)^{1/3 + 1/100}\right)\]
 for some $C > 0$ depending on $F$. The result then follows from an application of Proposition \ref{prop:Hardy}.

\subsubsection{The case $j =7$ of Proposition \ref{prop:hateful_eight}}
The restriction to ideals of norm at least $\sqrt{H}$ follows immediately from Proposition \ref{prop:XFHsize}.

Using Proposition \ref{prop:Hardy}, we find that we can restrict our attention to ideals with at least $(\log^{(2)} H)^{2/3 - \epsilon/2}$ prime factors below $\exp^{(3)}\left(\frac{2}{3} \log^{(3)} H\right)$ and $(\log^{(2)} H)^{1 - \epsilon/2}$ prime factors above $\exp^{(3)}\left(\frac{3}{4}\log^{(3)} H\right)$. The result then follows from the Chebotarev density theorem and Hoeffding's inequality.

\subsection{The joint distribution of Frobenius elements}
\label{ssec:Frob_prob}
The final case of Proposition \ref{prop:hateful_eight} relates to the distribution of Frobenius classes for the primes dividing $\mfh$. It is not too difficult to prove this case directly using the De Moivre--Laplace  local limit theorem \cite{Sinai92} together with the Chebotarev density theorem. However, in order to estimate the proportion of twists in $\mathbb{X}_F(H, (\chi_v)_{ v \in \Vplac_0})$ that are favored for a given $N$, we will need a more robsut model for this distribution of Frobenius classes. Our next goal is to produce this model.

\begin{defn}
\label{defn:frob_model}
Choose a nonempty finite set $G$. We take $X_1, X_2, \dots$ to be a sequence of independent random variables on $G$ with the uniform distribution. Given $n \ge 1$, we take $g_n$ to be a random variable valued in $\Z^G$ defined by
\[g_n(\sigma) = \# \{i \le n \,:\,\, X_i = \sigma \}\quad\text{ for all } \sigma \in G.\]
We take $V\in \R^G$ to be a random function with a multivariate normal distribution whose mean is $0$ and whose covariance matrix $\Sigma: G \times G \to \mathbb{R}$ is given by
\[\Sigma_{\sigma, \tau} = \begin{cases} |G|^{-1} - |G|^{-2}&\text{ if } \sigma = \tau\\ -|G|^{-2} &\text{ otherwise.}\end{cases}\]
Finally, we fix some element $\sigma_0$ in $G$.
\end{defn}

\begin{prop}
\label{prop:G1_model}
Take $G$ and $\sigma_0$ as in Definition \ref{defn:frob_model}, and define the random variables $X_1, X_2, \dots $ and random functions $V, g_1, g_2, \dots$ as in that definition. Choose a real number $\delta \in (0, 1/2)$ and a positive integer $R$. Then there is some $C > 0$ depending on $G$, $R$, and $\delta$ so we have the following:

Choose $k \ge 0$, real numbers $b_1, \dots, b_k$ in $[-1, 1]$, and nonzero integral functions $f_1, \dots f_k: G \to \mathbb{\Z}$.  Also choose an integer $n > C$ and a function $a: G\backslash \{\sigma_0\} \to \Z/R\Z$. Take $P$ to be the probability
\[\textup{Pr}\left(\sum_{i \le n} f_j(X_i) \ge b_jn^{\delta} \,\text{ for all } j \le k \,\,\text{ and }\,\,g_n(\sigma) \equiv a(\sigma) \,\textup{ mod } R\,\text{ for all } \sigma \in G\backslash \{\sigma_0\}\right).\]
Take $B$ to be the maximum magnitude attained by any function $f_j$. We then we have the following:
\begin{itemize}
\item If, for some $j \le k$, $\sum_{\sigma \in G} f_j(\sigma)$ is negative, then $P \le C\exp\left(-C^{-1}B^{-2}n\right)$.
\item Otherwise, take $Z$ to be the set of $j \le k$ for which $\sum_{\sigma \in G} f_j(\sigma)$ is zero, and take
\[P_0 = \textup{Pr}\left(\sum_{\sigma \in G} V_{\sigma} \cdot f_j(\sigma) \ge 0 \,\text{ for all } j \in  Z\right).\]
Then
\[\left|P - R^{1 -|G|}P_0\right| \le CkBn^{-1/2 + \delta}.\]
\end{itemize}
\end{prop}
\begin{proof}
The first part follows immediately from Hoeffding's inequality \cite[Theorem 1]{Okam59}. Hoeffding's inequality also allows us to ignore functions with a positive sum when estimating $P$, so we reduce to the case where $Z = \{1, \dots, k\}$.

The random function $g_n$ has a uniform multinomial distribution. As $n$ tends to $\infty$, the normalized random function $n^{-1/2}\left(g_n - n |G|^{-1}\right)$
has distribution converging  to the random function $V$ \cite[Section 2.7]{Serf80}. 

We note that the covariance matrix of each normalized random function equals that of $V$. This covariance matrix has rank $|G| -1$; projecting these normalized random functions to $\R^{G \backslash \{\sigma_0\}}$ gives a random function with an invertible covariance matrix. We may then apply a multivariate generalization of the Berry--Essen theorem \cite[Theorem 1.1]{Bentk05} to say that there is some $C_0 > 0$ depending just on $G$ so that, for any convex subset $K$ of $\R^{G \backslash \{\sigma_0\}}$, we have
\[\left|\textup{Pr}\left( \left(n^{-1/2} \cdot g_n(\sigma) - |G|^{-1}\right)_{\sigma \in G \backslash \{\sigma_0\}}  \in K\right)- \textup{Pr}\left((V_{\sigma})_{\sigma \in G \backslash \{\sigma_0\}} \in K\right)\right| \le C_0n^{-1/2}.\]
This implies that the probability that $\sum_{i \le n} f_j(X_i)$ has magnitude bounded by $Bn^{\delta}$ for some $j \le k$ can be bounded by $C_1kBn^{-1/2 + \delta}$, where $C_1 > 0$ depends only on $H$ and $\delta$. In particular, if we take $m = \lfloor n^{\delta} \rfloor$, we see that the condition
\[\sum_{i= 1 }^n f_j(X_i) \ge b_jn^{\delta} \quad\text{ if and only if }\quad  \sum_{i= m }^n f_j(X_i) \ge b_jn^{\delta}\]
holds for all $j \le k$ with probability at least $1 - C_1kBn^{-1/2 + \delta}$.

Take $\varphi: \Z^G \to (\Z/R\Z)^{G\backslash \{\sigma_0\}}$ to be the standard projection. The sequence 
\[\varphi \circ g_0, \,\varphi \circ g_1,\, \varphi \circ g_2,\, \dots\]
 is an irreducible, aperiodic Markov chain on the state space $(\Z/R\Z)^{G\backslash \{\sigma_0\}}$. In particular, given any $a: G \backslash \{\sigma_0\} \to \Z/R\Z$, the convergence theorem for aperiodic irreducible Markov chains \cite[Theorem 4.9]{LePe17} allows us to conclude that
\[ \left|\textup{Pr}\Big(g_m \equiv a(\sigma) \,\text{ for all } \sigma \in G \backslash \{\sigma_0\}\Big) - R^{1-|G|}\right| \le Ce^{-c m}\]
for some $C, c > 0$ depending just on $G$ and $R$. The proposition follows.
\end{proof}

We now finish the proof of Proposition \ref{prop:hateful_eight}. 
\begin{proof}[The final case of Proposition \ref{prop:hateful_eight}]
Choose an unfiltered grid of ideals $Y$  of height $H$ not meeting $\mathscr{H}_{\textup{bad}, j}(H)$ with $j = 1, 3, 4$. Given a class function $f$ as in Notation \ref{notat:hateful_eight} , we may apply the Chebotarev density theorem and Proposition \ref{prop:G1_model} to bound the number of $(\mfp_s)_s$ in $Y$ satisfying  $\left|\sum_{s \in S} f(\Frob\,\mfp_s) \right|\le |S|^{1/4}$ by $C|S|^{-1/4 + \epsilon}$ for some $C$ determined from $(K/F, \Vplac_0, \FFF)$ and $\epsilon > 0$. The result follows from a count of class functions.
\end{proof}

A similar but more involved argument gives the following.

\begin{prop}
\label{prop:pfav}
Given a twistable module $N$ defined with respect to $(K/F, \Vplac_0, \FFF)$, and given $\epsilon > 0$, there are real numbers $P_0, C \ge 0$ so, for $H > C$ and any set of local twists $(\chi_v)_v$, we have
\[\left| \frac{\#\XXfav_{F, N}(H, (\chi_v)_v)}{\#\mathbb{X}_F(H)}  - P_0 \right| \le (\log^{(2)} H)^{-\frac{1}{6} + \frac{1}{100} + \epsilon}.\]
If $N$ is potentially favored in the sense of Definition \ref{defn:pfav2}, $P_0$ is positive. Otherwise, $P_0$ is zero.
\end{prop}

\begin{proof}
Take $Y$ to be an unfiltered grid of ideals of height $H$. We suppose that $Y$ does not meet $\mathscr{H}_{\textup{bad}, j}$ for $j \le 4$. From Proposition \ref{prop:hateful_eight}, we see that it suffices to estimate the proportion of twists in $\bigcup_{\mfh \in Y} \mathbb{X}_F(\mfh)$ that lie in $\XXfav_{F, N}$. One method to produce a twist in this union uniformly at random is to independently choose 
\begin{itemize}
\item An element $\chi_0$ uniformly at random from $\msS_{\FFF/F}(\Vplac_0)$, 
\item An element $\kappa_s$ uniformly at random from $\FFF(-1)^{\times}$ for each $s \in S$, and 
\item A pair of primes $(\mfp_s, \mathfrak{P}_s)$ for each $s \in S$, with $\mfp_s$ chosen uniformly at random from the primes in $Y_s$  and $\mathfrak{P}_s$ chosen uniformly at random among primes of  $K(\Vplac_0)$  dividing $\mfp_s$.
\end{itemize}
Choosing an arbitrary prime $\ovp_s$ of $\ovQQ$ dividing $\mathfrak{P}_s$ for each $s \in S$, we can then define our random twist by
\[\chi = \chi_0 + \sum_{s\in S}\sum_{s \in S} \mfB_{\ovp_s,\, \FFF}(\kappa_s).\]
For each $s \in \Smed \cup \Slg$, we may think of $\FrobF{F}{\ovp_s}$ as a random variable valued in $G_1$. By the Chebotarev density theorem, these random variables approximately take the uniform distribution. By throwing out a negligible portion of each $Y_s$, we may assume that these random variables are exactly equidistributed.

We can now apply Proposition \ref{prop:G1_model} to estimate the probability that a twist chosen this way is favored. Take $T_0 =N[\omega], T_1, \dots, T_k$ to be a maximal set of $G_F$-submodules of $N[\omega]$ so that, for $0 \le i < j \le k$, there is some $\sigma$ so 
\[\dim (N/T_i)[\omega]^{\sigma} \ne \dim (N/T_j)[\omega]^{\sigma}.\]
We define $f_i: \Gal(K(\Vplac_0)/F) \to \Z$ by 
\[f_i(\sigma) =  \dim N[\omega]^{\sigma} -\dim (N/T_i)[\omega]^{\sigma} .\]
First, if $N$ is not potentially favored because \eqref{eq:not_super_unfavored} does not hold for some $T_i$, we can apply the first part of Proposition \ref{prop:G1_model} to $f_i$ to bound the probability that $\chi$ is favored by $C\exp(-c|\Smed \cup \Slg|)$ for some $C, c > 0$ depending just on $N$ and $(K/F, \Vplac_0, \FFF)$. This is comfortably within the bounds of the proposition.

Otherwise, if $N$ is not potentially favored, the hyperplane separation theorem implies that some convex sum of the $f_i$ is $0$. In this case, for $\chi$ to be favored, we must have $\sum_{s \in S} f_j(\FrobF{F}{\ovp_s}) = 0$ for at least one $j \le k$. The probability of this event is at most $(\log \log H)^{-1/2 + \epsilon}$ for $H$ sufficiently large, as can be seen from Proposition \ref{prop:G1_model}

This leaves the case that $N$ is potentially favored. In this case, there is a nonempty open set of superlatives among the set of functions $w: G_1 \to \R$ satisfying $\sum_{\sigma \in G_1} w(\sigma) = 0$. Taking $V$ to be the random function of Definition \ref{defn:frob_model} for $G = G_1$, we find that
\[\sum_{\sigma \in G_1} V_{\sigma} f_j(\sigma) > 0 \,\,\text{ for all } \, j \le k\]
happens with positive probability. After taking $R = |\FFF|$ to control the local conditions, we find that the result follows from Proposition \ref{prop:G1_model}.
\end{proof}

\section{Rough bounds for moments of sizes of $\omega$-Selmer groups}
\label{sec:rough}
With Proposition \ref{prop:hateful_eight} proved, we have all the tools we need to prove the coarse moment bound Theorem \ref{thm:main_mmnt_coarse}. To start, we will transition from grids of ideals over $F$ to grids of twists.

\begin{defn}
\label{defn:grid_refinement}
Suppose $Y = \prod_{s \in S} Y_s$ is a filtered grid of ideals in $F$. For $s \in \Smed \cup \Slg$, we will fix a set of primes $\overline{Y}_s$ of $\ovQQ$ so that every prime in $\overline{Y}_s$ divides some prime in $Y_s$, and so that every prime of $K(\Vplac_0)$ dividing some prime in $Y_s$  is divisible by exactly one prime in $\overline{Y}_s$. For $s \in \Ssm$, we take $\overline{Y}_s = \{\ovp_s\}$, where $\ovp_s$ is any prime of $\ovQQ$ dividing the unique prime in $Y_s$.

Choosing some $\sigma_s \in \Gal\left(K(\Vplac_0)/F\left(\mu_{|\FFF|}\right)\right)$ for every $s \in S$, we may consider the grid of ideals in $\ovQQ$ of the form $X = X((\sigma_s)_s) = \prod_{s} X_s$, where we have taken 
\[X_s = \left\{\ovp \in \overline{Y}_s\,:\,\, \FrobF{F}{\ovp}\equiv \sigma_s \,\text{ in } \Gal(K(\Vplac_0)/F)/\sim\right\},\]
where $\Gal(K(\Vplac_0)/F)/\sim$ is the set of classes associated to $(K/F, \Vplac_0, |\FFF|)$ as in \cite[Definition 3.6]{Smi22a}.

We call $X$ a $\ovQQ$-refinement of $Y$. Given $\chi_0$ and $(\gamma_s)_s$ as in Section \ref{ssec:explicit_local}, we may follow the construction of that section to define a twist $\chi(x) = \chi(\chi_0, (\gamma_s)_s, x)$ for every $x \in X$.

\end{defn}
With this definition, we have associated a collection of elements in $\mathbb{X}_F$ to $Y$ decorated with the extra data $((\sigma_s)_s, (\gamma_s)_s, \chi_0)$. We note that a single homomorphism $\chi$ in $\mathbb{X}_F(\mfh)$ may appear multiple times in the collection of twists associated to $((\sigma_s)_s, (\gamma_s)_s, \chi_0)$, and that the collection of twists associated to some other tuple may still contain $\chi$. The following proposition shows that this redundancy is fine, as every $\chi$ in $\bigcup_{\mfh \in X} \mathbb{X}_F(\mfh)$ is overcounted the same amount.

\begin{prop}
Given a filtered grid of ideals $Y = \prod_{s \in S} Y_s$ as above, there is a positive integer $A$ so that, for every $\chi\in \bigcup_{\mfh \in Y} \mathbb{X}_F(\mfh)$, we have
\[ \sum_{(\sigma_s)_s} \sum_{(\gamma_s)_s}  \sum_{\chi_0} \# \left\{x \in X((\sigma_s)_s) \,:\,\, \chi = \chi(\chi_0, (\gamma_s)_s, x)\right\}\,=\,A.\]
\end{prop}
\begin{proof}
This reduces to the claim that, for any $s \in \Smed \cup \Slg$, the number of elements in $\overline{Y}_s$ dividing $\mfp_s$ does not depend on the choice of $\mfp_s$ in $Y_s$. This in turn follows from the fact that the conjugacy class of $\Frob\, \mfp_s$ does not depend the choice of this element.
\end{proof}

The need for the following proposition was a major influence on our definition of good grids of ideals.

\begin{prop}
\label{prop:real_grid_small_c}
Given $(K/F, \Vplac_0, \FFF)$, there are real constants $c, C > 0$ so we have the following:

Choose $H > C$, and choose a filtered grid of ideals $Y = \prod_{s \in S} Y_s$ in $F$ of height $H$, with $\Ssm$, $\Smed$, and $\Slg$ as in Definition \ref{defn:ideal_grid}. Choose any nonempty $\ovQQ$-refinement $X = \prod_{s \in S} X_s$ of $X$, and choose $\chi_0, (\gamma_s)_s$ as in Definition \ref{defn:grid_refinement}. Take
\[\delta = \exp^{(3)}\left(\tfrac{1}{3}\log^{(3)} H\right)^{-c}.\]
Choose any twistable module $N$ defined with respect to $(K/F, \Vplac_0, \FFF)$, and define $c_{\textup{Cheb}}$ and $c_{\textup{bilin}}$ as in Notation \ref{notat:cJuti} for the collection of twists corresponding to $(\chi_0, (\gamma_s)_s, X)$.

We then have $c_{\textup{bilin}} \le \delta$. Furthermore,  if $Y$ contains no ideals lying in $\mathscr{H}_{\textup{bad}, j}(H)$   for $j \le 6$, we have
$c_{\textup{Cheb}} \le \delta$.
\end{prop}

\begin{proof}
Suppose $\big((r_s)_s, \,(\phi_0, (n_s)_s)\big)$ is bilinearly ignorable. Choose distinct indices $s, t$ $\Smed \cup \Slg$ and $\tau_0 \in G_F$ so $(\tau_0 r_s - r_t) \cdot (\tau_0n_s - n_t)$ is nonzero. Choose $\ovp_{0u} \in X_u$ for every $u \in S \backslash \{s, t\}$,  and take 
\[Z = X_{s} \times X_{t} \times \prod_{u \ne s,t} \{\ovp_{0u}\}.\]
From Proposition \ref{prop:change_p}, we find there are functions $a_t: X_t \to \tfrac{1}{\ell}\Z/\Z$ and $a_s: X_s \to \frac{1}{\ell}\Z/\Z$ so
\begin{align*}
\big\langle\,(r_s)_s, &\,(\phi_0,\, (n_s)_s)\big\rangle_x + a_t(\pi_t(x)) + a_s(\pi_s(x)) \\
&= \sum_{\tau\in B(\sigma_t, \sigma_s)}  \textup{lg}\big( \gamma_s \gamma_t \,(\tau r_s - r_t)\cdot (\tau n_s - n_t) \cdot \kappa\left(\symb{\pi_t(x)}{\pi_s(x)}'(\tau^{-1})\right)\big).
\end{align*}
for all $x \in Z$.

Take $H_s = \max_{\ovp \in X_s} N_{F/\QQ}(\ovp \cap F)$, and define $H_t$ similarly. Using the Chebotarev density theorem, we find that
\[ \min(\#X_s/H_s,\, \#X_t/H_t) \ge \exp^{(3)}\left(\tfrac{1}{4} \log^{(3)} H\right)^{-2}\]
so long as $H$ is sufficiently large. The upper bound for $c_{\textup{bilin}}$ then follows from \cite[Theorem 5.2]{Smi22a}.

Supposing that $Y$ does not meet  $\mathscr{H}_{\textup{bad}, j}(H)$ for $j \le 6$, we now wish to bound $c_{\textup{Cheb}}$. Take $\widetilde{K}$ to be the minimal normal extension of $\QQ$ that contains $K$, and take $L$ to be the minimal abelian extension of $\widetilde{K}$ of exponent dividing $\#\FFF$ that is ramified only at archimedean places and over rational primes divisible by some prime in $\Vplac_0 \cup \bigcup_{s \in \Ssm} Y_s$.

Using the bound on $\# \Ssm$ provided by the definition of $\mathscr{H}_{\textup{bad}, 3}(H)$, we find that there is $C_0 > 0$ so the degree $n_L$ and discriminant $\Delta_L$ of $L/\QQ$ satisfy
\[n_L \le \exp^{(2)}\left(\left(\tfrac{1}{3} + \tfrac{1}{50}\right) \log^{(3)} H\right)\quad\text{and}\quad \Delta_L \le \exp^{(3)}\left(\left(\tfrac{1}{3} + \tfrac{1}{50}\right) \log^{(3)} H\right)\]
so long as $H > C_0$. Since  $\mathscr{H}_{\textup{bad}, 6}(H)$ does not meet $Y$, we find that the Dedekind zeta function for $L$ has no real zero in the interval
\[\left(1 - c\cdot \exp^{(2)} \left(\tfrac{2}{5}\log^{(3)} H\right)^{-1}, \,1\right)\]
by Proposition \ref{prop:dxsie}, where $c > 0$ is the absolute constant of that proposition. Choose $s \in \Slg$, take $C_s \in G_1/\sim$ to be the class corresponding to $Y_s$, and take $D$ to be the subset of $\Gal(L/F)$ projecting to the conjugacy class of $C_s$. Then Theorem \ref{thm:LaOdCheb} implies that,  for any class function $\phi: \Gal(L/F) \to \R$ of magnitude at most $1$, we have
\[\left| \sum_{\mfp \in Y_s} \phi(\textup{Frob} \,\mfp) \,-\, \frac{1}{\# D} \left(\sum_{\sigma \in D} \phi(\sigma)\right) |X_s| \right| \le |X_s| \cdot \exp^{(3)}\left(\left(\tfrac{1}{2} - \epsilon\right) \log^{(3)} H\right)^{-1}\]
for $s \in \Slg$ and $\epsilon > 0$, so long as $H$ is sufficiently large given $(K/F, \Vplac_0, \FFF)$ and $\epsilon$. In other words, if we choose an element $\mfp$ of $X_s$ uniformly at random, if we choose a prime $\mathfrak{P}$ of $L$ over $\mfp$ uniformly at random, and finally choose some prime $\ovp$ over $\mathfrak{P}$, the value of $\FrobF{F}{\ovp}$ will be distributed approximately uniformly at random through $D$. By \cite[Proposition 3.23]{Smi22a}, we find the tuple of symbols 
\[\left(\symb{\ovp}{\ovp_u}\,:\,\, u \in \Ssm, X_u = \{\ovp_u\}\right)\]
 will be approximately uniformly distributed at random through the set of possible tuples of symbols
\[\prod_{s \in \Ssm} \symb{\class{X_s}}{\class{X_u}}\]
if $\ovp$ is chosen uniformly at random from $X_s$. The bound on $c_{\textup{Cheb}}$ then follows from Proposition \ref{prop:change_p}.
\end{proof}

We can now combine Proposition \ref{prop:hateful_eight} and our work in Section \ref{sec:ignorable} in a meaningful way.

\begin{prop}
\label{prop:bye_bad}
Given $(K/F, \Vplac_0, \FFF)$ as above, given $g_0  > 0$, and given any $\epsilon > 0$, there is a real number $C > 0$ so we have the following:

Choose $H > C$, and choose a twistable module $N$ defined with respect to $(K/F, \Vplac_0, \FFF)$. Taking $g$ to be the corank of $N$, we assume that
\[g \le \frac{1}{200 \log \ell} \log^{(3)} H.\]
Then
\[\sum_{j = 1}^7 \sum_{\mfh \in \mathscr{H}_{\textup{bad}, j}(H)} \sum_{\chi \in \mathbb{X}_F(\mfh)} \frac{\# \Sel^{\omega} N^{\chi}}{\max_T \mathcal{T}_{N, T}(\chi)} \le H (\log H)^{\kappa - 1} \cdot \exp^{(2)}\left(\tfrac{1}{3} \log^{(3)} H\right) ^{-1},\]
where $\kappa$ is taken to be $\# \FFF(-1)^{\times}/ [F(\zeta) : F]$ as above. Furthermore, if we instead suppose
\[g \le C^{-1} \sqrt{\log^{(3)} H},\]
we may conclude
\[\sum_{\mfh \in \mathscr{H}_{\textup{bad}, 8}(H)} \sum_{\chi \in \mathbb{X}_F(\mfh)} \frac{\# \Sel^{\omega} N^{\chi}}{\max_T \mathcal{T}_{N, T}(\chi)} \le H (\log H)^{\kappa - 1} \cdot \left(\log^{(2)} H\right)^{-1/4 + \epsilon}.\]
\end{prop}

\begin{proof}
There is a real number $C_0 > 0$ determined from $(K/F, \Vplac_0, \FFF)$ and $N$ so that the dimension of $ \Sel^{\omega} N^{\chi}$ is bounded by $(C_0 + \omega(\mfh_F(\chi)))g$ for every $\chi$ in $\mathbb{X}_F(H)$. This and Proposition \ref{prop:hateful_eight} allow us to bound the impact of twists corresponding to bad ideals in the cases $j = 1, 2$.

To handle the cases $j = 3, 4, 5, 6$, we apply Proposition \ref{prop:rough_grid_average_nocheb} to an $\ovQQ$-refinement $X$ of a bad grid of ideals of these types, using Proposition \ref{prop:real_grid_small_c} to bound the constant $c_{\textup{bilin}}$ in these estimates. We then find that the average size of 
\begin{equation}
\label{eq:Sel_ratio}
\frac{\# \Sel^{\omega} N^{\chi(x)}}{\max_T \mathcal{T}_{N, T}(\chi)}
\end{equation}
over $X$ is bounded by an expression of the form $e^{Cg^2} \cdot \ell^{g \cdot\# \Ssm}$. The case then follows from Propositions \ref{prop:hateful_eight} and \ref{prop:XFHsize}.

To handle the cases $j = 7, 8$ , we apply Propositions \ref{prop:borderline_rare} and \ref{prop:rough_grid_average} to an $\ovQQ$-refinement $X$ of a bad grid of ideals of these types. We may use Proposition \ref{prop:real_grid_small_c} to bound the constants $c_{\textup{bilin}}$ and $c_{\textup{Cheb}}$ in these estimates, and we may use the definition of bad ideals of type $j = 3, 4, 5$ to bound the number $b$ appearing in Proposition \ref{prop:borderline_rare}. We then can bound the average of  \eqref{eq:Sel_ratio} over such a grid by an expression of the form $e^{Cg^2}$, and the cases follow from Propositions \ref{prop:hateful_eight} and \ref{prop:XFHsize}.
\end{proof}

Extending this proof a little further gives Theorem \ref{thm:main_mmnt_coarse}.

\begin{proof}[Proof of Theorem \ref{thm:main_mmnt_coarse}]
We work with the twistable module $N^{\oplus g}$. As in  the proof of Proposition \ref{prop:bye_bad}, apply Propositions \ref{prop:borderline_rare} and  \ref{prop:rough_grid_average} to an $\ovQQ$-refinement $X$ of a grid of either good ideals or ideals lying in $\mathscr{H}_{\textup{bad}, 8}(H)$. The average size of \eqref{eq:Sel_ratio} on such a refinement is then bounded by an expression of the form $e^{Cg^2}$. The result then follows from Propositions \ref{prop:bye_bad} and \ref{prop:XFHsize}.
\end{proof}
\section{Finding the distribution of Selmer ranks}
\label{sec:distribution}

At this point, there are only a few more steps in proving our main result on the distribution of Selmer ranks, Theorem \ref{thm:main}. After considering the parity of Selmer ranks in the alternating case, we will calculate some moments of the sizes of $\omega$-Selmer groups in the twist family $\mathbb{X}_F(H, (\chi_v)_v)$ using the methods developed so far in this paper, and will then derive the distributional results for $\omega$-Selmer ranks from a complex-analytic argument. To finish the proof of Theorem \ref{thm:main}, we will find the distribution of higher Selmer ranks in these twist families by applying \cite[Theorem 4.16]{Smi22a}, the main theorem of that paper.

\subsection{Parity of Selmer ranks in the alternating case}
\label{ssec:parity}

Choose a twistable module $N$ in the alternating case. As part of this case, we have assumed that $\FFF$ has order $2$. Taking $\varepsilon$ to be the map \eqref{eq:Morgan_map}, we may use  \cite[Theorem 6.6]{Morg17} to show that there is $k_0 \in \FFF_2$ so that
\begin{equation}
\label{eq:disparity}
\dim_{\FFF_{\ell}} \Sel^{\omega}(N^{\chi}) \equiv k_0 + \sum_{\mfp | \mfh_F(\chi)} \varepsilon(\Frob\, \mfp)\,\,\text{ mod } 2
\end{equation}
for all $\chi$ in $\mathbb{X}_F((\chi_v)_{v \in \Vplac_0})$. From Hilbert reciprocity, we see that this expression is constant across this set of twists if $\varepsilon$ is a homomorphism.

Conversely, we have the following:
\begin{prop}
\label{prop:parity_noninvariant}
Choose an unpacked starting tuple $(K/F, \Vplac_0, \FFF)$ with $\FFF = \pm 1$, choose a set of local twists $(\chi_v)_{v \in \Vplac_0}$, and choose a twistable module $N$ with local conditions defined with respect to $(K/F, \Vplac_0, \FFF)$ that is in the non-parity-invariant alternating case with respect to $(\chi_v)_{v \in \Vplac_0}$. Then, for any $\epsilon > 0$, there is $C > 0$ so that, for $H > C$, we have
\[\left|\frac{\#\{\chi \in \XXfav_{F, N}(H, (\chi_v)_{v \in \Vplac_0}) \,:\,\, \dim \Sel^{2} N^{\chi} \,\text{ is even}\}}{\# \XXfav_{F, N}(H, (\chi_v)_{v \in \Vplac_0}) } - \frac{1}{2}\right| \le (\log^{(2)} H)^{-\frac{1}{6} +\frac{1}{100} + \epsilon}.\]
\end{prop}
\begin{proof}
Following the method of Proposition \ref{prop:pfav}, we see that the result follows from an application of Proposition \ref{prop:G1_model} with $R = 2$. 
\end{proof}

\subsection{Moments of sizes of $\omega$-Selmer groups}
With Proposition \ref{prop:bye_bad} proved, we can estimate the moments of the sizes of $\omega$-Selmer groups in the family $\mathbb{X}_F(H, (\chi_v)_{v \in \Vplac_0})$ by estimating these moments for twists coming from grids of good ideals. We need some notation.
\begin{notat}
\label{notat:grell}
Given integers $n \ge j \ge 0$ and a prime $\ell$, take $\textup{gr}_{\ell}(j, n)$ to equal the number of $j$-dimensional subspaces of $\FFF_{\ell}^n$. We can calculate
\begin{equation}
\label{eq:grell_explicit}
\textup{gr}_{\ell}(j, n) = \frac{(\ell^n - 1) (\ell^{n-1} - 1) \dots (\ell^{n- j + 1} - 1) }{(\ell^j - 1) (\ell^{j-1} - 1) \dots (\ell - 1)}.
\end{equation}
\end{notat}

\begin{prop}
\label{prop:good_moments}
Take $N$ to be a twistable module with local conditions defined over $(K/F, \Vplac_0, \FFF)$. Choosing a set of local twists $(\chi_v)_{v \in \Vplac_0}$, we assume that $N$ is in either the alternating case or non-self-dual case with respect to $(\chi_v)_{v \in \Vplac_0}$. Also choose $\epsilon > 0$. Then there is $C > 0$ so that, for $H > C$, we have the following:

Take $Y$ to be a filtered grid of good ideals of height $H$, and take $X$ to be the set of twists $\chi$ in $\mathbb{X}_F((\chi_v)_v)$ for which $\mfh_F(\chi)$ lies in $Y$. We assume that $X$ is nonempty and that any/every twist in $X$ is favored for $N$. Take $m$ to be a nonnegative integer that satisfies
\[m \le (\log^{(2)} H)^{1/8 - \epsilon}.\]
In the non-self-dual case, take $u$ to equal the right hand side of \eqref{eq:Wiles}.

Then, defining
\[b(j) = \begin{cases} \tfrac{j(j+1)}{2} &\text{ if } N \text{ is in the alternating case} \\ j \cdot u &\text{ otherwise,}\end{cases}\]
we find that
\[\frac{1}{|X|}\sum_{\chi \in X} \left(\# \Sel^{\omega} N^{\chi}\right)^m \,-\,   (\# H^0(G_F, N[\omega]))^m \cdot \sum_{j = 0}^m \textup{gr}_{\ell}(j, m) \cdot \ell^{b(j)}\]
has magnitude bounded by $\exp(-(\log \log H)^{1/4 - \epsilon})$.
\end{prop}
\begin{proof}
Take $(W_v)_v$ to be the local conditions for $N$, and endow  $N^{\oplus m}$ with the local conditions $(W_v^{\oplus m})_v$. We have
\[  \left(\# \Sel^{\omega} N^{\chi}\right)^m = \# \Sel^{\omega} (N^{\oplus m})^{\chi} = \sum_T \#\{\phi \in \Sel^{\omega}(N^{\oplus m})^{\chi}\,:\,\, \bfQ(\phi) =T\},\]
where the sum is over all $G_F$-submodules $T$ of $N$. From Proposition \ref{prop:cofav_form}, the cofavored submodules of $N^{\oplus m}$ are those of the form
\[A \otimes_{\FFF_{\ell}} N[\omega],\]
where $A$ is a subspace of $\FFF_{\ell}^m$. If $T$ is not of this form, Remark \ref{rmk:not_quite_cofavored} and part 8 of the definition of a good ideal gives
\[\mathcal{T}_{N^{\oplus m} , T}(\chi) \le   \exp(-(\log \log H)^{1/4 - \epsilon})\] 
for $\chi$ in $X$ if $H$ is sufficiently large. From this, Propositions \ref{prop:borderline_rare} and  \ref{prop:rough_grid_average}, and the bounds on $c_{\text{bilin}}$ and $c_{\text{Cheb}}$ of Proposition \ref{prop:real_grid_small_c}, we conclude that
\[\frac{1}{|X|}\sum_{\chi \in X}\sum_{T \textup{ uncofav.}} \#\{\phi \in \Sel^{\omega}(N^{\oplus m})^{\chi}\,:\,\, \bfQ(\phi) =T\} \le \exp(-(\log \log H)^{1/4 - \epsilon})\]
for $H$ sufficiently large given $N$, $(K/F, \Vplac_0, \FFF)$, and $\epsilon$.

This just leaves the sum over cofavored submodules. Here, an application of Proposition \ref{prop:accurate_grid_moment} finishes the proof.
\end{proof}

\subsection{From moments to ranks}
The moments calculated in Proposition \ref{prop:good_moments} are consistent with the probabilities defined in Cases \ref{case:nodual} and \ref{case:alt}.
\begin{prop}
\label{prop:moments_ranks}
Take $N$ and $b$ as in Proposition \ref{prop:good_moments}, and define $P(j\,|\,\infty)$ as in Case \ref{case:nodual} in the non-self-dual case and as in Case \ref{case:alt} in the alternating case.  Given any integer $m \ge 0$, we have
\begin{equation}
\label{eq:moments_ranks}
\sum_{j = 0}^{\infty}  \ell^{mj} \cdot P(j\,|\,\infty) = \sum_{j = 0}^m \textup{gr}_{\ell}(j, m) \cdot \ell^{b(j)}.
\end{equation}
\end{prop}
\begin{proof}
We focus on the alternating case, where we have assumed that $\ell = 2$. Given a positive integer $n$, take $T$ to be a random alternating $n \times n$ matrix over $\FFF_2$. The definition of a mean gives
\[\sum_{j = 0}^{n}  2^{mj} \cdot P(j\,|\,n) = \mathbb{E}\left( \#\left\{ (v_1, \dots, v_m )\in (\FFF_2^n)^m \,:\,\, v_i \text{ is in the kernel of } T \text{ for all } i \le m \right\}\right).\]
Given $(v_1, \dots, v_m)$ spanning a $d$-dimensional subset of $\FFF_{2}^n$, the probability that $Tv_i = 0$ for all $i \le m$ is given by
\[2^{-n + 1} \cdot 2^{-n + 2} \cdot \dots 2^{-n + d} = 2^{-nd + \frac{d(d+1)}{2}}.\]
Meanwhile, the number of $(v_1, \dots, v_m)$ spanning a $d$-dimensional space equals the number of choices of an $m-d$-dimensional space of linear dependencies among these vectors times the number of choices of $d$ linearly independent vectors in $\FFF_2^n$. So  there are
\[\textup{gr}(d, m) \cdot (\ell^{n} - 1) \dots (\ell^n - \ell^{d - 1})\]
such tuples. Combining, we find
\[\lim_{n \to \infty} \sum_{j = 0}^{n}  2^{mj} \cdot P(j\,|\,n) = \sum_{j = 0}^m \textup{gr}_{\ell}(d, m) \cdot 2^{b(j)}.\]
The explicit form of $P(j\,|\,n)$ given in Case \ref{case:alt} allows us to conclude
\[\lim_{n \to \infty} \sum_{j = 0}^{n}  2^{mj} \cdot P(j\,|\,n)  = \sum_{j = 0}^{\infty}  2^{mj} \cdot P(j\,|\,\infty) ,\]
and the proposition follows in the alternating case.  A similar calculation gives the proposition in the non-self-dual case.
\end{proof}

To move from moments to rank distributions, we will follow the trend of the literature and use a brief argument in the theory of holomorphic functions. The specific argument we  will use is similar to that of \cite[Lemma 18]{Heat94}.

\begin{prop}
\label{prop:unicity_ranks}
There is $C > 0$ so we have the following:

Take $a_0, a_1, \dots$ to be a sequence of real numbers, and take
\[f(z) = \sum_{i \ge 0} a_iz^i.\]
Choose an integer $m \ge 1$ and a rational prime $\ell$. We take
\[B = \sum_{i \ge 0} |a_i| \ell^{mi}.\]
We assume that $B$ is finite. Choose $\epsilon > 0$.
\begin{itemize}
\item Suppose $|f(z)| \le \epsilon$ for $z$ in $\{1, \ell, \dots, \ell^{m- 1}\}$. Then
\[|a_i| \le  \ell^{(-i + C)m}\left( B \ell^{-\frac{1}{2}m^2} + \epsilon\right) \quad\text{for } i \ge 0.\]
\item Suppose $|f(z)| \le \epsilon$ for $z$ in $\{\pm 1, \pm \ell, \dots, \pm \ell^{m- 1}\}$. Then
\[|a_i| \le \ell^{(-i + C)m}\left( B \ell^{-m^2} + \epsilon\right) \quad\text{for } i \ge 0.\]
\end{itemize}
\end{prop}
\begin{proof}
There is a unique polynomial $g(z)$ of degree at most $m - 1$ so that $f - g$ is zero on $\{1, \ell, \dots, \ell^{m- 1}\}$. If we take 
\[p(z) = \prod_{\substack{i =0}}^{m-1} \left(1 - \frac{z}{\ell^i}\right) \quad\text{and}\quad p_j(z) = \left(1 - \frac{z}{\ell^j}\right)^{-1}p(z)\quad\text{for } j \le m -1,\]
we may explicitly write this polynomial in the form
\[g(z) = \sum_{j=0}^{m-1}\frac{f(\ell^j)}{p_j(\ell^j)} p_j(z).\]
Since $(1 - \ell^{j-i})$ is an integer for $j > i$, we get the bound
\[\left|p_j(\ell^j)\right| \ge \prod_{i = j+1}^m (1 - \ell^{j-i}) \ge \prod_{i = 1}^{\infty} (1 - 2^{-i}) > 1/4.\]
Meanwhile, if we have $|z| = \ell^m$ and $i < m$, we find 
\[\ell^{m - i - 1} \le  |1 - \ell^{-i} z|  \le \ell^{m-i + 1}.\]
These imply $|p_j(z)| \le \ell^{\frac{(m+1)(m+2)}{2}}$ and $|p(z)|  \ge \ell^{\frac{m(m-1)}{2}}$, so
\[|g(z)|  \le 4m\epsilon \ell^{\frac{(m+1)(m+2)}{2}}\]
for $z$ of magnitude $\ell^m$.

Take $q(z) = (f(z) - g(z))/p(z)$, and write $\sum_{i \ge 0} b_i z^i$ for the power series expansion for $q$ around $0$. This function is continuous in the region $\{z\,:\,\, |z| \le \ell^m\}$. The maximum magnitude of $q$ on the boundary of this disk, and hence on the disk, is at most $4m \epsilon \ell^{2m  + 1} + B \ell^{-\frac{m(m-1)}{2}}$. Applying Cauchy's integral formula to the circles with center $0$ and radii $(\ell^m - \delta)$ and letting $\delta$ approach $0$ then gives
\[b_i \le  \ell^{-im}\left(4m \epsilon \ell^{2m + 1} + B \ell^{-\frac{m(m-1)}{2}}\right).\]
The bounds on the coefficients $a_i$ then follow after we note that $p$ has coefficients bounded in magnitude by an absolute real constant. This gives the first part. The second part follows similarly.
\end{proof}

\subsection{Applying the main result of \cite{Smi22a}}

We are finally in a position to apply \cite[Theorem 4.16]{Smi22a} to finish the proof of Theorem \ref{thm:main}.
\begin{proof}[Proof of Theorem \ref{thm:main}]
From Theorem \ref{thm:main_mmnt_coarse}, we see that
\[\frac{\# \left\{\chi \in \XXfav_{F,N}(H, (\chi_v)_{v \in \Vplac_0})\,:\,\, r_{\omega}(N^{\chi}) > (\log \log \log H)^{1/4} \right\}}{\#\XXfav_{F,N}(H, (\chi_v)_{v \in \Vplac_0})} \le \exp\left(-c\cdot (\log \log \log H)^{1/2}\right) \]
for $H > C$, where $c, C > 0$ are determined from  $(K/F, \Vplac_0, \FFF)$ and $N$. We then see that it suffices to prove that
\[\sum_{r_{\omega} \ge r_{\omega^2} \ge \dots } \Big|P_{\le H}(r_{\omega}, r_{\omega}^2, \dots) \, -\, P_{\textup{th}}(r_{\omega}, r_{\omega^2}, \dots)\Big|\]
is at most $\exp\left( -2c \cdot (\log \log \log H)^{1/2}\right)$ for each $r_{\omega} \le  (\log \log \log H)^{1/4}$. So we fix some $r_{\omega}$ in this range.

Take $X$ to be a $\overline{F}$-refinement of a filtered grid of good ideals of height $H$. Fixing any $\chi_0$ and $(\kappa_s)_s$ as in the proof of Proposition \ref{prop:pfav}, we may associate each $x \in X$ with a twist $\chi(x)$. We assume these twists lie in $\XXfav_{F,N}((\chi_v)_{v \in \Vplac_0})$. 

In the alternating case, the parity of $\dim \Sel^{\omega} N^{\chi(x)}$ does not depend on the choice of $x$ in $X$, and we assume that $r_{\omega}$ has this same parity.

Take
\[P = \frac{1}{\# X}\cdot \# \left\{x \in X\,:\,\, r_{\omega^k}(N^{\chi(x)}) = r_{\omega^k} \,\text{ for } k \ge 1\right\}.\]
From Propositions \ref{prop:XFHsize}, \ref{prop:hateful_eight}, \ref{prop:pfav}, and \ref{prop:parity_noninvariant}, we see that it suffices to prove
\[\left| P - \alpha P_{\textup{th}}(r_{\omega}, r_{\omega}^2, \dots)  \right| \le \tfrac{1}{2}\exp\left( -2c \cdot (\log \log \log H)^{1/2}\right)\]
for $H$ sufficiently large given $N$, $(K/F, \Vplac_0, \FFF)$, where $\alpha$ equals $2$ in the non-parity-invariant alternating case and otherwise equals $1$.

Take
\[f(z) = \frac{1}{|X|} \sum z^{r_{\omega}(N^{\chi(x)})} \,-\, \sum_{j \ge 0 } z^j \cdot \lim_{n \to \infty} P(j \,|\, 2n  + r_{\omega}).\]
Here, the appearance of $P(j \,|\, 2n  + r_{\omega})$ avoids issues with $P(j\,|\, \infty)$ in the alternating non-parity-invariant case. From Proposition \ref{prop:good_moments}, we find that
\[|f(\ell^k)| \le \exp(-(\log \log H)^{1/4 - \epsilon}) \quad\text{for }\, 0 \le k \le (\log \log H)^{1/8 - \epsilon}\]
for sufficiently large $H$ relative to $\epsilon > 0$. In the alternating case, we also have $|f(-\ell^k)| =|f(\ell^k)|$.

From \eqref{eq:grell_explicit}, there is an absolute constant $C_0 > 0$ so that 
\[\sum_{j = 0}^m \textup{gr}_{\ell}(j, m) \cdot \ell^{b(j)} < C_0 \ell^{\frac{m(m+1)}{2}}\]
in the alternating case. Noting that \eqref{eq:grell_explicit} attains a maximum for a fixed $n$ at $j = \lfloor n/2\rfloor$, we find that this constant may be chosen so
\[\sum_{j = 0}^m \textup{gr}_{\ell}(j, m) \cdot \ell^{b(j)} \le \ell^{mC_0 + m|u| + \frac{m^2}{4}}\]
in the non-self-dual case. We may now apply Proposition \ref{prop:unicity_ranks} to conclude that
\[\left|\frac{\#\left\{x \in X\,:\,\, r_{\omega}(N^{\chi(x)}) = r_{\omega}\right\} }{|X|} \,-\, \alpha P(r_{\omega}\,|\,\infty)\right| \le  \exp(-(\log \log H)^{1/4 - \epsilon}).\]
for $H$ sufficiently large given $\epsilon > 0$.

To finish the proof, we choose $x_0$ uniformly at random from $\left\{x \in X\,:\,\, r_{\omega}(N^{\chi(x)}) = r_{\omega}\right\}$ and consider the probability that we may not apply \cite[Theorem 4.16]{Smi22a} to its grid class $\class{x_0}$. We do this by considering the reasons that $\class{x_0}$ might not be suitable for higher work.

First, the grid $X$ itself may not be suitable because $X_s$ is too small for some $s \in \Smed \cup \Slg$ or because $|S|$ is too large. The latter does not happen by the definition of a good ideal, and the former does not happen by the Chebotarev density theorem for large enough $H$. This step uses the fact that the intersection of any ideal in $X$ with $F$ has rational norm at least $\sqrt{H}$.

Next, the ratio $\# \class{x_0}/ \# X$ might be smaller than $\exp^{(3)}\left(\tfrac{1}{4} \log^{(3)} H^{1/2}\right)^{-1}$. Since the total number of possible grid classes is bounded by an expression of the form $e^{C|S|^2}$, we find that this happens for a negligible set of $x_0$.

Taking $\mathscr{M}$ as in Section \ref{ssec:explicit_local}, choose $m_1, \dots, m_{r_{\omega}}$ in $\mathscr{M}$ parameterizing a basis for 
\[\frac{\Sel^{\omega} N^{\chi(x)}}{\text{im}(H^0(G_F, N^{\chi(x)}[\omega]))}\]
for $x$ in $\class{x_0}$. We may then view $m = (m_1, \dots, m_{r_{\omega}})$ as parameterizing an element in the $\omega$-Selmer group of $(N^{\oplus r_{\omega}})^{\chi(x)}$ for $x$ in $\class{x_0}$. By using Remark \ref{rmk:not_quite_cofavored} and  Propositions \ref{prop:borderline_rare}, \ref{prop:rough_grid_average}, and \ref{prop:real_grid_small_c} as in the proof of Proposition \ref{prop:good_moments}, we find that $\bfQ(m)$ is cofavored in $N[\omega]^{\oplus r_{\omega}}$ for $x_0$ outside a negligible set. Since $N$ lacks universal $\omega$-Selmer elements, we may apply Proposition \ref{prop:cofav_form} to show that
\[\bfQ(m) = N[\omega]^{\oplus r_{\omega}}.\]
Now, for any $v \in N[\omega]^{\oplus r_{\omega}}$ and a subset $S_0$ of $\Slg$, we call $m \in \mathscr{M}^{\oplus r_{\omega}}$ \emph{$v$-deficient on} $S_0$ if $\pi_s(m) = v$ for fewer that $\log^{(3)} H$ of the $s$ in $S_0$. From Hoeffding's inequality we see that, if $|S_0| \ge 2 \cdot \# N[\omega]^{r_{\omega}} \cdot \log^{(3)} H$, we have
\[\frac{\#\{m \in \mathscr{M}^{\oplus m} \,:\,\, m \,\text{ is $v$-deficient on $S_0$}\}}{\mathscr{M}^{\oplus m} } \le \exp\left(-c_1^{r_{\omega}}\cdot |S_0|\right).\]
where $c_1 > 0$ is determined by $N$. By applying Proposition \ref{prop:borderline_rare} to restrict to non-borderline $m$ and then applying Proposition  \ref{prop:borderline_fewpair} together with the bounds on $c_{\textup{bilin}}$ and $c_{\textup{Cheb}}$ for a grid of good ideals, we find
\[\frac{\#\{(x, m) \in X \times \mathscr{M}^{\oplus r_{\omega}}\,:\,\, \Psi_x(m) \in \Sel^{\omega}N^{\chi(x)} \,\text{ and }\, m \,\text{ is $v$-deficient on $S_0$}\}}{\# X} \]
is at most $C_2^{r_{\omega}^2} \cdot \exp\left(-c_2^{r_{\omega}}\cdot |S_0|\right)$, where
$C_2, c_2> 0$ are determined from $N$. Using this estimate and the seventh part of definition of good ideal in $F$, we find that the probability that $S_{\text{pot-pre}}$ and $S_{\text{pot-a/b}}$ cannot be selected for the given class $\class{x_0}$ is at most $\exp\left(-(\log \log H)^{1/2}\right)$ for sufficiently large $H$. In the alternating case, we may now apply \cite[Theorem 4.16]{Smi22a} to finish the proof of the theorem. 

In the non-self-dual case, given $x_0$, we also need to choose $m'_1, \dots, m'_{r_{\omega}'}$ in the parameter space for $N[\omega]^{\vee}$ parameterizing a basis for
\[\frac{\Sel^{\omega} (N^{\vee})^{\chi(x)}}{\text{im}(H^0(G_F, (N^{\vee})^{\chi(x)}[\omega]))}.\]
We then follow the same argument as before adjusted for the tuple 
\[(m_1, \dots, m_{r_{\omega}}, m'_1, \dots, m'_{r_{\omega}'})\]
 parameterizing an $\omega$-Selmer element of
\[\left(N^{\oplus r_{\omega}} \oplus (N^{\vee})^{\oplus r_{\omega}'}\right)^{\chi(x)}\]
for $x$ in $\class{x_0}$. This gives the theorem.
\end{proof}

\printbibliography

\end{document}